\documentclass[12pt]{amsart}
\usepackage{latexsym}
\usepackage{amssymb}
\usepackage{amscd}
\usepackage{mathrsfs}
\usepackage[all]{xy}

\renewcommand\a{\alpha}
\renewcommand\b{\beta}
\newcommand\g{\gamma}
\renewcommand\d{\delta}
\newcommand\la{\lambda}

\renewcommand\th{\theta}

\newcommand\s{\sigma}

\renewcommand\r{\rho}

\newcommand\vS{\varSigma}

\newcommand\vD{\varDelta}

\newcommand\ve{\varepsilon}

\newcommand\Ql{\bar{\mathbf Q}_l}
\newcommand\BA{\mathbf A}
\newcommand\BP{\mathbf P}
\newcommand\BQ{\mathbf Q}
\newcommand\BF{\mathbf F}

\newcommand\BZ{\mathbf Z}

\newcommand\BN{\mathbf N}
\newcommand\BL{\mathbf L}

\newcommand\bB{\mathbf B}

\newcommand\BU{\mathbf U}
\newcommand\BV{\mathbf V}

\newcommand\Bn{\mathbf n}

\newcommand\Bc{\mathbf c}
\newcommand\Be{\mathbf e}

\newcommand\Bk{\mathbf k}
\newcommand\Bh{\mathbf h}

\newcommand\Bd{\mathbf d}

\newcommand\Bla{\boldsymbol\lambda}

\newcommand\CP{\mathcal{P}}

\newcommand\SC{\mathscr{C}}

\newcommand\SL{\mathscr{L}}

\newcommand\SO{\mathscr{O}}

\newcommand\SX{\mathscr{X}}

\newcommand\Fg{\mathfrak g}

\newcommand\Fs{\mathfrak s}
\newcommand\Fl{\mathfrak l}

\newcommand\iv{^{-1}}
\newcommand\wh{\widehat}
\newcommand\wt{\widetilde}

\newcommand\ck{^{\vee}}

\newcommand\ol{\overline}
\newcommand\ul{\underline}

\newcommand\lra{\leftrightarrow}

\newcommand\lv{\prec}
\newcommand\gv{\succ}
\newcommand\lve{\preceq}

\newcommand\IC{\operatorname{IC}}

\newcommand\ch{\operatorname{ch}}

\renewcommand\Im{\operatorname{Im}}
\newcommand\Dim{\operatorname{Dim}}
\newcommand\nat{^{\natural}}

\newcommand\cl{\operatorname{cl}}

\newcommand\re{\operatorname{re}}
\newcommand\im{\operatorname{im}}

\newcommand{\isom}{\,\raise2pt\hbox{$\underrightarrow{\sim}$}\,}
\numberwithin{equation}{section}

\newtheorem{thm}{Theorem}[section]
\newtheorem{lem}[thm]{Lemma}
\newtheorem{cor}[thm]{Corollary}
\newtheorem{prop}[thm]{Proposition}

\def \para#1{\par\medskip\textbf{#1}
              \addtocounter{thm}{1}}

\def \remark#1{\par\medskip\noindent
                \textbf{Remark #1}
                \addtocounter{thm}{1}}

\def \remarks#1{\par\medskip\noindent
                \textbf{Remarks #1}
                \addtocounter{thm}{1}}

\begin{document}
\setlength{\baselineskip}{4.9mm}
\setlength{\abovedisplayskip}{4.5mm}
\setlength{\belowdisplayskip}{4.5mm}
\renewcommand{\theenumi}{\roman{enumi}}
\renewcommand{\labelenumi}{(\theenumi)}
\renewcommand{\thefootnote}{\fnsymbol{footnote}}
\renewcommand{\thefootnote}{\fnsymbol{footnote}}
\allowdisplaybreaks[2]
\parindent=20pt
\medskip
\begin{center}
{\bf Diagram automorphisms and canonical bases \\ 
for quantum affine algebras, II}  
\par
\vspace{1cm}
Ying Ma, Toshiaki Shoji and Zhiping Zhou
\\
\title{}
\end{center}

\begin{abstract}
Let $\BU_q^-$ be the negative part of 
the quantum enveloping algebra, and 
$\s$ the algebra automorphism on $\BU_q^-$
induced from a diagram automorphism.  
Let $\ul\BU_q^-$ be the quantum algebra obtained 
from $\s$, and $\wt\bB$ (resp. $\wt{\ul\bB}$) 
the canonical signed basis of $\BU_q^-$ 
(resp. $\ul\BU_q^-$).   Assume that $\BU_q^-$ 
is simply-laced of finite or affine type.  In our previous
papers [SZ1, 2], we have proved by an elementary method, 
that there exists a natural bijection 
$\wt\bB^{\s} \simeq \wt{\ul\bB}$ in the case where 
$\s$ is admissible.  
In this paper, we show that such a bijection exists
even if $\s$ is not admissible, possibly except 
some small rank cases.  
\end{abstract}

\maketitle
\pagestyle{myheadings}


\begin{center}
{\sc Introduction}
\end{center}
\par\medskip
This paper is a continuation of [SZ1, 2]. 
Let $X$ be a simply-laced Cartan datum with the vertex set $I$, and $\BU_q^-$
the negative part of the quantum enveloping algebra associated to $X$.
We assume that $X$ is irreducible, of finite or affine type, 
and let $\s: I \to I$ be the diagram automorphism of $X$.  We further assume that 
$\s$ is admissible (see 1.3 for the definition). 
Then the order $\ve$ of $\s$ is 2, 3 or 4. $\s$ induces an algebra automorphism 
of $\BU_q^-$. Let $\ul I$ be the set of $\s$-orbits in $I$.  
One can 
construct the Cartan datum $\ul X$ with the vertex set $\ul I$.  We denote by 
$\ul\BU_q^-$ the negative part of the quantum enveloping algebra associated to 
$\ul X$. Let $\bB$ (resp. $\ul\bB$) be the canonical basis of $\BU_q^-$ 
(resp. $\ul \BU_q^-$). Then $\s$ acts on $\bB$ as a permutation.  We denote by 
$\bB^{\s}$ the set of $\s$-fixed elements in $\bB$.  It is known by Lusztig [L1]
that there exists a natural bijection $\bB^{\s} \simeq \ul\bB$. 
\par
Let $\BA = \BZ[q, q\iv]$, and ${}_{\BA}\BU_q^-$ Lusztig's integral form of 
$\BU_q^-$, which is an $\BA$-subalgebra of $\BU_q^-$. Assume that 
$\ve = 2, 3$, and let $\BF = \BZ/\ve\BZ$ be the finite field of $\ve$ elements.
Set $\BA' = \BF[q, q\iv]$, and define ${}_{\BA'}\BU_q^{-,\s}$ as the $\BA'$-subalgebra
of $\BA'\otimes_{\BA}{}_{\BA}\BU_q^-$ 
consisting of $\s$-fixed elements. 
The $\BA'$-algebra ${}_{\BA'}\ul\BU_q^-$ is defined as 
$\BA'\otimes_{\BA}{}_{\BA}\ul\BU_q^-$.
In [SZ1, 2], 
it is shown that, in the case where $\s$ is admissible with $\ve \ne 4$, 
there exists an $\BA'$-algebra isomorphism $\Phi : {}_{\BA'}\ul\BU_q^- \isom \BV_q$, 
where $\BV_q$ is a certain quotient algebra of ${}_{\BA'}\BU_q^{-,\s}$. 
By using this isomorphism, we have constructed, in an elementary way, 
the canonical signed basis $\wt{\ul\bB} = \ul\bB \sqcup -\ul\bB$ of $\ul\BU_q^-$, 
and a bijection $\wt\bB^{\s} \simeq \wt{\ul\bB}$, where $\wt\bB$ is the canonical 
signed basis of $\BU_q^-$. 
This result was generalized in [MSZ] to the case where $X$ is of general Kac-Moody type,
$\s$ is admissible, and $\ve$ is a power of a prime number. 
\par
In this paper, we consider the case where $\s$ is not admissible.  
Already in [L2], Lusztig proved that in the case where $X$ is of finite type, 
the bijection $\bB^{\s} \simeq \ul\bB$ still holds. So, one can expect that 
our isomorphism ${}_{\BA'}\ul\BU_q^- \isom \BV_q$ will have a generalization 
for non-admissible cases. 
However, in the case where $X = A_n^{(1)}$, and $\s$ is
the one corresponding to the cyclic quiver, there does not exist $\s$-stable
canonical basis (see Remark 1.9).  So we must exclude this case. 
Then we show in Theorem 2.6, that similar results as in [SZ1, 2] hold 
for $X$ of finite or affine type, possibly except  
$\ul X$ is of type $A_2^{(2)}$ or $A_1^{(1)}$. 
The latter cases are excluded simply because 
the computation becomes more complicated. 
It is likely that our results will hold in those cases. 
\par
Note that our discussion in [SZ1, 2] heavily depends on the 
property of PBW-bases. 
But in the non-admissible case, the action of $\s$ on PBW-bases becomes quite 
different from the admissible case. The simplest example is that $X$ is of type $A_2$ 
and $\ul X$ is of type $C_1$.  In this case, there does not exist $\s$-stable PBW-bases 
except the trivial one. Thus the usual PBW-bases are not so convenient for 
the non-admissible case.  The main ingredient for our discussion is 
the modified PBW-bases, which is defined by mixing the usual PBW-bases and 
the canonical bases for $A_2$. By using the modified PBW-bases, one can prove that $\Phi$ is 
a homomorphism. The injectivity of $\Phi$ is proved by a general argument as in 
[MSZ].  In the case where $X$ is of finite type, $\s$ acts on the set of modified
PBW-bases, and the surjectivity is proved in Theorem 3.7, 
by a similar argument as in [SZ1]. 
In the affine case, the modified PBW-bases are constructed, by making use of 
the PBW-bases associated to the convex order due to Muthiah-Tingley [MT] as in [SZ2]. 
Since it is not certain whether $\s$ leaves 
the set of modified PBW-bases invariant, we prove the surjectivity 
in Proposition 6.21, by using some properties of Kashiwara operators. 
Note, in contrast to the discussion in [SZ2], that we need to use modified PBW-bases 
in order to define Kashiwara operators for the $\s$-setup. 
Finally, we give an explicit description of
the bijection $\wt\bB^{\s} \isom \wt{\ul\bB}$ in terms of modified PBW-bases.          

\par\bigskip
\section{Diagram automorphisms}

\para{1.1.}
We basically follow the notation in [MSZ, 1].
Let $X = (I, (\ ,\ ))$ be a symmetric Cartan datum, 
where $I$ is a finite set, and $(\ ,\ )$ is a symmetric bilinear form on 
the vector space $\bigoplus_{i \in I}\BQ\a_i$ with basis 
$\{ \a_i \mid i \in I\}$.  Hence the Cartan matrix 
$A = (a_{ij})$ is symmetric, where 
$a_{ij} = \frac{2(\a_i,\a_j)}{(\a_j,\a_j)}$.
Let $Q = \sum_{i \in I}\BZ\a_i$ be the root lattice of $X$. 
Set $Q_+ = \sum_{i \in I}\BZ_{\ge 0}\a_i$ and $Q_{-} = - Q_+$.   
Let $\BU_q^-$ be the negative part of the quantized enveloping algebra 
$\BU_q$ over $\BQ(q)$ associated to $X$ with generators $f_i \ (i \in I)$. 
It is known by Lusztig [L1] that $\BU_q^-$ has the 
canonical basis $\bB$.  
\par
Let $\s : I \to I$ be a diagram automorphism on $X$, namely 
$\s$ satisfies the property $(\a_{\s(i)}, \a_{\s(j)}) = (\a_i, \a_j)$
for any $i,j \in I$. 
Then $\s$ induces an automorphism 
$\s : \BU_q^- \to \BU_q^-$ as $\BQ(q)$-algebras, by $f_i \mapsto f_{\s(i)}$. 
We denote by $\BU_q^{-,\s}$ 
the fixed point subalgebra of $\BU_q^-$ by $\s$. 
$\s$ induces a permutation on the canonical basis $\bB$.  We denote by 
$\bB^{\s}$ the set of $\s$-fixed elements in $\bB$.  
The notion of canonical signed bases is given as in [MSZ, 1.18].  
Then $\wt\bB = \bB \sqcup -\bB$ coincides with the canonical signed basis. 
We denote by $\wt\bB^{\s}$ the set of $\s$-fixed elements in $\wt\bB$. 

\para{1.2.}
Let $\BA = \BZ[q, q\iv]$.  We denote by ${}_{\BA}\BU_q^-$ Lusztig's 
integral form of $\BU_q^-$, i.e., ${}_{\BA}\BU_q^-$ is the $\BA$-subalgebra 
of $\BU_q^-$ generated by $f_i^{(n)} = f_i^n/[n]^!_{d_i}$ for $i \in I, n \in \BN$, 
where $d_i = (\a_i,\a_i)/2 \in \BZ_{>0}$. 
$\s$ stabilizes ${}_{\BA}\BU_q^-$, and set 
${}_{\BA}\BU_q^{-,\s} = \BU_q^{-,\s} \cap {}_{\BA}\BU_q^-$, the $\BA$-subalgebra 
of ${}_{\BA}\BU_q^-$ consisting of $\s$-fixed elements. 
\par
We assume that the order of $\s$ is a power of a prime number $p$, and 
let $\BF = \BZ/p\BZ$ be the finite field of $p$-elements. 
Set $\BA' = \BF[q, q\iv]$, and consider an $\BA'$-algebra ${}_{\BA'}\BU_q^-$ by 
\begin{equation*}
{}_{\BA'}\BU_q^{-,\s} = \BA' \otimes_{\BA}{}_{\BA}\BU_q^{-,\s} 
       \simeq {}_{\BA}\BU_q^{-,\s}/p({}_{\BA}\BU_q^{-,\s}). 
\end{equation*}
For each $x \in \BU_q^-$, we denote by $O(x)$ the orbit sum of $x$ defined by 
$O(x) = \sum_{0 \le i < k}\s^i(x)$, where $k$ is the smallest integer such that 
$\s^k(x) = x$. 
Hence $O(x) \in \BU_q^{-,\s}$. $O(x)$ is similarly defined for ${}_{\BA'}\BU_q^-$. 
Let $J$ be the $\BA'$-submodule of ${}_{\BA'}\BU_q^{-,\s}$ generated by 
$O(x)$ for $x \in {}_{\BA'}\BU_q^-$ such that $\s(x) \ne x$. 
Then $J$ is the two-sided ideal of ${}_{\BA'}\BU_q^{-,\s}$.  We define an 
$\BA'$-algebra $\BV_q$ by $\BV_q = {}_{\BA'}\BU_q^{-,\s}/J$. Let 
$\pi : {}_{\BA'}\BU_q^{-\s} \to \BV_q$ be the natural projection. 

\para{1.3.}
Let $\s : I \to I$ be the diagram automorphism.  We denote by 
$\ul I$ the set of $\s$-orbits in $I$. 
$\s : I \to I$ is called admissible, if for any $\eta \in \ul I$, 
$(\a_i, \a_j) = 0$ for all $i \ne j \in \eta$. The algebra $\BV_q$ 
was defined in [MSZ, 3.2] under the assumption that $\s$ is admissible.
But, as is seen from 1.2,  the definition of $\BV_q$ makes sense even 
if $\s$ is not admissible.
\par
Assume that $\s$ is admissible. Then by [MSZ, 2.1], one can construct 
a Cartan datum $\ul X = (\ul I, (\ ,\ )_1)$ induced from $(X, \s)$.  
We denote by $\ul\BU_q^-$ the negative part of the quantized enveloping
algebra $\ul\BU_q$ associated to $\ul X$. 
The algebras ${}_{\BA}\ul\BU_q^-$
and ${}_{\BA'}\ul\BU_q^-$ are defined similarly to the case of $\BU_q^-$. 
\par
For each $\eta \in \ul I, a \in \BN$, set $\wt f_{\eta}^{(a)} = \prod_if_i^{(a)}$.
Since $\s$ is admissible, $\wt f_{\eta}^{(a)}$ does not depend on the order of 
the product, and $\wt f_{\eta}^{(a)} \in {}_{\BA}\BU_q^{-,\s}$. We denote its image
in ${}_{\BA'}\BU_q^{-,\s}$ also by $\wt f_{\eta}^{(a)}$.  
We define $g_{\eta}^{(a)} \in \BV_q$ by 
\begin{equation*}
\tag{1.3.1}
g_{\eta}^{(a)} = \pi(\wt f_{\eta}^{(a)}).
\end{equation*} 
Note that ${}_{\BA}\ul\BU_q^-$ is generated by 
$\ul f_{\eta}^{(a)} = \ul f_{\eta}^a/[a]^!_{d_{\eta}}$ for 
$\eta \in \ul I$ and $a \in \BN$.  We denote by the same symbol $\ul f_{\eta}^{(a)}$ 
its image in ${}_{\BA'}\ul\BU_q^-$.
\par
The following result was proved in Theorem 3.4 and Theorem 4.18 
(see also Remark 4.19) 
in [MSZ]

\begin{thm}  
Assume that $\s$ is admissible, and the order of $\s$ is a power of
a prime number $p$.  Then 
\begin{enumerate}
\item \ The assignment $\ul f_{\eta}^{(a)} \mapsto g_{\eta}^{(a)}$ 
gives an $\BA'$-algebra isomorphism 
$\Phi : {}_{\BA'}\ul\BU_q^- \isom \BV_q$. 
\item \ There exists the canonical signed basis $\wt{\ul\bB}$ of $\ul\BU_q^-$, 
and a natural bijection $\xi : \wt\bB^{\s} \isom \wt{\ul\bB}$.
\end{enumerate}
\end{thm} 

\para{1.5.}
We consider the case where $X$ is finite or affine type.
Let $\ve$ be the order of $\s$.
Assume that $X$ is irreducible, and $\s$ is admissible.  
In the case where $\ve = 2, 3$,  
Theorem 1.4 was proved by [SZ1, 2], by an elementary method.  
\par
In the remaining cases, $X$ are affine, and 
they are given as follows;

\begin{equation*}
\begin{aligned}
&(A_a1) &\quad X &= D_{2n}^{(1)}, &\quad \ul X &= A_{2n-2}^{(2)}, \ (n \ge 3), 
       &\quad  \ve &= 4, 
\\
&(A_a1') &\quad X &= D_4^{(1)}, &\quad \ul X &= A_2^{(2)}, 
        &\quad \ve &= 4, 
\\
&(A_a2)  &\quad X &= A_{n-1}^{(1)}, &\quad \ul X &= A_{m - 1}^{(1)}, \ (n = mc,   1 < c < n), 
        &\quad \ve &= c.  
\end{aligned}
\end{equation*}
The diagrams are given as follows;

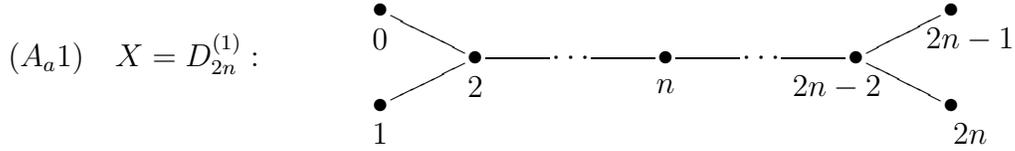
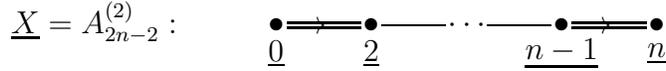
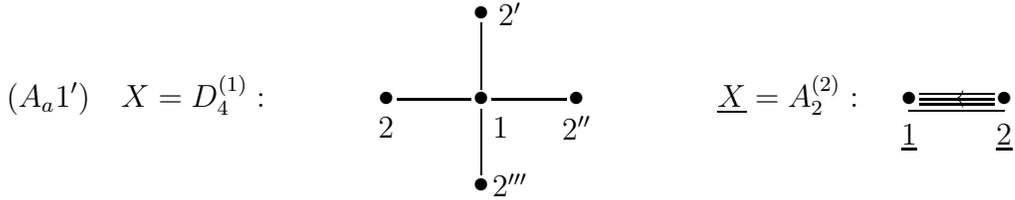
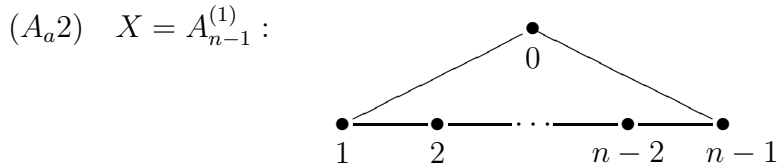
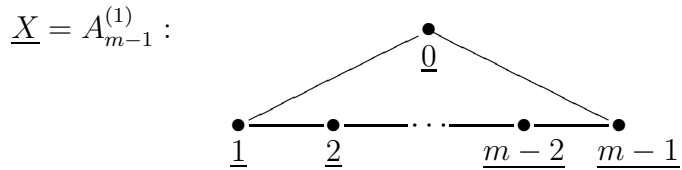
\SelectTips{cm}{12}
\objectmargin{1pt}

\begin{equation*}
\begin{aligned}
(A_a1) \quad X &= D_{2n}^{(1)} : &\qquad
&\xygraph{
\bullet([]!{+(0,-.3)} {2}) (
   - []!{+(-1, .5)}\bullet([]!{+(0,-.3)}{0}),
   - []!{+(-1,-.5)}\bullet([]!{+(0,-.3)}{1}),
   - [r] \cdots - [r]
       \bullet([]!{+(0,-.3)}{n}) - [r] \cdots - [r] 
       \bullet([]!{+(-0.2,-.3)}{2n-2}) (
       - []!{+(1,.5)}\bullet([]!{+(0.2,-.3)}{2n-1}),
       - []!{+(1, -.5)}\bullet([]!{+(0.2, -.3)}{2n})
      ) 
 )} 
\\ \\
\ul X &= A_{2n-2}^{(2)} : &\qquad
&\xygraph{!~:{@ 2{-}|@{>}}
\bullet([]!{+(0,-.3)} {\ul 0}) : [r]
\bullet([]!{+(0,-.3)} {\ul 2}) - [r] \cdots - [r]
\bullet([]!{+(0,-.3)} {\ul{n-1}}) : [r]
\bullet([]!{+(0,-.3)} {\ul {n}})
}
\\
\\
\\
(A_a1') \quad X &= D_4^{(1)} : &\quad 
&\xygraph{
\bullet([]!{+(0,-.3)} {2}) - [r] 
\bullet([]!{+(.2, -.3)} {1}) ( 
  - []!{+(0, .9)} \bullet([]!{+(.3, 0)} {2'}), 
  - []!{+(0,-.9)} \bullet([]!{+(.3, 0)} {2'''}),
  - [r] \bullet([]!{+(0,-.3)} {2''}))
 )}
\qquad\qquad
\ul X = A_2^{(2)} : \quad
\xygraph{!~:{@ 3{-}|@{<}}
\bullet([]!{+(0,-.4)} {\ul 1}) ( : [r] \bullet([]!{+(0,-.4)} {\ul 2}, 
    []!{(-0,-.12)}  - [r]) 
}
\\
\\
(A_a2) \quad X &= A_{n - 1}^{(1)} : &\qquad
&\hspace{-5mm}\xygraph{
   \bullet([]!{+(0,-.3)}{0}) - []!{+(-2, -1)}
       \bullet([]!{+(0,-.3)}{1}) - [r] 
       \bullet([]!{+(0,-.3)}{2}) - [r] \cdots - [r]
       \bullet([]!{+(0,-.3)}{n - 2}) - [r] 
       \bullet([]!{+(0.2, -.3)}{n - 1}) - []!{+(-2, 1)} 
    }
\\
\\
\ul X &= A_{m - 1}^{(1)} : &\qquad
&\hspace{-5mm}\xygraph{
   \bullet([]!{+(0,-.3)}{\ul 0}) - []!{+(-2, -1)}
       \bullet([]!{+(0,-.3)}{\ul 1}) - [r] 
       \bullet([]!{+(0,-.3)}{\ul 2}) - [r] \cdots - [r]
       \bullet([]!{+(0,-.3)}{\ul{m - 2}}) - [r] 
       \bullet([]!{+(0.2, -.3)}{\ul{m - 1}}) - []!{+(-2, 1)} 
    }
\\
\\
\end{aligned}
\end{equation*}
In the case ($A_a1$), 
$I = \{ 0,1, \dots, 2n\}$ and $\s$ is given by 
$0 \to 2n-1 \to 1 \mapsto 2n \mapsto 0$, 
$i \longleftrightarrow 2n - i$ for $i = 2, \dots, n$. 
Thus $\ve = 4$, and 
$\ul I = \{ \ul 0, \ul i \ (2 \le i \le n-1), \ul{n}\}$.   
In the case ($A_a1'$), $I = \{ 1, 2,2',2'', 2'''\}$ and $\s$ is given by  
$\s : 2 \mapsto 2' \mapsto 2'' \mapsto 2''' \mapsto 2$, and 
$\s(1) = 1$. Thus $\ve = 4$ with 
$\ul 1 = \{ 1\}$ and $\ul 2 = \{ 2,2',2'',2'''\}$.
In the case ($A_a2$), we identify $I = \{ 0,1, \dots, n-1\}$ with 
$\BZ/n\BZ$.  Then $\s$ is given by $i \mapsto i + m$ for $i \in \BZ/n\BZ$. 
Thus $\ve = n/m = c$, and $\ul I = \{ \ul 0, \dots \ul{m - 1}\} \simeq \BZ/m\BZ$.  

\para{1.6.}
Assume that $X$ is of finite type.  In [L2], Lusztig proved
the existence of a natural bijection $\bB^{\s} \isom \ul\bB$ 
even in the case where $\s$ is not admissible. 
Thus one can expect that Theorem 1.4 can be generalized to the
case where $\s$ is not admissible.  In the remainder of this section, 
we consider the case where $X$ is finite or affine, and $\s$ is 
not admissible.  Note that $X$ is simply-laced, i.e., 
$(\a_i, \a_j) \in \{ 0, -1\}$ for $i \ne j$, and $(\a_i, \a_i) = 2$ for any $i \in I$. 
For each $\eta \in \ul I$, set $\d_{\eta} = 1$ if $(\a_i,\a_j) = 0$ for any 
$i \ne j$ in $\eta$, and $\d_{\eta} = 2$ otherwise. 
Note that $\s$ is admissible if and only if $\d_{\eta} = 1$ for any $\eta \in \ul I$. 
We define a symmetric bilinear form $(\ ,\ )_1$ on the vector space
$\bigoplus_{\eta \in \ul I}\BQ \a_{\eta}$ over $\BQ$ by 
\begin{equation*}
\tag{1.6.1}
(\a_{\eta}, \a_{\eta'})_1 =  \begin{cases}
                    2\d_{\eta}|\eta|  &\quad\text{ if } \eta = \eta', \\
                    -\d_{\eta}\d_{\eta'}
         \sharp\{ (i,j) \in \eta \times \eta' \mid (\a_i, \a_j) \ne 0 \}
                                      &\quad\text{ if } \eta \ne \eta'.
                             \end{cases}
\end{equation*}
Then $\ul X = (\ul I, (\ , \ ))$ gives a Cartan datum since 
$(\a_{\eta}, \a_{\eta})_1 \in 2\BZ_{> 0}$, and
\begin{equation*}
\frac{2(\a_{\eta}, \a_{\eta'})_1}{(\a_{\eta'}, \a_{\eta'})_1}
    = \frac{1}{\d_{\eta'}|\eta'|}\sum_{j \in \eta'}
                \d_{\eta}\d_{\eta'}\sum_{i \in \eta}(\a_i, \a_j) \in \BZ_{\le 0}.
\end{equation*}  

\para{1.7.}
Assume that $\s$ is non-admissible. If $X$ is irreducible of finite type, 
only the following case occurs; $X $ is of type $A_{2n}$ with 
$I = \{ 1,2, \dots, 2n\}$, and $\s : i \lra 2n - i+1$ with $\ve = 2$. 
Then $\ul X$ is of type 
$C_n$ with $\ul I = \{ \ul 1, \dots, \ul n\}$. 
\par\bigskip

\begin{equation*}
\begin{aligned}
(F_n1) \quad X &= A_{2n} : &\qquad
&\xygraph{
\bullet([]!{+(0,-.3)} {1}) 
       - [r] \cdots - [r]
       \bullet([]!{+(0,-.3)}{n}) - [r] 
       \bullet([]!{+(0,-.3)}{n+1})  - [r] \cdots - [r]
       \bullet([]!{+(0,-.3)}{2n})
    }
\\ \\
\ul X &= C_n : &\qquad
&\xygraph{!~:{@ 2{-}|@{<}}
\bullet([]!{+(0,-.3)} {\ul 1}) - [r]
\bullet([]!{+(0,-.3)} {\ul 2}) - [r] \cdots - [r]
\bullet([]!{+(0,-.3)} {\ul{n-1}}) : [r]
\bullet([]!{+(0,-.3)} {\ul n})
}
\end{aligned}
\end{equation*} 
\par\medskip
with $\s : i \lra 2n+1-i$ for $1 \le i \le 2n$.  

\par\medskip
If $X$ is irreducible of affine type, the following cases occur.
\par\medskip\noindent
\begin{equation*}
\begin{aligned}
&(A_n1)  &\quad X &= D_{2n+1}^{(1)}, &\quad  \ul X &= A_{2n-1}^{(2)},\  (n \ge 3), 
       &\quad    \ve &= 2, 
\\
&(A_n2) &\quad X &= D_{2n+1}^{(1)}, &\quad  \ul X &= C_{n-1}^{(1)}, \ (n \ge 3),  
&\quad \ve &= 4, 
\\
&(A_n3) &\quad X &= A_{2n}^{(1)}, &\quad \ul X &= A_{2n}^{(2)}, \ (n \ge 2),  
&\quad \ve &= 2,
\\
&(A_n4) &\quad  X &= A_{2n+1}^{(1)}, &\quad  \ul X &= C_n^{(1)}, \ (n \ge 2),  
&\quad \ve &= 2, 
\\
&(A_n2') &\quad X &= D_5^{(1)}, &\quad \ul X &= A_1^{(1)}, &\quad \ve &= 4, 
\\
&(A_n3') &\quad  X &= A_2^{(1)}, &\quad \ul X &= A_2^{(2)}, 
&\quad \ve &= 2,
\\
&(A_n4') &\quad  X &= A_3^{(1)}, &\quad \ul X &= A^{(1)}_1, 
&\quad \ve &= 2,
\\
&(A_n5) &\quad X &= A_{n-1}^{(1)}, &\quad \ul X &= A_1, 
&\quad  \ve &= n.  
\end{aligned}
\end{equation*}

The diagrams are given as follows.

\begin{equation*}
\begin{aligned}
  (A_n1) \quad X &= D_{2n+1}^{(1)} : &\qquad
&\xygraph{
\bullet([]!{+(0,-.3)} {2}) (
   - []!{+(-1, .5)}\bullet([]!{+(0,-.3)}{0}),
   - []!{+(-1,-.5)}\bullet([]!{+(0,-.3)}{1}),
   - [r] \cdots - [r]
       \bullet([]!{+(0,-.3)}{n}) - [r] 
       \bullet([]!{+(0,-.3)}{n + 1}) - [r]  \cdots - [r] 
       \bullet([]!{+(-0.2,-.3)}{2n-1}) (
       - []!{+(1,.5)}\bullet([]!{+(0,-.3)}{2n}),
       - []!{+(1, -.5)}\bullet([]!{+(0, -.3)}{2n+1})
      ) 
 )} 
\\ \\
\ul X &= A_{2n-1}^{(2)} : &\qquad
&\xygraph{
    \bullet([]!{+(0,-.3)} {\ul 2}) (
   - []!{+(-1, .5)}\bullet([]!{+(0,-.3)}{\ul 0}),
   - []!{+(-1,-.5)}\bullet([]!{+(0,-.3)}{\ul 1}),
   - [r] \cdots - [r]  
!~:{@ 2{-}|@{<}}
\bullet([]!{+(0,-.3)} {\ul {n-1}}) : [r]
\bullet([]!{+(0,-.3)} {\ul {n}})
}
\\
\\
(A_n2) \quad  X &= D_{2n+1}^{(1)} : &\qquad
&\xygraph{
\bullet([]!{+(0,-.3)} {2}) (
   - []!{+(-1, .5)}\bullet([]!{+(0,-.3)}{0}),
   - []!{+(-1,-.5)}\bullet([]!{+(0,-.3)}{1}),
   - [r] \cdots - [r]
       \bullet([]!{+(0,-.3)}{n}) - [r] 
       \bullet([]!{+(0,-.3)}{n + 1}) - [r]  \cdots - [r] 
       \bullet([]!{+(-0.2,-.3)}{2n-1}) (
       - []!{+(1,.5)}\bullet([]!{+(0,-.3)}{2n}),
       - []!{+(1, -.5)}\bullet([]!{+(0, -.3)}{2n+1})
      ) 
 )} 
\\ \\
\ul X &= C_{n-1}^{(1)} : &\qquad
&\xygraph{!~:{@ 2{-}|@{>}}
\bullet([]!{+(0,-.3)} {\ul 0}) : [r]
\bullet([]!{+(0,-.3)} {\ul 2}) - [r] \cdots - [r]
!~:{@ 2{-}|@{<}}
\bullet([]!{+(0,-.3)} {\ul {n-1}}) : [r]
\bullet([]!{+(0,-.3)} {\ul {n}})
}
\\
\\
(A_n3) \quad  X &= A_{2n}^{(1)} : &\qquad
&\xygraph{
\bullet([]!{+(0,-.3)} {0}) (
   - []!{+(1, .5)}\bullet([]!{+(0,-.3)}{1}) - [r]
    \bullet([]!{+(0,-.3)}{2}) - [r] \cdots - [r] 
    \bullet([]!{+(0,-.3)}{n-1}) - [r] 
    \bullet([]!{+(0.2,-.3)}{n}) - []!{+(0,-1)},
   - []!{+(1,-.5)}\bullet([]!{+(0,-.3)}{2n}) - [r] 
       \bullet([]!{+(0,-.3)}{2n-1}) - [r] \cdots - [r]
       \bullet([]!{+(0,-.3)}{n + 2}) - [r] 
       \bullet([]!{+(0.2, -.3)}{n + 1}) 
       )} 
\\ \\
\ul X &= A_{2n}^{(2)} : &\qquad
&\xygraph{!~:{@ 2{-}|@{<}}
\bullet([]!{+(0,-.3)} {\ul 0}) : [r]
\bullet([]!{+(0,-.3)} {\ul 1}) - [r]
\bullet([]!{+(0,-.3)} {\ul 2}) - [r] \cdots - [r]
\bullet([]!{+(0,-.3)} {\ul {n-1}}) : [r]
\bullet([]!{+(0,-.3)} {\ul {n}})
}
\\  \\
(A_n4) \quad  X &= A_{2n+1}^{(1)} : &\qquad
&\hspace{-5mm}\xygraph{
   \bullet([]!{+(-0.2,-.3)}{0}) - []!{+(0, -1)}
       \bullet([]!{+(0,-.3)}{2n+1}) - [r] 
       \bullet([]!{+(0,-.3)}{2n}) - [r] \cdots - [r]
       \bullet([]!{+(0,-.3)}{n + 2}) - [r] 
       \bullet([]!{+(0.2, -.3)}{n + 1}), 
     - [r]
    \bullet([]!{+(0,-.3)}{1}) - [r] \cdots - [r] 
    \bullet([]!{+(0,-.3)}{n-1}) - [r] 
    \bullet([]!{+(0.2,-.3)}{n}) - []!{+(0,-1)}
    } 
\\
\\
\ul X &= C_{n}^{(1)} : &\qquad
&\xygraph{!~:{@ 2{-}|@{>}}
\bullet([]!{+(0,-.3)} {\ul 0}) : [r]
\bullet([]!{+(0,-.3)} {\ul 1}) - [r] \cdots - [r]
!~:{@ 2{-}|@{<}}
\bullet([]!{+(0,-.3)} {\ul {n-1}}) : [r]
\bullet([]!{+(0,-.3)} {\ul {n}})
}
\end{aligned}
\end{equation*}

\begin{equation*}
\begin{aligned}
(A_n2') \quad  X &= D_5^{(1)} : &\quad
&\xygraph{
\bullet([]!{+(0,-.3)} {2}) (
   - []!{+(-1, .5)}\bullet([]!{+(0,-.3)}{0}),
   - []!{+(-1,-.5)}\bullet([]!{+(0,-.3)}{1}),
   - [r] 
          \bullet([]!{+(-0,-.3)}{3}) (
       - []!{+(1,.5)}\bullet([]!{+(0,-.3)}{4}),
       - []!{+(1, -.5)}\bullet([]!{+(0, -.3)}{5})
      ) 
 )} 
\quad
&\ul X = A^{(1)}_1 : &\quad
&\xygraph{!~:{@ 2{-}}
\bullet([]!{+(0,-.3)} {\ul 0}) : [r]
\bullet([]!{+(0,-.3)} {\ul 1}) 
}
\\  \\
(A_n3') \quad  X &= A_{2}^{(1)} : &\quad
&\xygraph{
\bullet([]!{+(0,-.3)} {0}) (
   - []!{+(1, .5)}\bullet([]!{+(0.2,-.3)}{1})
   - []!{+(0,-1)},
   - []!{+(1,-.5)}\bullet([]!{+(0.2,-.3)}{2}) 
     )} 
\quad 
&\ul X = A_{2}^{(2)} : &\quad
&\xygraph{!~:{@ 3{-}|@{<}}
\bullet([]!{+(0,-.4)} {\ul 0}) ( : [r] \bullet([]!{+(0,-.4)} {\ul 1}, 
    []!{(-0,-.12)}  - [r]) 
}
\\  \\
(A_n4') \quad  X &= A_{3}^{(1)} : &\quad
&\xygraph{
   \bullet([]!{+(-0.2,-.3)}{0}) - []!{+(0, -1)}
       \bullet([]!{+(-0.2,-.3)}{3}) - [r] 
       \bullet([]!{+(0.2, -.3)}{2}), 
     - [r]
       \bullet([]!{+(0.2,-.3)}{1}) - []!{+(0,-1)}
    } 
\quad
&\ul X = A^{(1)}_1 : &\quad
&\xygraph{!~:{@ 2{-}}
\bullet([]!{+(0,-.3)} {\ul 0}) : [r]
\bullet([]!{+(0,-.3)} {\ul 1}) 
}
\\  \\
(A_n5) \quad  X &= A_{n - 1}^{(1)} : &\quad
&\hspace{-5mm}\xygraph{
   \bullet([]!{+(0,-.3)}{0}) - []!{+(-2, -1)}
       \bullet([]!{+(0,-.3)}{1}) - [r] 
       \bullet([]!{+(0,-.3)}{2}) - [r] \cdots - [r]
       \bullet([]!{+(0,-.3)}{n - 2}) - [r] 
       \bullet([]!{+(0.2, -.3)}{n - 1}) - []!{+(-2, 1)} 
    } 
\quad
&\ul X = A_1 : &\quad
&\xygraph{
\bullet([]!{+(0,-.3)} {\ul 0}) 
}
\end{aligned}
\end{equation*}

\par\medskip
The actions of $\s$ on $I$ are as follows.
In ($A_n1$), $\s : i \lra 2n+1- i$ for $0 \le i \le 2n+1$.  
In ($A_n2$), $\s : 0 \mapsto 2n \mapsto 1 \mapsto 2n+1 \mapsto 0$ and 
$i \lra 2n + 1 - i$ for $2 \le i \le 2n-1$. 
In ($A_n3$), $\s : i \lra 2n + 1 -i$ for $1 \le i \le 2n$, and 
$\s(0) = 0$. In ($A_n4$), $\s : i \lra 2n+1 -i$ for $0 \le i \le 2n+1$. 
In ($A_n2'$), $\s: 0 \mapsto 4 \mapsto 1 \mapsto 5 \mapsto 0$ and 
$2 \lra 3$. 
In ($A_n3'$), $\s: 1 \lra 2$ and $\s(0) = 0$. 
In ($A_n4'$), $\s : 0 \lra 3, 1 \lra 2$. 
In ($A_n5$), $\s : 0 \mapsto 1 \mapsto 2 \mapsto \cdots \mapsto n -1  \mapsto 0$. 

\remarks{1.8.}
(i) \ 
In [L2], Lusztig defines a Cartan datum $\ul X$ in the non-admissible case, 
by using the inner product which corresponds to 
$(\a_{\eta} ,\a_{\eta'} )_1/\d$, where $\d = \max\{ \d_{\eta}, \d_{\eta'}\}$.  
Although the Cartan matrix is the same, (1.6.1) is more convenient for our later discussion. 
\par
(ii) \ The numbering of $A_{2n}^{(2)}$ obtained from $A_{2n+1}^{(1)}$ 
by non-admissible $\s$ (the cases ($A_n3$), ($A_n3'$)) is reverse to the order obtained 
from $D_{2n+1}^{(1)}$ by admissible $\s$ (see 5.1 (5), (6) in [SZ2]). 
The numbering of ($A_n3$), ($A_n3'$)) coincides with those in Kac's textbook [K], 
but is reverse to the one in Beck-Nakajiam [BN]. 

\remark{1.9.}
In the case ($A_n5$), the $\s$-invariant canonical basis 
does not exist. We follow the description in the lecture note by Schiffmann [S, 2.4].
Consider the cyclic quiver $\overrightarrow{Q}$ associated to $X$, where 
the orientation is given by 
$\s: 0 \mapsto 1 \mapsto 2 \mapsto \cdots \mapsto n -1 \mapsto 0$. 
Then $\s$ gives an automorphism of the quiver $\overrightarrow{Q}$.
We identify $I$ with $\BZ/n\BZ$. We denote by $I_{[i;m]}$ ($i \in I, m \in \BN$) 
the unique nilpotent
indecomposable representation of $\overrightarrow{Q}$ with socle $\ve_i$ and length $m$.     
Let $\Pi^n$ be the set of $n$-tuple of partitions $\Bla = (\la^{(1)}, \dots, \la^{(n)})$, 
where $\la^{(i)} = (\la^{(i)}_1 \ge \la^{(i)}_2 \ge \cdots)$ is a partition.
The set of isomorphism classes of nilpotent representations of $\overrightarrow{Q}$ 
is identified with $\Pi^n$ by the correspondence
\begin{equation*}
\Bla = (\la^{(1)}, \dots, \la^{(n)}) \mapsto 
    M_{\Bla} = \bigoplus_{i \in I}\bigoplus_jI_{[i; \la^{(i)}_j]}.
\end{equation*}
Let $\nu = \Dim M_{\Bla}$, and consider the representation space 
$E(\overrightarrow{Q}, \nu)$ of $\overrightarrow{Q}$. We denote by 
$\SO_{\Bla}$ the nilpotent orbit in $E(\overrightarrow{Q},\nu)$ corresponding to
$M_{\Bla}$. Let $\BP_{\Bla} = \IC(\ol\SO_{\Bla},\Ql)$ be the simple 
perverse sheaf (up to shift) on $E(\overrightarrow{Q}, \nu)$ corresponding to the orbit $\SO_{\Bla}$.
Let $\CP_{\overrightarrow{Q}}$ be the set of simple perverse sheaves as defined 
in [S, 1.4]. The following result was proved by Lusztig [L3], [L4].
\par\medskip\noindent
(1.9.1) \ 
Assume that $n > 1$.  Then 
\begin{equation*}
\CP_{\overrightarrow{Q}} = \{ \BP_{\Bla} \mid \Bla \in \Pi^n, 
                             \Bla \text{ : aperiodic } \},
\end{equation*}
where $\Bla= (\la^{(i)}_j)  \in \Pi^n$ is called aperiodic if 
$\la^{(1)}, \cdots, \la^{(n)}$ have no common parts, namely for any 
integer $c > 0$, there exists $i$ such that $c$ does not appear 
as $\la^{(i)}_j$. 
\par\medskip
$\s$ maps $M_{\Bla}$ onto $M_{\s(\Bla)}$, where 
$\s(\Bla) = (\la^{(2)}, \la^{(3)}, \dots, \la^{(n)}, \la^{(1)})$. Hence 
$\s$ acts on the set of $\{ \BP_{\Bla} \mid \Bla \in \Pi^n\}$ as a permutation, 
$\BP_{\Bla} \mapsto \BP_{\s(\Bla)}$.  
In particular, $\BP_{\Bla}$ is $\s$-invariant if and only if 
$\la^{(1)} = \la^{(2)} = \cdots = \la^{(n)}$, in which case $\Bla$ is not aperiodic.
Thus (1.9.1) implies that 
$\s$ acts on $\CP_{\overrightarrow{Q}}$ as a permutation, and there does not
exist $\s$-invariant element.
From the definition of canonical 
basis, there exits a bijection $\CP_{\overrightarrow{Q}} \isom \bB$, which is 
compatible with the action of $\s$. 
It follows that $\bB^{\s} = \emptyset$.
Since $\pi(\bB^{\s})$ gives a basis of $\BV_q$, we conclude that $\BV_q = 0$.
Summing up the above discussion, we have
\par\medskip\noindent
(1.9.2)\ In the case of ($A_n5$), an analogue of Theorem 1.4 does not hold. 

\para{1.10.}
In the rest of this paper, we assume that $X$ is irreducible 
of finite or affine type, and by Remark 1.9, we exclude the case ($A_n5$). 
Let $\BU_q^-$ (resp. $\ul\BU_q^-$) the quantum enveloping algebra 
associated to $X$ (resp. $\ul X$).  We follow the notation in [SZ1, 2].
In particular, for each $i \in I$, let $T_i : \BU_q \to \BU_q$
be the braid group action, and similarly consider $T_{\eta}: \ul\BU_q \to \ul\BU_q$ 
for $\eta \in \ul I$. 
\par
Let $W$ be the Weyl group associated to $X$. 
Assume that $X$ is of finite type, and $w_0$ is the longest element in $W$.
Let $\Bh = (i_1, \dots, i_N)$ be a sequence of $i \in I$ such that
$s_{i_1}\cdots s_{i_N}$ is a reduced expression of $w_0$. For a given sequence $\Bh$,   
the PBW-basis $\SX_{\Bh} = \{ L(\Bc, \Bh) \mid \Bc \in \BN^N\}$ for $\BU_q^-$ 
is defined as in [SZ1, 1.7].  
In the case where $X$ is of affine type, let $\Bh = ( \dots, i_{-1}, i_0, i_1, \dots)$ 
be a doubly infinite sequence of $i \in I$ as defined in [BN, 3.1]. 
For a given $\Bh$, and an integer $p$, the PBW-basis 
$\SX_{\Bh,p} = \{ L(\Bc, p) \mid \Bc \in \SC\}$ of $\BU_q^-$ is defined as in 
[SZ2, 1.5], where $\SC$ is a certain parameter set.

\par\bigskip
\section{ The algebra $\BV_q$}

\para{2.1.}
In this section, we assume that $\s$ is not admissible. 
First we consider the simplest case, namely, $X = A_2$ and $\ul X = C_1$, 
hence $I = \{ 1,2\}$ with $\s : 1 \lra 2$ and $\ul I = \{ \ul 1 \}$. 
$\BU_q^-$ has two generators $f_1, f_2$.   
Set 
\begin{equation*}
\tag{2.1.1}
\begin{aligned}
f_{12} &= T_1(f_2) = f_2f_1 - qf_1f_2,  \\ 
f'_{12} &= T_2(f_1) = f_1f_2 - qf_2f_1.
\end{aligned}
\end{equation*}
The action of $\s$ on $\BU_q^-$ is given by $\s: f_1 \lra f_2$.
Thus $\s(f_{12}) = f'_{12}$. 
For $\Bh = (1,2,1)$, $\Bh' = (2,1,2)$, (see 1.10) 
\begin{align*}
\SX_{\Bh} = \{ L(\Bc, \Bh) &= f_1^{(c_1)}f_{12}^{(c_2)}f_2^{(c_3)}  
                  \mid \Bc = (c_1, c_2, c_3) \in \BN^3 \}  \\
\SX_{\Bh'} = \{ L(\Bc, \Bh') &= f_2^{(c_1)}{f'}_{12}^{(c_2)}f_1^{(c_3)}  
                  \mid \Bc = (c_1, c_2, c_3) \in \BN^3 \}
\end{align*} 
give two kinds of PBW-bases of $\BU^-_q$.  
The following formulas are known.

\begin{align*}
\tag{2.1.2}
f_2^{(\ell)}f_1^{(m)}f_2^{(n)} &= \sum_{k = 0}^{\ell}q^{(\ell - k)(m - k)}
                      \begin{bmatrix}
                               \ell - k + n \\
                               \ell - k
                      \end{bmatrix}f_1^{(m -k)}f_{12}^{(k)}f_2^{(\ell - k + n)}, \\
\tag{2.1.3}
f_1^{(\ell)}f_2^{(m)}f_1^{(n)} &= \sum_{k = 0}^nq^{(m-k)(n - k)}
                       \begin{bmatrix}
                               n - k + \ell \\
                               n - k
                       \end{bmatrix}f_1^{(n-k+\ell)}f_{12}^{(k)}f_2^{(m - k)}.
\end{align*}

By (2.1.2) and (2.1.3), we see that, if $m \ge \ell + n$, 
$f_2^{(\ell)}f_1^{(m)}f_2^{(n)}$ is 
the canonical basis corresponding to the PBW-basis 
$f_1^{(m - \ell)}f_{12}^{(\ell)}f_2^{(n)}$, and 
$f_1^{(\ell)}f_2^{(m)}f_1^{(n)}$ is the canonical basis corresponding to 
the PBW-basis $f_1^{(\ell)}f_{12}^{(n)}f_2^{(m-n)}$.
Thus 
\begin{equation*}
\tag{2.1.4}
\bB = \{ f_2^{(\ell)}f_1^{(m)}f_2^{(n)} \mid m \ge \ell + n\} 
     \cup \{ f_1^{(\ell)}f_2^{(m)}f_1^{(n)} \mid m \ge \ell + n\}
\end{equation*}
gives the canonical basis of $\BU_q^-$, and the overlapping occurs only 
when $f_2^{(\ell)}f_1^{(m)}f_2^{(n)} = f_1^{(n)}f_2^{(m)}f_1^{(\ell)}$ 
with $m = \ell + n$. 
Note that $\s(f_2^{(\ell)}f_1^{(m)}f_2^{(n)}) = f_1^{(\ell)}f_2^{(m)}f_1^{(n)}$.
Let $\bB^{\s}$ be the set of $\s$-invariant elements in $\bB$. 
It follows from the above discussion that 

\par\medskip\noindent
(2.1.5) \ We have  
$\s(f_2^{(\ell)}f_1^{(m)}f_2^{(m - \ell)}) = f_1^{(\ell)}f_2^{(m)}f_1^{(m - \ell)} 
      = f_2^{(m -\ell)}f_1^{(m)}f_2^{(\ell)}$ for $\ell \le m$.  
In particular, 
\begin{equation*}
 \bB^{\s} = \{ f_2^{(a)}f_1^{(2a)}f_2^{(a)} \mid a \in \BN \}. 
\end{equation*}

More generally, the following result is also obtained from  the previous discussion.
\par\medskip\noindent
(2.1.6) \ Assume that $\nu = - m(\a_1 + \a_2) \in Q_-$.  Then 
$\{ f_1^{(\ell)}f_{12}^{(m - \ell)}f_2^{(\ell)} \mid 0 \le \ell \le m \}$ 
gives the PBW -basis, and 
$\{ f_2^{(\ell)}f_1^{(m)}f_2^{(m - \ell)} \mid 0 \le \ell \le m\}$  gives 
the canonical basis of $(\BU_q^-)_{\nu}$, under the correspondence
$f_1^{(\ell)}f_{12}^{(m - \ell)}f_2^{(\ell)} \lra f_2^{(\ell)}f_1^{(m)}f_2^{(m - \ell)}$. 

\par\medskip
Concerning the action of $\s$ on PBW-basis, we have the following..

\begin{lem}  
Assume that $X$ is of type $A_2$. Then 
$\s$-invariant PBW-bases do not exist (except the trivial one).  
The PBW-basis corresponding to the $\s$-invariant canonical basis 
$f_2^{(a)}f_1^{(2a)}f_2^{(a)}$ is given by $f_1^{(a)}f_{12}^{(a)}f_2^{(a)}$
for $\Bh = (1,2)$, and by $f_2^{(a)}f'^{(a)}_{12}f_1^{(a)}$ for 
$\Bh' = (2,1)$. 
We have, for $a \ge 1$,  
\begin{equation*}
\s(f_1^{(a)}f^{(a)}_{12}f_2^{(a)}) = f_2^{(a)}f'^{(a)}_{12}f_1^{(a)} 
        \ne f_1^{(a)}f^{(a)}_{12}f_2^{(a)}.
\end{equation*}
\end{lem}

\begin{proof}
The latter assertion is clear from the previous discussion. 
We show that $\s$-invariant PBW-basis does not exist. 
If a PBW-basis $x$ is $\s$-invariant, then the corresponding canonical 
basis is $\s$-invariant, hence has the form $f_2^{(a)}f_1^{(2a)}f_2^{(a)}$ 
for some $a \in \BN$. It follows that $x = f_1^{(a)}f_{12}^{(a)}f_2^{(a)}$, 
and $y = \s(x) = f_2^{(a)}f'^{(a)}_{12}f_1^{(a)}$.  
It is enough to show that $x \ne y$ if $a \ge 1$. 
By (2.1.2), one can write as 
\begin{equation*}
\tag{2.2.1}
f_2^{(a)}f_1^{(2a)}f_2^{(a)} = f_1^{(a)}f_{12}^{(a)}f_2^{(a)} 
                                 + \sum_{0 \le k < a}c_kf_2^{(k)}f_1^{(2a)}f_2^{(2a - k)}
\end{equation*}
with $c_k \in \BQ(q)$, where $c_k \ne 0$ for some $k$.
If $f_1^{(a)}f_{12}^{(a)}f_2^{(a)}$ is $\s$-invariant, 
then the sum in the right hand side of (2.2.1) is $\s$-invariant. 
It follows by (2.1.5) that $c_k = c_{2a - k}$ for any $k$.
But since $0 \le k < a$, we have $2a - k > a$, and $c_{2a - k}$ does not appear 
in this sum.  This is absurd, and $f_1^{(a)}f_{12}^{(a)}f_2^{(a)}$ 
is not $\s$-invariant.  
The lemma is proved. 
\end{proof}

\para{2.3.} 
We define $\BV_q = {}_{\BA'}\BU_q^{-,\s}/J$ as in 1.2, and
let $\pi : {}_{\BA'}\BU_q^{-,\s} \to \BV_q$ be the natural projection.
Let $\bB$ be the canonical basis of $\BU_q^-$ (of type $A_2$).  Then 
$\s$ permutes $\bB$, and $\bB^{\s}$ is given as in (2.1.5).  
For each $a \ge 0$, set
\begin{equation*}
\tag{2.3.1}
g^{(a)}_{\ul 1} = \pi(f_2^{(a)}f_1^{(2a)}f_2^{(a)}) 
                = \pi(f_1^{(a)}f_2^{(2a)}f_1^{(a)})\in \BV_q
\end{equation*}
Then $\pi(\bB^{\s}) = \{ g^{(a)}_{\ul 1} \mid a \in \BN \}$ gives 
an $\BA'$-basis of $\BV_q$.  

We consider $\ul\BU_q^-$ associated to the Cartan datum $\ul X$ of type
$C_1$. The canonical basis of $\ul\BU_q^-$ is given by 
$\{ \ul f^{(a)}_{\ul 1} \mid a \in \BN \}$, where 
$\ul f^{(a)}_{\ul 1} = ([a]^!_4)\iv \ul f_{\ul 1}^a$. 
(Note that $d_{\ul 1} = (\a_{\ul 1}, \a_{\ul 1})_1/2 = 4$ by 
(1.6.1).)
${}_{\BA}\ul\BU_q^-$ is an $\BA$-subalgebra of $\ul\BU_q^-$ generated by 
$\ul f^{(a)}_{\ul 1}$ ($a \in \BN$).  $\BA'$-algebra ${}_{\BA'}\ul\BU_q^-$ is defined 
as before. We show the following.

\begin{prop}  
Assume that $X$ is of type $A_2$.   Then we have 
\begin{equation*}
\tag{2.4.1}
[a]^!_4g_{\ul 1}^{(a)} = g_{\ul 1}^a.
\end{equation*} 
The correspondence $\ul f_{\ul 1}^{(a)} \mapsto g_{\ul 1}^{(a)}$ 
gives an isomorphism 
$\Phi: {}_{\BA'}\ul\BU_q^- \isom \BV_q$ of $\BA'$-algebras.  
\end{prop}

\begin{proof}
Since $\ch \BF = 2$, in order to prove the proposition, it is enough to
show that 
\begin{equation*}
\tag{2.4.2}
(f_1f_2^{(2)}f_1)^a \equiv ([a]^!)^4f_1^{(a)}f_2^{(2a)}f_1^{(a)}  \mod J.
\end{equation*}

For $a \ge 1$, we consider the following statement. 

\begin{equation*}
\tag{2.4.3}
[a]^4f_1^{(a)}f_2^{(2a)}f_1^{(a)} 
      \equiv (f_1^{(a-1)}f_2^{(2a-2)}f_1^{(a-1)})(f_1f_2^{(2)}f_1)  \mod J.
\end{equation*}
Note that if (2.4.3) holds, then (2.4.2) follows by induction. 
We prove (2,4,3). 
Let $Z$ be the right hand side of (2.4.3). 
By using the formula (2.1.3), we can express $Z$ as a linear 
combination of PBW-bases $f_1^{(a)}f_{12}^{(b)}f_2^{(c)}$, 
\begin{align*}
\sum_{k = 0}^{a-1}
         \bigl(A_kf_1^{(2a-1-k)}f_{12}^{(k+1)}f_2^{(2a-1-k)}  
         + B_kf_1^{(2a-2-k)}f_{12}^{(k+2)}f_2^{(2a-2 -k)}
         + C_kf_1^{(2a-k)}f_{12}^{(k)}f_2^{(2a - k)}\bigr), 
\end{align*}
where 
\begin{align*}
A_k &= q^{(2a-2-k)(a-1-k)-k}\begin{bmatrix}
                             2a - 2 - k \\
                              a - 1 - k
                             \end{bmatrix}[2a-1-k]^2[k + 1]  \\
      &\qquad + q^{(2a-2-k)(a-1-k) + (2a- 1 - 2k)}
                            \begin{bmatrix}
                             2a - 2 - k \\
                             a  - 1  - k
                              \end{bmatrix}[2a-1-k]^2[k+1][2a - 2 - k], \\         
B_k &= q^{(2a-2-k)(a-1-k)- (2a-3-k)}
                            \begin{bmatrix}
                              2a-2-k  \\
                               a - 1 - k
                            \end{bmatrix}[2a-2-k][k+1][k+2] \\
       &\qquad + q^{(2a-2-k)(a-1-k) + 2}
                             \begin{bmatrix}
                               2a - 2 - k \\
                                a -1 -k
                             \end{bmatrix} [2]\iv[2a-3-k][2a-2-k][k+1][k+2],  \\
C_k &= q^{(2a-2-k)(a-1 k) + (4a-2 - 4k)}
                             \begin{bmatrix}
                               2a - 2 - k   \\
                                a - 1 - k
                             \end{bmatrix}[2]\iv [2a - 1 - k]^2[2a - 2 - k]^2. 
\end{align*}
From this formula, we see that the coefficient of $f_1^{(2a-k)}f_{12}^{(k)}f_2^{(2a-k)}$
in $Z$ is equal to 
\begin{equation*}
\tag{2.4.4}
\begin{cases}
     A_{a-1} + B_{a-2}  &\quad\text{ if $k = a$, } \\
     B_{a-1}            &\quad\text{ if $k = a+1$, } \\
     0                  &\quad\text{ if $k \ge a+2$.}
\end{cases}
\end{equation*}

On the other hand, by using (2.1.6), $Z$ can be expressed as 
\begin{equation*}
\tag{2.4.5}
Z = \sum_{0 \le i \le 2a} m_i f_1^{(i)}f_2^{(2a)}f_1^{(2a-i)}.
\end{equation*}
Note that since $Z$ is $\s$-invariant, we see, by (2.1.5), 
 that $m_i = m_{2a-i}$ for each $i$. 
By using (2.1.3), each $f_1^{(i)}f_2^{(2a)}f_1^{(2a-i)}$ is written 
by the PBW -basis. Then the coefficient of $f_1^{(2a-k)}f_{12}^{(k)}f_2^{(2a-k)}$
in $Z$ is given as 
\begin{equation*}
\tag{2.4.6}
\sum_{i = 0}^{2a-k}m_iq^{(2a-k)(2a- k - i)}\begin{bmatrix}
                                         2a - k \\
                                         2a - k - i
                                      \end{bmatrix}.
\end{equation*}
By comparing (2.4.4) with (2.4.6), we see that
$m_0 = \cdots = m_{a-2} = 0$, and 
\begin{equation*}
m_{a-1} = B_{a-1}, \qquad m_{a-1}q^a[a] + m_a = A_{a-1} + B_{a-2},  
\end{equation*}
where
\begin{align*}
B_{a-1} &= q^{-a + 2}[a-1][a][a+1] + q^2[2]\iv [a-2][a-1][a][a+1], \\
A_{a-1} + B_{a-2} &= (q^{-a+1}[a]^3 + q[a]^3[a-1]) + 
                       (q[a]^3[a-1] + q^{a+2}[2]\iv [a-1]^2[a]^3).
\end{align*}
This implies that $m_a = [a]^4$. 
$Z$ is $\s$-invariant, and in the expansion of $Z$ in (2.4.5)
the only $\s$-invariant term  is $f_1^{(a)}f_2^{(2a)}f_1^{(a)}$, 
whose coefficient is equal to $m_a = [a]^4$.
It follows that $Z \equiv [a]^4f_1^{(a)}f_2^{(2a)}f_1^{(a)} \mod J$, and   
(2.4.3) holds. The lemma is proved. 
\end{proof}

\para{2.5.}
We now consider the general case, namely, $X$ is irreducible of finite or affine type, 
and $\s$ is non-admissible. We exclude the case ($A_n5$).
$\BV_q$ is defined as in 1.2, and we can consider 
${}_{\BA'}\ul\BU_q$ corresponding to $\ul X$. 
But as observed in the case where 
$X = A_2$ (Proposition 2.4), the definition of $\Phi : {}_{\BA'}\ul\BU_q^- \to \BV_q$
for the admissible case in 1.3 is not applicable for the non-admissible case. 
We modify the definition of $\Phi$ as follows.  
Note that if $X$ is irreducible and $\d_{\eta} = 2$, then 
$\ve = 2$ or 4, and $\eta = \{ i, i'\}$ with $(\a_i, \a_{i'}) = -1$
(see 1.7). 
\par
For $\eta \in \ul I$, and $a \in \BN$,
we define $\wt f_{\eta} \in {}_{\BA}\BU_q^-$ by 
\begin{equation*}
\tag{2.5.1}
\wt f^{(a)}_{\eta} = \begin{cases}
               f^{(a)}_if_{i'}^{(2a)}f_i^{(a)}
                          &\quad\text{ if $\d_{\eta} = 2$ and 
                    $\eta = \{ i,i'\}$,} \\
               \prod_{i \in \eta}f_i^{(a)}
                          &\quad\text{ if $\d_{\eta}  = 1$,}
                 \end{cases}
\end{equation*} 
and set
\begin{equation*}
\tag{2.5.2}
g_{\eta}^{(a)} = \pi(\wt f_{\eta}^{(a)}).
\end{equation*}
Note that $f_i^{(a)}f_{i'}^{(2a)}f_i^{(a)} = f_{i'}^{(a)}f_i^{(2a)}f_{i'}^{(a)}$ 
if $\d_{\eta} = 2$, and $\wt f^{(a)}_{\eta}$ does not depend on the order
of the product if $\d_{\eta} = 1$.  

We prove the following.

\begin{thm}  
Assume that $X$ is irreducible of finite or affine type, and 
$\s$ is not admissible. 
We exclude the cases ($A_n2'$), ($A_n3'$), ($A_n4'$), ($A_n5$) in 1.8.
Then the correspondence $\ul f^{(a)}_{\eta} \mapsto g_{\eta}^{(a)}$ 
gives rise to an isomorphism $\Phi : {}_{\BA'}\ul\BU_q^- \isom  \BV_q$. 
\end{thm}

\remark{2.7.}
The theorem does not hold for the case ($A_n5$) since 
$\BV_q = \{ 0\}$ by Remark 1.9.
It is likely that the theorem holds for the cases ($A_n2'$) $\sim$ ($A_n4'$).
Actually, in those three cases, the verification of the fact that 
$\Phi : {}_{\BA'}\ul\BU_q^- \to \BV_q$ is a homomorphism (Proposition 2.8) 
becomes more complicated, and it is not yet done. If we can find 
a combinatorial discussion as in [MSZ, 5], this might be possible.    
\par\medskip
The first step towards the proof of the theorem is the following.

\begin{prop}  
Under the assumption of Theorem 2.6, 
the assignment $\ul f^{(a)}_{\eta} \mapsto g_{\eta}^{(a)}$ 
gives rise to a homomorphism $\Phi : {}_{\BA'}\ul\BU_q^- \to \BV_q$. 
\end{prop}

\para{2.9.}
Proposition 2.8 will be proved in Section 4.  
Here, assuming Proposition 2.8, we continue the discussion. 
The algebra homomorphism $r : \BU_q^- \to \BU_q^-\otimes \BU_q^-$ 
is defined by $r(f_i) = f_i\otimes 1 + 1 \otimes f_i$, with respect 
to the twisted algebra structure of $\BU_q^-\otimes \BU_q^-$.  The 
symmetric bilinear form  $(\ ,\ )$ on $\BU_q^-$ is defined by 
using the property of $r$.  The bilinear form $(\ ,\ )$ is 
non-degenerate.  A similar result also hold for $\ul\BU_q^-$ (see [MSZ, 1.3]). 
Let $\BF(q)$ be the field of rational functions of $q$ with coefficients in $\BF$, 
which contains $\BA' = \BF[q,q\iv]$. 
Set ${}_{\BF(q)}\BV_q = \BF(q)\otimes_{\BA'}\BV_q$. 
Then by [MSZ, 3.6], the bilinear form $(\ ,\ )$ on $\BU_q^-$ induces 
a non-degenerate symmetric bilinear form $(\ ,\ )$ on ${}_{\BF(q)}\BV_q$.
(Note that the discussion there holds even if $\s$ is not admissible). 
On the other hand, the bilinear form $(\ ,\ )$ on $\ul\BU_q^-$ 
induces a symmetric bilinear form  on ${}_{\BF(q)}\ul\BU_q^-$. This 
bilinear form is non-degenerate by [MSZ, Prop. 3.7]. 
We have the following proposition.

\begin{prop}  
\begin{enumerate}
\item \ For any $x, y \in {}_{\BA'}\ul \BU_q^-$, we have 
$(\Phi(x), \Phi(y)) = (x,y)$. 
\item The map $\Phi: {}_{\BA'}\ul\BU_q^- \to \BV_q$ is injective. 
\end{enumerate}
\end{prop}
   
\begin{proof}
The proposition is proved in a similar way as in 
[SZ2, Prop. 2.8], [MSZ, Prop. 3.8].
The map $\Phi$ can be extended to a homomorphism 
$\wt\Phi : {}_{\BF(q)}\ul\BU_q^- \to {}_{\BF(q)}\BV_q$. 
In order to prove (ii), it is enough to show that $\wt\Phi$ is injective. 
This follows from (i) since the bilinear form on ${}_{\BF(q)}\ul\BU_q^-$ 
is non-degenerate.  We prove (i).
The $\BA$-submodule $J_1$ of $({}_{\BA}\BU_q^-\otimes {}_{\BA}\BU_q^-)^{\s}$ 
is defined as in the proof of Proposition 3.8 in [MSZ]. 
Note that in our case, $\ve = 2$ or 4, $J_1$ coincides with the $\BA$-submodule 
generated by $x + \s(x)$ such that $\s^2(x) = x, \s(x) \ne x$. 
We note the 
formula, for $\eta \in \ul I$, 
\begin{equation*}
\tag{2.10.1}
r(\wt f_{\eta}) \equiv \wt f_{\eta}\otimes 1 + 1 \otimes \wt f_{\eta} \mod J_1.
\end{equation*}
In fact, (2.10.1) was proved in [MSZ] if $\d_{\eta} = 1$. 
We verify this for $\eta$ such that $\d_{\eta} = 2$. 
Assume that $\eta = \{ i,i'\}$ with $(\a_i, \a_{i'}) = -1$. 
We have 
\begin{align*}
\tag{2.10.2}
r(f_if_{i'}^{(2)}f_i) &= [2]\iv (f_i\otimes 1 + 1 \otimes f_i)
                               (f_{i'}\otimes 1 + 1 \otimes f_{i'})^2
                               (f_i\otimes 1 + 1 \otimes f_i)  \\
      &= f_if_{i'}^{(2)}f_i \otimes 1 + 1 \otimes f_if_{i'}^{(2)}f_i + Z, 
\end{align*}
where 
\begin{align*}
Z &= q\iv(f_if_{i'} \otimes f_{i'}f_i + f_{i'}f_i \otimes f_if_{i'}) \\
     &\quad + q[2](f_i^{(2)}\otimes f_{i'}^{(2)} + f_{i'}^{(2)}\otimes f_i^{(2)}) \\ 
     &\quad + f_i\otimes f_if_{i'}^{(2)} + f_i \otimes f_{i'}^{(2)}f_i 
                 + f_{i'}\otimes f_if_{i'}f_i  \\
     &\quad + f_{i'}^{(2)}f_i \otimes f_i + f_if_{i'}^{(2)} \otimes f_i
                 + f_if_{i'}f_i \otimes f_{i'}. 
\end{align*}
The first two terms in $Z$ are of the form $x + \s(x)$, thus they are 
contained in $J_1$. 
In the third term, by using the Serre relation 
$f_if_{i'}^{(2)} + f_{i'}^{(2)}f_i = f_{i'}f_if_{i'}$. we have
\begin{equation*}
f_i\otimes f_if_{i'}^{(2)} + f_i\otimes f_{i'}^{(2)}f_i= f_i\otimes f_{i'}f_if_{i'}. 
\end{equation*}
Thus the third term is contained in $J_1$.  Similarly, the 4th term is contained in $J_1$. 
Hence $Z \in J_1$, and (2.10.1) holds. 
Then the proof of (2.8.2) and (2.8.7) in [SZ2] is done in a similar 
way as in the proof of Proposition 3.8 in [MSZ]. 
As in [SZ2], [MSZ], the proof of the proposition is reduced to the computation of 
$(\wt f_{\eta}, \wt f_{\eta'})$. 
If $\d_{\mu} = \d_{\mu'} = 1$, $(\wt f_{\eta}, \wt f_{\eta'})$ was computed in 
(3.8.1) in [MSZ]. If $\d_{\eta} \ne \d_{\eta'}$, $(\wt f_{\eta}, \wt f_{\eta'}) = 0$ 
since the weights are different.  By the same reason, 
$(\wt f_{\eta}, \wt f_{\eta'}) = 0$ if $\d_{\eta} = \d_{\eta'} = 2$ 
and $\eta \ne \eta'$.   Hence we may assume that $\eta = \eta'$ with $\d_{\eta} = 2$.
We have
\begin{align*}
(\wt f_{\eta}, \wt f_{\eta}) = (f_if_{i'}^{(2)}f_i, f_if_{i'}^{(2)}f_i) 
              = [2]\iv(r(f_if_{i'}^{(2)}f_i), f_if_{i'}\otimes f_{i'}f_i).
\end{align*}
$r(f_if_{i'}^{(2)}f_i)$ is computed in (2.10.2).  By comparing the weights, 
the last term is equal to 
\begin{align*}
q\iv[2]\iv&(f_if_{i'}\otimes f_{i'}f_i + f_{i'}f_i\otimes f_if_{i'}, 
               f_if_{i'}\otimes f_{i'}f_i)  \\
&= q\iv[2]\iv \bigl((f_if_{i'}, f_if_{i'})(f_{i'}f_i, f_{i'}f_i) 
            + (f_{i'}f_i, f_if_{i'})(f_if_{i'}, f_{i'}f_i)\bigr)  \\
&= q\iv [2]\iv (q^2 +1)(f_i, f_i)(f_{i'}, f_{i'})(f_{i'}, f_{i'})(f_i, f_i) \\
&= (1 - q^2)^{-4}.  
\end{align*}

Thus we have
\begin{equation*}
(\wt f_{\eta}, \wt f_{\eta'}) = \begin{cases}
            (1 - q^2)^{-4}  &\quad\text{ if } \eta = \eta', \d_{\eta} = 2,  \\
            (1 - q^2)^{-|\eta|}    &\quad\text{ if } \eta = \eta', \d_{\eta} = 1,  \\
                              0                    &\quad\text{ if } \eta \ne \eta'.
                            \end{cases} 
\end{equation*}
Since $q_{\eta} = q^{d_{\eta}} = q^4$ if $\d_{\eta} = 2$, and 
$q_{\eta} = q^{|\eta|}$ if $\d_{\eta} = 1$ by (1.6.1), by comparing with 
\begin{equation*}
(\ul f_{\eta}, \ul f_{\eta'})_1 = \begin{cases}
                         (1 - q^2_{\eta})\iv  &\quad\text{ if $\eta = \eta'$ }  \\
                                0             &\quad\text{ if $\eta \ne \eta'$ }
                                \end{cases}
\end{equation*}
we obtain the result corresponding to (2.8.3) in [SZ2]. 
Now the assertion (i) follows from (2.8.3) as in [SZ2]. 
The proposition is proved. 
\end{proof}

\para{2.11.}
For the surjectivity of $\Phi$, the discussion in [SZ2] or [MSZ] 
cannot be applied directly. The main ingredient for the proof 
is the modified PBW-bases.  In the case where $X$ is of finite type, 
$\s$ acts on the set of modified PBW-bases as a permutation. 
In this case, the surjectivity will be proved in Theorem 3.7, 
in a similar way as in [SZ1].  While in the case where $X$ is of affine type, 
it is not certain whether $\s$ acts on the set of modified PBW-bases.  
In this case, we prove the surjectivity in Proposition 6.21, by applying the properties 
of Kashiwara operators.  The discussion is similar to [SZ2].  However, 
in contrast to [SZ2], we construct Kashiwara operators by making use of 
modified PBW-bases.   
\par\bigskip

\section{ Modified PBW-basis - finite case }

\para{3.1.}
In this section, we assume that $X$ is of finite type.  
Let $W$ (resp. $\ul W$) be the Weyl group of $X$ (resp. of $\ul X$).
Let $\ul \Bh$ be a sequence of $\ul I$ with respect to $\ul W$.
Then there exists a sequence $\Bh$ of $I$ obtained from $\ul\Bh$, unique 
up to the permutation of $i$ and $j$ such that $(\a_i,\a_j) = 0$  
(see [SZ1, 1.6]).
Let $\SX_{\Bh}$ be the PBW-basis of $\BU_q^-$ and $\ul\SX_{\ul\Bh}$ 
the PBW-basis of $\ul\BU_q^-$ as in 1.10.
It is shown in [SZ1, Lemma 1.8], in the case where $\s$ is admissible, 
that $\s$ induces a permutation on 
the set $\SX_{\Bh}$, 
and the set of $\s$-invariant PBW-basis $\SX_{\Bh}^{\s}$ 
corresponds bijectively to the set $\SX_{\ul \Bh}$, 
If $\s$ is not admissible, this does not hold, since for $X = A_2$, 
there does not exist $\s$-invariant PBW-basis (Lemma 2.2). 
In order to extend Lemma 1.8 in [SZ1] to the non-admissible
case, we will modify the definition of PBW-basis. 

\para{3.2.}
We recall the definition of PBW-bases in [SZ1].  Let 
$\Bh = (i_1, \dots, i_{\nu})$ be a sequence of $I$ (called the reduced sequence) 
corresponding 
to a reduced expression $w_0 = s_{i_1}\cdots s_{i_N}$ of $w_0 \in W$, 
where $N = l(w_0)$. 
Let $\vD \subset Q$ be the set of roots, and $\vD^+ \subset Q_+$ 
the set of positive roots.  
Then $\Bh$ gives a total
order $\{\b_1, \dots, \b_N \}$ of $\vD^+$, where $\b_i$ is 
defined by 
$\b_{i_k} = s_{i_1}\dots s_{i_{k-1}}(\a_{i_k})$. 
We define a root vector $f_{\b_k}^{(c)}$ corresponding to $\b_k \in \vD^+$ 
and $c \in \BN$ by 
\begin{equation*}
\tag{3.2.1}
f_{\b_k}^{(c)} = T_{i_1}\cdots T_{i_{k-1}}(f_{i_k}^{(c)}).
\end{equation*}
Note that $\b_{i_1} = \a_{i_1}$ and so $f_{\b_{i_1}}^{(c)} = f_{i_1}^{(c)}$. 
For $\Bc = (c_1, \dots, c_N) \in \BN^N$, 
set $L(\Bc,\Bh) = f_{\b_{i_1}}^{(c_1)}\cdots f_{\b_{N}}^{(c_{N})}$. 
Then $\SX_{\Bh} = \{ L(\Bc, \Bh) \mid \Bc \in \BN^N \}$ gives a PBW-basis
of $\BU_q^-$ associated to $\Bh$. 
For a sequence $\ul \Bh$ of $\ul I$, the PBW-basis $\ul\SX_{\ul \Bh}$ of $\ul\BU_q^-$ 
is defined similarly. 

\par
Let $\s : I \to I$ be the diagram automorphism. 
Here we assume that $\s$ is not necessarily admissible. 
$\s$ induces an automorphism 
$\s: W \to W$. Let $\ul W$ be the Weyl group associated to $\ul I$.
For each $\eta \in \ul I$, we denote by $t_{\eta} \in W$ the longest element 
in the Weyl subgroup $W_{\eta}$ of $W$ generated by $\{ s_i \mid i \in \eta\}$.  
Then $t_{\eta} \in W^{\s}$, 
and the correspondence $s_{\eta} \mapsto t_{\eta}$ gives an isomorphism 
$\ul W  \isom W^{\s}$.  
\par
Let $\ul\Bh = (\eta_1, \dots, \eta_{\ul N})$ 
be a reduced sequence of $\ul I$, where $\ul N = l(\ul w_0)$.  
For each $\eta \in \ul I$, we denote by $\Bh_{\eta}$ a 
reduced sequence of $W_{\eta}$.  Hence 
$\Bh_{\eta} = (i_1, \dots, i_{|\eta|})$ for $\eta = \{i_1, \dots, i_{|\eta|}\}$ 
if $\d_{\eta} = 1$, and $\Bh_{\eta} = (i,j,i)$ or $(j,i,j)$ if $\eta = \{ i,j\}$
with $\d_{\eta} = 2$.  
Then by a juxtaposition of $\Bh_{\eta}$ for each $\eta \in \ul I$, 
we obtain a reduced sequence $\Bh$ of $I$, namely,
\begin{equation*}
\tag{3.2.2}
\Bh = (\underbrace {i_1, \dots, i_a}_{\Bh_{\eta_1}}, 
         \underbrace {i_{a+1}, \dots, i_b}_{\Bh_{\eta_2}}, i_{b+1}, \dots )
\end{equation*}

\par
For each $\eta \in \ul I$, we define $\wt T_{\eta}$ by 
\begin{equation*}
\tag{3.2.3}
\wt T_{\eta} = \begin{cases}
                 \prod_{i \in \eta}T_i &\quad\text{ if $\d_{\eta} = 1$,}  \\ 
                 T_iT_jT_i = T_jT_iT_j &\quad\text{ if $\d_{\eta} = 2$ with 
                                           $\eta = \{ i,j\}$.}
               \end{cases}
\end{equation*} 
Then $\wt T_{\eta}$ does not depend on the choice of the order in $\eta$. 
Since $\s \circ T_i \circ \s\iv = T_{\s(i)}$, $\wt T_{\eta}$ commutes
with $\s$.                                                    

\par
Now assume that $\s$ is admissible. 
For each $\Bh$ and $\Bc \in \BN^N$, set
\begin{equation*}
\tag{3.2.4}
\wt f_{\eta_k}^{(\Bc_k)} = \prod_{i_{\ell} \in \eta_k}f^{(c_{\ell})}_{\b_{\ell}},
\end{equation*}
where $\Bc_k$ is a subsequence of $\Bc$ consisting of 
$c_{\ell}$ such that $i_{\ell} \in \eta_k$.  Thus 
we have $L(\Bc, \Bh) 
   = \wt f_{\eta_1}^{(\Bc_1)}\cdots \wt f_{\eta_{\ul N}}^{(\Bc_{\ul N})}$.  
Note that $\wt f_{\eta_k}^{(\Bc_k)}$ can be expressed as
\begin{equation*}
\tag{3.2.5}
\wt f_{\eta_k}^{(\Bc_k)} = \wt T_{\eta_1}\cdots \wt T_{\eta_{k-1}}
                      \bigl(\prod_{i_{\ell} \in \eta_k}f_{i_{\ell}}^{(c_{\ell})}\bigr).
\end{equation*}
It follows from (3.2.5) that 
\begin{equation*}
\s(\wt f_{\eta_k}^{(\Bc_k)}) = \wt T_{\eta_1}\cdots \wt T_{\eta_{k-1}}
              \bigl(\prod_{i_{\ell} \in \eta_k}f_{\s(i_{\ell})}^{(c_{\ell})}\bigr).
\end{equation*}
This implies that $\s(L(\Bc, \Bh)) = L(\s(\Bc), \Bh)$, where 
the action of $\s$ on $\Bc \in \BN^N$ is given, for each 
$\Bc_k = (a_1, \dots, a_{m-1}, a_m)$ by $\s(\Bc_k) = (a_m, a_1, \dots, a_{m-1})$.  
\par
For each $\ul \Bc = (c_1, \dots, c_{\ul N}) \in \BN^{\ul N}$, 
define $\Bc \in \BN^N$ by 
\begin{equation*}
\tag{3.2.6}
\Bc = (\underbrace{c_1, \dots, c_1}_{\eta_1}, 
          \underbrace{c_2, \dots, c_2}_{\eta_2}, c_3, \dots ), 
\end{equation*}
Then $L(\Bc, \Bh)$ is $\s$-invariant. Summing up the above discussion 
we have 
\begin{lem}[{[SZ1, Lemma 1.8]}]  
Assume that $\s$ is admissible.  Then 
$\s$ acts on $\SX_{\Bh}$, as a permutation, 
$\s : L(\Bc, \Bh) \mapsto L(\s(\Bc), \Bh)$. 
Let $\SX_{\Bh}^{\s}$ be the set of $\s$-invariant
PBW-basis of $\BU_q^-$. 
For each $\ul \Bc \in \BN^{\ul N}$, 
define $\Bc \in \BN^{N}$ as in (3.2.6). 
Then $L(\ul\Bc, \ul\Bh) \mapsto L(\Bc, \Bh)$ 
gives a bijection $\ul\SX_{\ul\Bh} \isom \SX_{\Bh}^{\s}$. 
\end{lem} 

\para{3.4.}
We now assume that $\s$ is not admissible. 
We modify the definition of $\wt f_{\eta_k}^{(\Bc_k)}$ as follows. 
Assume that $\d_{\eta} = 2$ with $\eta = \{ i, j\}$.  Let 
$\Bh_{\eta} = (i,j,i)$ or $(j,i,j)$ be a reduced sequence associated 
to $W_{\eta}$. Let $\SX_{\Bh_{\eta}}$ be the set 
of PBW-basis of the subalgebra $\BL^-_{\eta}$ of $\BU_q^-$ isomorphic to 
$\BU_q(\Fs\Fl_3)^-$.  
For each PBW-basis $L(\Bc, \Bh_{\eta}) \in \SX_{\Bh_{\eta}}$ in $\BL^-_{\eta}$, 
the canonical basis $b(\Bc, \Bh_{\eta}) \in \BL^-_{\eta}$ can be assigned, 
and $\bB_{\eta} = \{ b(\Bc, \Bh_{\eta}) \mid \Bc \in \BN^3\}$
gives the canonical basis of $\BL^-_{\eta}$. 
We define $\wt f_{\eta_k}^{(\Bc_k)}$ by 
\begin{equation*}
\tag{3.4.1}
\wt f_{\eta_k}^{(\Bc_k)} = \wt T_{\eta_1}\cdots \wt T_{\eta_{k-1}}
                           (b(\Bc_k, \Bh_{\eta_k})).
\end{equation*}  
In the case where $\d_{\eta} = 1$, define $\wt f_{\eta_k}^{(\Bc_k)}$
by using (3.2.5).  For any $\Bc \in \BN^{N}$, define
\begin{equation*}
\tag{3.4.2}
L\nat(\Bc, \Bh) = \wt f_{\eta_1}^{(\Bc_1)}\cdots \wt f_{\eta_{\ul N}}^{(\Bc_{\ul N})}, 
\end{equation*} 
and set $\SX_{\Bh}\nat = \{ L\nat(\Bc, \Bh) \mid \Bc \in \BN^{N}\}$.   
In the case where $\d_{\eta} = 2$, $\s$ leaves $\BL^-_{\eta}$ invariant, and $\s$ permutes
the canonical basis $\bB_{\eta}$.  It follows that 
$\s(\wt f_{\eta_k}^{(\Bc_k)}) = \wt f_{\eta_k}^{(\Bc'_k)}$, where 
$\Bc_k'$ is determined by the condition that 
$\s(b(\Bc_k, \Bh_{\eta_k})) = b(\Bc_k', \Bh_{\eta_k})$.  
From this, we see that $\s$ acts on $\SX_{\Bh}\nat$ as a permutation. 
Note that, for a given $\ul\Bc$, the definition of $\Bc$ in (3.2.6) 
makes sense even for the non-admissible case (under the assignment 
$c \mapsto (c,c,c)$ for the case $\d_{\eta} = 2$), and one can define 
$L\nat(\Bc, \Bh)$.   
\par
The following result is a generalization of Lemma 3.3 to the case
where $\s$ is non-admissible. 

\begin{lem}  
Assume that $\s$ is not necessarily admissible. 
Then $\SX_{\Bh}\nat$ is a $\BQ(q)$-basis of $\BU_q^-$, and an 
$\BA$-basis of ${}_{\BA}\BU_q^-$.  $\s$ acts as a permutation 
on $\SX_{\Bh}\nat$. The correspondence $L(\ul\Bc, \ul\Bh) \mapsto L\nat(\Bc, \Bh)$
gives a bijection $\ul\SX_{\ul\Bh} \isom (\SX_{\Bh}\nat)^{\s}$.  
\end{lem}

\begin{proof}
Assume that $\eta_k = \{ i, j\}$ with $\d_{\eta_k} = 2$, 
and take $\Bh_{\eta_k} = (i,j, i)$.  Then 
$b(\Bc_k, \Bh_{\eta_k})$ is written as $b(\Bc_k, \Bh_{\eta_k}) = L(\Bc_k, \Bh_{\eta_k}) + z$, 
where $z$ is a $q\BZ[q]$-linear combination of PBW-bases 
$L(\Bc'_k, \Bh_{\eta_k})$, with $\Bc'_k > \Bc_k$ (the lexicographic order on $\BN^3$ 
from the left). 
Furthermore, 
$L(\Bc'_k, \Bh_{\eta_k}) = f_i^{(c_1')}T_i(f_j^{(c_2')})T_iT_j(f_i^{(c_3')})$ 
for $\Bc'_k = (c_1', c_2', c_3')$. 
It follows that 
\begin{align*}
\wt T_{\eta_1}\cdots \wt T_{\eta_{k-1}}(L(\Bc'_k, \Bh_{\eta_k}))
               &= T_{i_1}\cdots T_{i_{\ell-3}}
                   \bigl(f_{i_{\ell-2}}^{(c_1')}
                     T_{i_{\ell-2}}(f_{i_{\ell-1}}^{(c_2')})
                        T_{i_{\ell-2}}T_{i_{\ell-1}}(f_{i_{\ell}}^{(c_3')})\bigr)  \\
               &= f_{\b_{\ell-2}}^{(c_1')}f_{\b_{\ell-1}}^{(c_2')}
                                f_{\b_{\ell}}^{(c_3')}, 
\end{align*}  
where we write $\Bh_{\eta_k} = (i_{\ell-2}, i_{\ell - 1}, i_{\ell})$ as a subsequence of $\Bh$.
Thus $\wt f_{\eta_k}^{(\Bc_k)}$ is a linear combination of 
$f_{\b_{\ell-2}}^{(c_1')}f_{\b_{\ell-1}}^{(c_2')}f_{\b_{\ell}}^{(c_3')}$ for 
various $(c_1', c_2', c_3')$.
If $\eta_k$ is such that $\d_{\eta_k} = 1$, $\wt f_{\eta_k}^{(\Bc_k)}$ 
is the same as in 3.2. 
Hence $L\nat(\Bc, \Bh)$ can be written as 
$L\nat(\Bc,\Bh) = L(\Bc, \Bh) + z'$, where $z'$ is a $q\BZ[q]$-linear 
combination of PBW-bases $L(\Bc',\Bh)$ with $\Bc' > \Bc$, the lexicographic 
order on $\BN^{N}$ from the left. 
Here the transition matrix between $\SX_{\Bh}$ and $\SX\nat_{\Bh}$ is 
expressed as a triangular matrix 
where the diagonal entries are 1, and off diagonals are contained 
in $q\BZ[q]$.  It follows that $\SX_{\Bh}\nat$ gives an $\BA$-basis of 
${}_{\BA}\BU_q^-$, and an $\BQ(q)$-basis of $\BU_q^-$.  
We know already that $\s$ acts on $\SX_{\Bh}\nat$ as a permutation. 
From the discussion in 3.4, for $\eta_k$ such that $\d_{\eta_k} = 2$, 
$\wt f_{\eta_k}^{(\Bc_k)}$ is $\s$-invariant if and only if $b(\Bc_k, \Bh_{\eta_k})$
is $\s$-invariant. By Lemma 2.2, this condition is given by 
$c_{\Bk} = (c, c, c)$ for some $c \in \BN$.   
Hence the condition that $L\nat(\Bc, \Bh)$ is $\s$-invariant is given 
by $\Bc \in \BN^{N}$ obtained from $\ul\Bc$ as in (3.2.6).  The lemma is proved. 
\end{proof}

\para{3.6.}
$\SX\nat_{\Bh} = \{ L\nat(\Bc, \Bh) \}$ is called  the modified PBW-basis associated to
$\Bh$. Concerning the relation with the canonical basis, 
the modified PBW-basis satisfies a similar property as 
the PBW-basis. The proof is straightforward from the construction in 3.4.
Let $b(\Bc, \Bh)$ be the canonical basis corresponding to $L(\Bc, \Bh)$. Then 
\begin{equation*}
\tag{3.6.1}
b(\Bc, \Bh) = L\nat(\Bc, \Bh) + \sum_{\Bc < \Bd}a_{\Bd}L\nat(\Bd, \Bh),
\end{equation*}
where $a_{\Bd} \in q\BZ[q]$. 
Note that the canonical bases are characterized by (3.6.1) and 
by the property $\ol{b(\Bc, \Bh)} = b(\Bc, \Bh)$. 
\par 
We consider the map $\pi: {}_{\BA'}\BU_q^{-,\s} \to \BV_q$.
The image of $\bB^{\s}$ (resp. $(\SX_{\Bh}\nat)^{\s}$) 
on $\BV_q$ gives an $\BA'$-basis of $\BV_q$, which we denote by the same symbol.  
For $L(\ul\Bc, \ul\Bh) \in \ul\SX_{\ul\Bh}$, we denote by $b(\ul\Bc, \ul\Bh)$ 
the corresponding canonical basis in $\ul\BU_q^-$.  
\par
In the case where $X$ is of finite type, we can prove 
Theorem 2.6, assuming Proposition 2.8, in a similar way 
as in [SZ1]. 
The following result is a generalization of [SZ1, Thm. 1.14] to the non-admissible
case.

\begin{thm}  
Assume that $(X, \ul X) = (A_{2n}, C_n)$, i.e. $(F_n1)$ in 1.8. 
Let $L(\ul\Bc, \ul\Bh) \mapsto L\nat(\Bc, \Bh)$ be the bijection 
$\ul\SX_{\ul\Bh} \isom (\SX_{\Bh}\nat)^{\s}$ given in Lemma 3.5.
Let $\Phi: {}_{\BA'}\ul\BU_q^- \to \BV_q$ be the homomorphism 
given in Proposition 2.8.  Then 
\begin{enumerate}
\item \  The map $\Phi$ gives an isomorphism ${}_{\BA'}\ul\BU_q^- \isom \BV_q$. 
\item   \ $\Phi(L(\ul\Bc, \ul\Bh)) = L\nat(\Bc,\Bh)$.
\item \ $\Phi(b(\ul\Bc,\ul\Bh)) = b(\Bc, \Bh)$. 
\end{enumerate}
\end{thm}

\begin{proof}
We know already that $\Phi$ is injective.  The surjectivity of 
$\Phi$ will follow from (ii).  Thus (i) follows from  (ii). 
We prove (ii).
By (3.4.1) and (3.4.2), in order to prove (ii), it is enough to show that 
\begin{equation*}
\tag{3.7.1}
\Phi\bigl(\ul T_{\eta_1}\cdots \ul T_{\eta_{k-1}}(\ul f_{\eta_k})\bigr)
    = \pi\bigl(\wt T_{\eta_1}\cdots \wt T_{\eta_{k-1}}(\wt f_{\eta_k})\bigr),
\end{equation*}
where $\ul T_{\eta}$ denotes the braid group action on $\ul\BU_q$. 
By a similar argument as in the proof of Theorem 1.14 in [SZ1], 
the proof of (3.7.1) is reduced to the case where the rank of $\ul X$ is 2. 
Thus we nay assume that $X $ is of type $A_4$ and $\ul X $ is of type $C_2$.
We use the notation $I = \{ 1,1', 2, 2'\}$ with 
$\s : 1 \lra 1', 2 \lra 2'$, $\ul I = \{ \ul 1, \ul 2\}$. 
By considering the 
$*$-action on $\BU_q^-$ and $\ul\BU_q^-$, we may assume that 
$\ul\Bh = (\ul 1, \ul 2, \ul 1, \ul 2)$ (see [SZ1, 4.1]).  
It is enough to check the following equalities,
\par\medskip 
(a) \
$\Phi(\ul f_1) = \pi(\wt f_1)$, 
\par 
(b) \ 
$\Phi(\ul T_{\ul 1}(\ul f_{\ul 2})) = \pi(\wt T_1(\wt f_2))$, 
\par 
(c) \
$\Phi(\ul T_{\ul 1}\ul T_{\ul 2}(\ul f_{\ul 1})) = \pi(\wt T_1\wt T_2(\wt f_1))$, 
\par 
(d) \
$\phi(\ul T_{\ul 1}\ul T_{\ul 2}\ul T_{\ul 1}(\ul f_{\ul 2}))
                     = \pi(\wt T_1\wt T_2\wt T_1(\wt f_2))$.
\par\medskip  
In the discussion below, we use the notation in Section 4 
for the description of PBW-basis. 
Since $\ul T_{\ul 1}\ul T_{\ul 2}\ul T_1(\ul f_{\ul 2}) = \ul f_{\ul 2}$, 
 $\wt T_1\wt T_2 \wt T_1 (\wt f_2) = \wt  f_2$, (a) and (d) follows from 
the definition of $\Phi$. We verify (b) and (c).
First we show (c).  It is easy to see that 
$\ul T_{\ul 1}\ul T_{\ul 2}(\ul f_{\ul 1}) = 
       \ul f_2\ul f_{\ul 1} - q^4\ul f_{\ul 1}\ul f_{\ul 2}$. 
Hence we have
\begin{equation*}
\Phi\bigl(\ul T_{\ul 1}\ul T_{\ul 2}(\ul f_{\ul 1})\bigr) = 
      \pi(f_2f_{2'}^{(2)}f_2f_1f_{1'} - q^4f_1f_{1'}f_2f_{2'}^{(2)}f_2).
\end{equation*}
On the other hand, 
$\wt T_1\wt T_2 (\wt f_1) = T_1T_{1'}T_2T_{2'}T_2(f_1f_{1'}) = f_{1'22'}f_{122'}$. 
By a similar computation as in Section 4, one can check that 
\begin{equation*}
f_2f_{2'}^{(2)}f_2f_1f_{1'} \equiv q^4f_1f_{1'}f_2f_{2'}^{(2)}f_2 + f_{1'22'}f_{122'}
     \mod J.
\end{equation*} 
Hence (c) holds.  For (b), we have
$\ul T_{\ul 1}(\ul f_{\ul 2}) = \ul f_{\ul 2}\ul f_{\ul 1}^{(2)} 
  - q^2\ul f_{\ul 1}\ul f_{\ul 2}\ul f_{\ul 1} + q^4\ul f_1^{(2)}\ul f_{\ul 2}$.
Hence
\begin{align*}
\tag{3.7.2}
\Phi(\ul T_{\ul 1}(\ul f_{\ul 2})) = 
    \pi(f_2f_{2'}^{(2)}f_2f_1^{(2)}f_{1'}^{(2)} - q^2f_1f_{1'}f_2f_{2'}^{(2)}f_2f_1f_{1'}
        + q^4f_1^{(2)}f_{1'}^{(2)}f_2f_{2'}^{(2)}f_2).
\end{align*} 
On the other hand, we have $T_1T_{1'}(f_2f_{2'}^{(2)}f_2) = f_{12}f_{1'2'}^{(2)}f_{12}$.
By a direct computation, we have

\begin{align*}
f_2f_{2'}^{(2)}f_2f_1^{(2)}f_{1'}^{(2)} 
      &\equiv q^8f_1^{(2)}f_{1'}^{(2)}f_2f_{2'}^{(2)}f_2  + f_{12}f_{1'2'}^{(2)}f_{12}  
               + q^2f_1f_{1'}f_{1'22'}f_{122'}  \mod J, \\
f_1f_{1'}f_2f_{2'}^{(2)}f_2f_1f_{1'} 
      &\equiv f_1f_{1'}f_{1'22'}f_{122'} + q^4[2]^2f_1^{(2)}f_{1'}^{(2)}f_2f_{2'}^{(2)}f_2
          \mod J.    
\end{align*}
Substituting this into the right hand side of (3.7.2), we obtain the 
equality (b).  Thus (3.7.1) is verified, and (ii) follows.    
\par
Since the bar-involution is compatible with the map 
$\Phi$, (iii) follows from (ii) by using (3.6.1) and the corresponding formula for
$\ul\BU_q^-$.  The theorem is proved. 
\end{proof}

\par\bigskip
\section{ The proof of Proposition 2.8}

\para{4.1.}
We assume that $X$ is irreducible, $\s$ is not admissible, 
and that $(X, \ul X)$ satisfies the condition in Theorem 2.6. 
For the proof of Proposition 2.8, we need to show, 
for $\eta, \eta' \in\ul I$,  that 
\begin{align*}
\tag{4.1.1}
&\sum_{k=0}^{1 - a_{\eta\eta'}}(-1)^kg_{\eta}^{(k)}g_{\eta'}
                      g_{\eta}^{(1-a_{\eta\eta'} - k)} = 0, 
                   \quad (\eta \ne \eta'), \\
\tag{4.1.2}
&[a]^!_{|\eta|}g_{\eta}^{(a)} = g_{\eta}^a, \quad (a \in \BN). 
\end{align*}

(4.1.2) is known if $\d_{\eta} = 1$ (see [SZ1, 3.1]), 
and it was proved 
in Proposition 2.4 if $\d_{\eta} = 2$. 
If $\d_{\eta} = \d_{\eta'} = 1$, (4.1.1) was already verified in 
[SZ1, 3.2]. Thus from  the tables in 1.8, for the proof of (4.1.1), 
it is enough to consider the case where $X = A_4, \ul X = C_2$
(note that we have excluded the cases ($A_n2'$) $\sim$ ($A_n4'$) and ($A_n5$) in 1.8). 
We write $I = \{ 1, 2, 2', 1'\}$
with $\s : 1 \lra 1', 2\lra 2'$.  Set $\ul I = \{ \ul 1, \ul 2\}$ 
with $\d_{\ul 1} = 1, \d_{\ul 2} = 2$.  
We have $(\a_{\ul 1}, \a_{\ul 1})_1 = 4, (\a_{\ul 2}, \a_{\ul 2})_1 = 8$. 
The Cartan matrix $\ul A$ is given by 
\begin{equation*}
\ul A = \begin{pmatrix}
          2  &  -1   \\
         -2  &  2
        \end{pmatrix}.
\end{equation*}
Thus we have only to prove the following two formulas.
\begin{align*}
\tag{4.1.3}
&g_{\ul 1}g_{\ul 2}^{(2)} - g_{\ul 2}g_{\ul 1}g_{\ul 2} + g_{\ul 2}^{(2)}g_{\ul 1} = 0,  \\
\tag{4.1.4}
&g_{\ul 2}g_{\ul 1}^{(3)} - g_{\ul 1}g_{\ul 2}g_{\ul 1}^{(2)} 
             + g_{\ul 1}^{(2)}g_{\ul 2}g_{\ul 1} - g_{\ul 1}^{(3)}g_{\ul 2} = 0.
\end{align*}

\para{4.2.}
Let $W = S_5$ be the Weyl group of type $A_4$, 
and $w_0$ the longest element
in $W$.  Then $l(w_0) = 10$. We consider the following reduced sequence $\Bh$ of $I$.
\begin{equation*}
\Bh = (i_1, i_2, \dots, i_{10}) = (1, 1', 2, 2', 2, 1, 1', 2, 2', 2) .
\end{equation*} 
Correspondingly, we have a total order of $\vD^+$ as follows;
\begin{equation*}
\vD^+ = \{ \b_1, \dots, \b_{10} \} = 
        \{ 1, 1', 12, 11'22', 1'2', 1'22', 122', 2, 22', 2'\},  
\end{equation*}
where $\b_k = s_{i_1}s_{i_2}\cdots s_{i_{k-1}}(\a_{i_k})$.
We define $f_{\b_k}^{(c)} = T_{i_1}\cdots T_{i_{k-1}}(f_{i_k}^{(c)})$, 
and set $L(\Bc, \Bh) = f_{\b_1}^{(c_1)}\cdots f_{\b_{10}}^{(c_{10})}$ 
for $\Bc = (c_1, \dots, c_{10}) \in \BN^{10}$. 
Then $\{ L(\Bc, \Bh) \mid \Bc \in \BN^{10}\}$ gives a PBW-basis of $\BU_q^-$.  
$f_{\b_1}, \dots, f_{\b_{10}}$ are written explicitly as follows;

\begin{equation*}
\begin{aligned}
f_{\b_1} &= f_1, \qquad f_{\b_2} = f_{1'}, \qquad f_{\b_{10}} = f_{2'}, \qquad
f_{\b_8} = f_2,  \qquad 
f_{\b_3} = f_2f_1 - qf_1f_2, \\
f_{\b_5} &= f_{2'}f_{1'} - qf_{1'}f_{2'}, \qquad f_{\b_9} = f_{2'}f_2 - qf_2f_{2'}, 
  \qquad f_{\b_7} = f_{22'}f_1 - qf_1f_{22'},  \\
f_{\b_6} &= f_2f_{1'2'} - qf_{1'2'}f_2, \qquad f_{\b_4} = f_{1'2'}f_{12} - qf_{12}f_{1'2'}.  
\end{aligned}
\end{equation*}

(4.1.3) and (4.1.4) are equivalent to the formulas
\begin{align*}
\tag{4.2.1}
&(f_1f_{1'})f_2^{(2)}f_{2'}^{(4)}f_2^{(2)} - f_2f_{2'}^{(2)}f_2(f_1f_{1'})f_2f_{2'}^{(2)}f_2
+ f_2^{(2)}f_{2'}^{(4)}f_2^{(2)}(f_1f_{1'}) \equiv 0  \mod J, \\
\tag{4.2.2}
&(f_2f_{2'}^{(2)}f_2)f_1^{(3)}f_{1'}^{(3)} - f_1f_{1'}(f_2f_{2'}^{(2)}f_2)f^{(2)}_1f^{(2)}_{1'}  \\
     &\qquad\qquad + f_1^{(2)}f_{1'}^{(2)}(f_2f_{2'}^{(2)}f_2)f_1f_{1'} 
     - f_1^{(3)}f_{1'}^{(3)}(f_2f_{2'}^{(2)}f_2) \equiv 0 \mod J.
\end{align*}

We shall prove (4.2.1) and (4.2.2) by expressing them in terms of
modified PBW-bases.  We compute them by making use of various commutation relations 
for $f_{\b_k}$.

\para{4.3.}
We prove (4.2.2).  
First we compute

\begin{align*}
f_1^{(2)}&f_{1'}^{(2)}(f_2f_{2'}^{(2)}f_2) f_1f_{1'}  \\
    &= q^2[3]\bigl(f_1^{(2)}f_{1'}^{(3)}f_{122'}f_2f_{2'} + 
                   f_1^{(3)}f_{1'}^{(2)}f_{1'22'}f_{2'}f_2\bigr)  \\
    &\quad + q\bigl(f_1^{(2)}f_{1'}^{(2)}f_{12}f_{1'22'}f_{2'}
                  + f_1^{(2)}f_{1'}^{(2)}f_{1'2'}f_{122'}f_2\bigr)  \\
    &\quad + q\bigl(f_1^{(2)}f_{1'}^{(2)}f_{1'2'}f_{12}f_2f_{2'} 
                  + f_1^{(2)}f_{1'}^{(2)}f_{12}f_{1'2'}f_{2'}f_2\bigr)  \\
    &\quad + q^3[3]\bigl(f_1^{(2)}f_{1'}^{(3)}f_{12}f_2f_{2'}^{(2)}
                  + f_1^{(2)}f_{1'}^{(3)}f_{12}f_{2'}^{(2)}f_2 
                  + f_1^{(3)}f_{1'}^{(2)}f_{1'2'}f_2f_{2'}f_2\bigr)  \\
    &\quad + q^4[3]^2f_1^{(3)}f_{1'}^{(3)}f_2f_{2'}f_2 
             + f_1^{(2)}f_{1'}^{(2)}f_{1'22'}f_{122'}.
\end{align*}
The first three terms in the right hand side are of the form 
$x + \s(x)$, hence contained in $J$. 
Using the relation 
$f_2^{(2)}f_{2'} - f_2f_{2'}f_2 + f_{2'}f_2^{(2)} = 0$, 
the fourth term can be rewritten as 
$q^3[3](f_1^{(2)}f_{1'}^{(3)}f_{12}f_{2'}f_{2}f_{2'} + 
      f_1^{(3)}f_{1'}^{(2)}f_{1'2'}f_2f_{2'}f_2)$, hence again contained in $J$.
It follows that 
\begin{equation*}
\tag{4.3.1}
\begin{split}
f_1^{(2)}f_{1'}^{(2)}&(f_2f_{2'}^{(2)}f_2) f_1f_{1'}  \\  
     &\equiv q^4[3]^2f_1^{(3)}f_{1'}^{(3)}f_2f_{2'}f_2 
             + f_1^{(2)}f_{1'}^{(2)}f_{1'22'}f_{122'} \mod J.
\end{split}
\end{equation*}

Next we compute 

\begin{align*}
f_2f_{2'}^{(2)}&f_2f_1^{(3)}f_{1'}^{(3)}  \\
   &= q^6(f_1f_{1'}^{(3)}f_{12}f_{122'}f_{2'} + f_1^{(3)}f_{1'}f_{1'2'}f_{1'22'}f_2) \\
   &\quad + q^8[2](f_1f_{1'}^{(3)}f_{12}^{(2)}f_{2'}^{(2)} 
                      + f_1^{(3)}f_{1'}f_{1'2'}^{(2)}f_2^{(2)}) \\
   &\quad + q^8(f_1^{(2)}f_{1'}^{(3)}f_{122'}f_2f_{2'} 
                      + f_1^{(3)}f_{1'}^{(2)}f_{1'22'}f_{2'}f_2)  \\
   &\quad + q^5(f_1^{(2)}f_{1'}^{(2)}f_{12}f_{1'2'}f_{2'}f_2 
                 + f_1^{(2)}f_{1'}^{(2)}f_{1'2'}f_{12}f_2f_{2'})  \\
   &\quad + q^2(f_1f_{1'}^{(2)}f_{12}f_{1'2'}f_{122'}
                 + f_1^{(2)}f_{1'}f_{1'2'}f_{12}f_{1'22'}) \\
   &\quad + q^5(f_1^{(2)}f_{1'}^{(2)}f_{1'2'}f_{122'}f_2 
                 + f_1^{(2)}f_{1'}^{(2)}f_{12}f_{1'22'}f_{2'})  \\
   &\quad + q^9(f_1^{(2)}f_{1'}^{(3)}f_{12}f_2f_{2'}^{(2)}
                 + f_1^{(2)}f_{1'}^{(3)}f_{12}f_{2'}^{(2)}f_2 
                 + f_1^{(3)}f_{1'}^{(2)}f_{1'2'}f_2f_{2'}f_2) \\
   &\quad + q^3(f_1^{(2)}f_{1'}f_{1'2'}^{(2)}f_{12}f_2 
                  + f_1^{(2)}f_{1'}f_{12}f_{1'2'}^{(2)}f_2  
                  + f_1f_{1'}^{(2)}f_{12}f_{1'2'}f_{12}f_{2'})  \\
   &\quad + q^{12}f_1^{(3)}f_{1'}^{(3)}f_2f_{2'}^{(2)}f_2
          + q^4f_1^{(2)}f_{1'}^{(2)}f_{1'22'}f_{122'} 
          + f_1f_{1'}f_{12}f_{1'2'}^{(2)}f_{12}. 
\end{align*}
The first 6 terms in the right hand side 
are of the form $x + \s(x)$, hence are contained in $J$. 
The 7th term is contained in $J$ by the same reason as the previous computation.
By applying $T_1T_{1'}$ to the relation 
$f_2f_{2'}f_2 = f_{2'}f_2^{(2)} + f_2^{(2)}f_{2'}$, we obtain 
$f_{1'2'}f_{12}f_{1'2'} = f_{12}f_{1'2'}^{(2)} + f_{1'2'}^{(2)}f_{12}$.
Thus, by a similar reason, the 8th term is also contained in $J$.  
It follows that 
\begin{equation*}
\tag{4.3.2}
\begin{split}
f_2&f_{2'}^{(2)}f_2f_1^{(3)}f_{1'}^{(3)}  \\
    &\equiv q^{12}f_1^{(3)}f_{1'}^{(3)}f_2f_{2'}^{(2)}f_2
          + q^4f_1^{(2)}f_{1'}^{(2)}f_{1'22'}f_{122'} 
          + f_1f_{1'}f_{12}f_{1'2'}^{(2)}f_{12}   \mod J. 
\end{split}
\end{equation*} 

Finally we compute 

\begin{align*}
f_1f_{1'}(&f_2f_{2'}^{(2)}f_2)f_1^{(2)}f_{1'}^{(2)}  \\
   &= q^3[2]^2(f_1^{(2)}f_{1'}^{(2)}f_{12}f_{1'22'}f_{2'}
          + f_1^{(2)}f_{1'}^{(2)}f_{1'2'}f_{122'}f_2) \\
   &\quad + q^5[2][3](f_1^{(2)}f_{1'}^{(3)}f_{122'}f_2f_{2'}
               + f_1^{(3)}f_{1'}^{(2)}f_{1'22'}f_{2'}f_2)  \\
   &\quad + q^3[2]^2(f_1^{(2)}f_{1'}^{(2)}f_{12}f_{1'2'}f_{2'}f_2
               + f_1^{(2)}f_{1'}^{(2)}f_{1'2'}f_{12}f_2f_{2'})  \\
   &\quad + q^4[3](f_1f_{1'}^{(3)}f_{12}f_{122'}f_{2'}
               + f_1^{(3)}f_{1'}f_{1'2'}f_{1'22'}f_2)  \\
   &\quad + q^6[2][3](f_1f_{1'}^{(3)}f_{12}^{(2)}f_{2'}^{(2)}
               + f_1^{(3)}f_{1'}f_{1'2'}^{(2)}f_2^{(2)}) \\
   &\quad + q[2](f_1f_{1'}^{(2)}f_{12}f_{1'2'}f_{122'} 
               + f_1^{(2)}f_{1'}f_{1'2'}f_{12}f_{1'22'}) \\
   &\quad + q^6[2][3](f_1^{(2)}f_{1'}^{(3)}f_{12}f_{2'}^{(2)}f_2 
               + f_1^{(2)}f_{1'}^{(3)}f_{12}f_2f_{2'}^{(2)}
               + f_1^{(3)}f_{1'}^{(2)}f_{1'2'}f_2f_{2'}f_2)  \\
   &\quad + q^2[2](f_1^{(2)}f_{1'}f_{12}f_{1'2'}^{(2)}f_2 
               + f_1^{(2)}f_{1'}f_{1'2'}^{(2)}f_{12}f_2
               + f_1f_{1'}^{(2)}f_{12}f_{1'2'}f_{12}f_{2'}) \\
   &\quad + q^8[3]^2f_1^{(3)}f_{1'}^{(3)}f_2f_{2'}^{(2)}f_2
               + q^2[2]^2f_1^{(2)}f_{1'}^{(2)}f_{1'22'}f_{122'}
               + f_1f_{1'}f_{12}f_{1'2'}^{(2)}f_{12}.  
\end{align*}
The first 6 terms in the right hand side are of the form $x + \s(x)$, hence 
are contained in $J$. The 7th and 8th terms are contained in $J$ by the 
same reason as before. It follows that

\begin{equation*}
\tag{4.3.3}
\begin{split}
f_1f_{1'}(f_2f_{2'}^{(2)}&f_2)f_1^{(2)}f_{1'}^{(2)}  
  \equiv q^8[3]^2f_1^{(3)}f_{1'}^{(3)}f_2f_{2'}^{(2)}f_2  \\
               &+ q^2[2]^2f_1^{(2)}f_{1'}^{(2)}f_{1'22'}f_{122'}
               + f_1f_{1'}f_{12}f_{1'2'}^{(2)}f_{12} \mod J.
\end{split}
\end{equation*}  

Now putting together (4.3.1), (4.3.2) and (4.3.3), we see that 
the left hand side of (4.2.2) is equal to

\begin{equation*}
- F_3(q^2)f_1^{(3)}f_{1'}^{(3)}f_2f_{2'}^{(2)}f_2 
+ F_2(q^2)f_1^{(2)}f_{1'}^{(2)}f_{1'22'}f_{122'} 
  - F_1(q^2)f_1f_{1'}f_{12}f_{1'2'}^{(2)}f_{12}  \mod J,
\end{equation*}
where, for $a \ge 1$,  $F_a(q)$ is a polynomial in $q$ defined by 
\begin{equation*}
F_a(q) = \sum_{t = 0}^a(-1)^tq^{t(a-t)}
           \begin{bmatrix}
               a  \\
               t
           \end{bmatrix}.
\end{equation*}
Since $F_a(q)$ is identically zero by [L1, 1.3.4 (a)], 
(4.2.2) holds. 

\para{4.4}
We prove (4.2.1). 
We compute 
\begin{align*}
f_2f_{2'}^{(2)}f_2f_1f_{1'}
      &= q^2(f_1f_{1'22'}f_{2'}f_2 + f_{1'}f_{122'}f_2f_{2'}) \\
      &\quad + q(f_{12}f_{1'2'}f_{2'}f_2 + f_{1'2'}f_{12}f_2f_{2'}) \\
      &\quad + q(f_{12}f_{1'22'}f_{2'}  + f_{1'2'}f_{122'}f_2) \\
      &\quad + q^3(f_{1'}f_{12}f_2f_{2'}^{(2)} 
                + f_{1'}f_{12}f_{2'}^{(2)}f_2
                + f_1f_{1'2'}f_2f_{2'}f_2)  \\
      &\quad + q^4f_1f_{1'}f_2f_{2'}^{(2)}f_2 + f_{1'22'}f_{122'}.
\end{align*}
(Actually, this is a part of the computation of 
$f_1^{(2)}f_{1'}^{(2)}(f_2f_{2'}^{(2)}f_2)f_1f_{1'}$
in 4.3.)  Since the first three terms are written as $x + \s(x)$, they are
contained in $J$. The 4th term is contained in $J$ by a similar discussion 
as in 4.3.  It follows that 

\begin{equation*}
\tag{4.4.1}
f_2f_{2'}^{(2)}f_2f_1f_{1'} \equiv q^4f_1f_{1'}f_2f_{2'}^{(2)}f_2 + f_{1'22'}f_{122'}.
\end{equation*}
On the other hand, we have

\begin{align*}
f_2^{(2)}f_{2'}^{(4)}f_2^{(2)}f_1f_{1'}  
  &= q^3(f_{1'2'}f_{12}f_2^{(2)}f_{2'}^{(3)}f_2 
           + f_{12}f_{1'2'}f_2f_{2'}^{(3)}f_2^{(2)})  \\
  &\quad + q^4(f_1f_{1'22'}f_2f_{2'}^{(3)}f_2^{(2)}
           + f_{1'}f_{122'}f_2^{(2)}f_{2'}^{(3)}f_2)   \\
  &\quad + q^2(f_{12}f_{1'22'}f_2f_{2'}^{(3)}f_2
           + f_{12}f_{1'22'}f_{2'}^{(3)}f_2^{(2)}
           + f_{1'2'}f_{122'}f_2^{(2)}f_{2'}^{(2)}f_2)  \\
  &\quad + q^6(f_1f_{1'2'}f_2^{(2)}f_{2'}^{(3)}f_2^{(2)}
           + f_{1'}f_{12}f_2^{(2)}f_{2'}^{(4)}f_2
           + f_{1'}f_{12}f_2f_{2'}^{(4)}f_2^{(2)})  \\
  &\quad + q^8f_1f_{1'}f_2^{(2)}f_{2'}^{(4)}f_2^{(2)}  
           + f_{1'22'}f_{122'}f_2f_{2'}^{(2)}f_2.  
\end{align*}
The first two terms are written as  $x + \s(x)$, hence they are contained 
in $J$.  For the third term and the 4th term, we note the relations
\begin{align*}
\tag{4.4.2}
f_2f_{2'}^{(3)}f_2 + f_{2'}^{(3)}f_2 = f_{2'}^{(2)}f_2^{(2)}f_{2'},  \\
\tag{4.4.3}
f_2f_{2'}^{(4)}f_2^{(2)} + f_2^{(2)}f_{2'}^{(4)}f_2 
                          = f_{2'}^{(2)}f_2^{(3)}f_{2'}^{(2)}. 
\end{align*}
(4.4.2) and (4.4.3) are verified by using (2.1.2) and (2.1.3). 
Hence the third and the 4th terms are also contained in $J$.  It follows that
\begin{equation*}
\tag{4.4.4}
f_2^{(2)}f_{2'}^{(4)}f_2^{(2)}f_1f_{1'}  
    \equiv q^8f_1f_{1'}f_2^{(2)}f_{2'}^{(4)}f_2^{(2)}  
           + f_{1'22'}f_{122'}f_2f_{2'}^{(2)}f_2  \mod J.
\end{equation*}

By multiplying $f_2f_{2'}^{(2)}f_2$ on (4.4.1) from the right, we have
\begin{align*}
f_2f_{2'}^{(2)}f_2(f_1f_1')f_2f_{2'}^{(2)}f_2
&\equiv q^4f_1f_{1'}(f_2f_{2'}^{(2)}f_2)^2 + f_{1'22'}f_{122'}f_2f_{2'}^{(2)}f_2 \\
&\equiv q^4[2]^4f_1f_{1'}f_2^{(2)}f_{2'}^{(4)}f_2^{(2)} 
            + f_{1'22'}f_{122'}f_2f_{2'}^{(2)}f_2 \mod J. 
\end{align*}
The second identity follows from (2.4.2) for $a = 2$. 
Combining this formula with (4.4.4), we see that the left hand side 
of (4.2.1) is equal to 
\begin{equation*}
F_2(q^4)f_1f_{1'}f_2^{(2)}f_{2'}^{(4)}f_2^{(2)} + 
          F_1(q^4)f_{1'22'}f_{122'}f_2f_{2'}^{(2)}f_2 \equiv  0 \mod J.
\end{equation*}
Hence (4.2.1) holds.  Proposition 2.8 is proved.  

\par\bigskip

\section{ $\s$-action on root systems }

\para{5.1.}
In this section, we assume that $X$ is irreducible of affine type, 
and consider the pair $(X, \ul X)$ for $\s$ : admissible 
or non-admissible (but exclude the case ($A_n5$)). 
Let $\Fg$ be the affine Kac-Moody algebra associated to $X$.
We write $I = \{ 0, 1, \dots, n\}$, and set $I_0 = I - \{ 0\}$, where 
$I_0$ corresponds to the subalgebra $\Fg_0$ of $\Fg$ of finite type.
Recall that $Q = \bigoplus_{i \in I}\BZ \a_i$ and let $\vD \subset Q$ be the affine root 
system for $\Fg$.
Since $\s : I \to I$, $\s$ acts on $Q$ by $\a_i \mapsto \a_{\s(i)}$. 
$\vD$ is decomposed as $\vD  = \vD^{\re} \sqcup \vD^{\im}$, where 
$\vD^{\re}$ is the set of real roots, and $\vD^{\im} = (\BZ-\{0\})\d$ is the 
set of imaginary roots.
$\d$ is written as $\d = \sum_{i \in I}a_i\a_i$. From the explicit table
of coefficients $a_i$ (see [K]), we see $a_i$ is symmetric with respect to $\s$.
Thus $\d$ is stable by $\s$. On the other hand, $\vD^{\re} = W\Pi$, where 
$\Pi = \{ \a_i \mid i \in I\}$ is the set of simple roots.  
Since $\Pi$ is $\s$-stable, and $W$ is $\s$-stable, 
we see that $\vD^{\re}$ is $\s$-stable.  Hence the action of $\s$ on $Q$ leaves
$\vD$ invariant.   
Let $\vD_0$ be the subsystem of $\vD$ corresponding to $I_0$. 

\para{5.2.}
Let $\eta_0 \subset I$ be the $\s$-orbit containing 0. Set $I_0' = I - \eta_0$
and let $\vD_0'$ be the subsystem of $\vD_0$ obtained from $I_0'$.  Then 
$I_0'$ and $\vD_0'$ are $\s$-stable. 
We have $\ul I = \ul I_0' \sqcup \{\eta_0\}$, and $\ul I_0'$ corresponds to
the Cartan datum $\ul X_0$ of finite type. 
\par
For each $\b \in \vD$,  let $\SO_{\b}$ be the $\s$-orbit of $\b$ in $\vD$.
If $\SO_{\b} = \{ \b, \s(\b)\}$ with $\b + \s(\b) \in \vD^{\re}$, 
set $O(\b) = O(\s(\b)) = O(\b + \s(\b)) = 2(\b + \s(\b))$
(this only occurs in the case where $\s$ is non-admissible). 
Otherwise, 
set $O(\b) = \sum_{\g \in \SO_{\b}}\g$.
Let $Q^{\s}$ be the set of $\s$-fixed elements in $Q$, and  
$(Q^{\s})'$ the $\BZ$-submodule of $Q^{\s}$ spanned by $\wt\a_{\eta}$ 
with $\eta \in \ul I$, where 
$\wt\a_{\eta} = O(\a_i)$ for $i \in \eta$.  
Recall that $\ul Q = \bigoplus_{\eta \in \ul I}\BZ\ul\a_{\eta}$. 
The map $\wt\a_{\eta} \mapsto \ul\a_{\eta}$ gives an isomorphism 
$(Q^{\s})' \isom \ul Q$.  This map is compatible with the inner product
$(\ ,\ )$ on $Q$ and the inner product $(\ ,\ )_1$ on $\ul Q$ defined in (1.6.1). 
\par
Set $\ul\vD_0 = \{ O(\b) \mid \b \in \vD_0'\}$,  
which we consider as a subset of $\ul Q$ under the identification 
$(Q^{\s})' \simeq \ul Q$. Then it is easy to see that 
$\ul\vD_0$ gives a root system of type $\ul X_0$ with 
simple system $\{ O_{\eta} \mid \eta \in \ul I_0'\}$, where 
$O_{\eta} = O(\a_i)$ for $i \in \eta$. 

\para{5.3.}
Hereafter we assume that $X$ is irreducible and simply-laced, hence it is of
untwisted type. 
Set $\vS_0^+ = \{ \a \in \vD_0^+ \mid O(\a) \in \BZ \d \}$. 
We shall determine the set $\vS_0^+$ for each $(X, \ul X)$. 
If $\s$ leaves $\vD_0$ invariant, namely, $\eta_0 = \{ 0\}$, 
then $\vD_0 = \vD_0'$, and in this case, 
$\vS_0^+ = \emptyset$.  Thus we assume that $\vD_0' \ne \vD_0$.  In the case
where $\s$ is admissible, and $\ve \ne 4$, $\vS_0^+$ is already determined 
in [SZ2]. In the remaining cases $(X, \ul X)$ are given as follows;

\par\medskip
(A) \ $(D_{2n}^{(1)}, A_{2n -2}^{(2)}), 
            \qquad \!\!(n \ge 2, \text{ $\s$ : admissible, $\ve = 4$})$,   
\par
(B) \ $(D_{2n+1}^{(1)}, C_{n-1}^{(1)}), 
            \quad \ (n \ge 2, \text{ $\s$ : non-admissible, $\ve = 4$, 
                          $\ul X = A_1^{(1)}$ if $n = 2$})$,  
\par
(C) \ $(D_{2n+1}^{(1)}, A_{2n-1}^{(2)}),
            \quad  (n \ge 2, \text{ $\s$ : non-admissible, $\ve = 2$})$,  
\par
(D) \ $(A_{2n+1}^{(1)}, C_n^{(1)}), 
            \qquad \!(n \ge 1, \text{ $\s$ : non-admissible, $\ve = 2$, 
$\ul X = A_1^{(1)}$ if $n = 1$})$. 

\para{5.4.}
First consider the case (A), (B).  Let $X = D_m^{(1)}$, where
$m = 2n$ in the case (A), and $m = 2n+1$ in the case (B). 
Set $I = \{ 0,1, \dots, m\}$, then $\s : I \to I$ is given by  
$\s: 0 \mapsto m - 1 \mapsto 1 \mapsto m \mapsto 0$, 
$i \lra m  - i$ for $2 \le i \le m -2$. $\eta_0 = \{ 0,1,  m - 1, m\}$
and $I'_0 = \{2, \dots, m-2\}$.  $\vD_0$ is of type $D_m$, and 
$\vD_0'$ is of type $A_{m - 3}$. The highest root $\th \in \vD^+_0$ 
is given by $\th = 2\a_1 + \cdots + 2\a_{m -2} + \a_{m -1} + \a_m$. 
Write $\a = \sum_{1 \le i \le m}c_i\a_i$, and 
$z = \sum_{2 \le i \le m-2}c_i\a_i$.  We have
\begin{align*}
\tag{5.4.1}
\s(\a) &= c_1\a_m + c_{m - 1}\a_1 + c_m(\d - \th) + \s(z) \\
\tag{5.4.2}
\s^2(\a) &= c_1(\d - \th) + c_{m -1}\a_m + c_m\a_{m -1} + z,    
\end{align*}   
since $\s^2$ acts trivially on $\vD_0'$. 
\par
For $i = 1,2,4$, set $Z_i = \{ \a \in \vD_0 \mid |\SO_{\a}| = i\}$.
First assume that $\a \in Z_1$.  Then $c_1 = c_m = 0$, and 
so $\a = z \in \vD_0'$. If we write  
$I_0' = \{ \ve_2 - \ve_3, \dots, \ve_{m - 2} - \ve_{m -1}\}$, then  
\begin{equation*}
Z_1 = \begin{cases}
        \{ \ve_2 - \ve_{m -1}, \ve_3 - \ve_{m -2}, \dots, \ve_{n} - \ve_{n+1} \} 
               &\quad\text{if $m = 2n$,} \\
       \{  \ve_2 - \ve_{m -1}, \ve_3 - \ve_{m -2}, \dots, \ve_n - \ve_{n+2} \}
               &\quad\txt{if  $m = 2n+1$,}
      \end{cases}   
\end{equation*} 
Note that $Z_1 \cap \vS_0^+ = \emptyset$.
\par
Next assume that $\a \in Z_2$. 
From (5.4.2), the condition $\a \in Z_2$ is given by $c_1 = 0$ and 
$c_m = c_{m-1}$. If we write  $\vD_0 = \{ \ve_i \pm \ve_j \mid 1 \le i < j \le m\}$, 
this condition implies that 
\begin{equation*}
Z_2 = \{ \ve_i \pm \ve_j \mid 2 \le i < j \le m - 1\} - Z_1. 
\end{equation*}
The action of $\s$ on $Z_2$ is given as follows;
\begin{equation*}
\tag{5.4.3}
\begin{aligned}
\ve_i - \ve_j &\lra \ve_{m +1-j} - \ve_{m +1-i}, \\
\ve_i + \ve_j &\lra \d - (\ve_{m +1-j} + \ve_{m + 1 -i}). 
\end{aligned}
\end{equation*}
From this, we see that 

\begin{equation*}
\tag{5.4.4}
Z_2 \cap \vS_0^+ = \begin{cases}
      \{ \ve_2 + \ve_{m -1}, \ve_3 + \ve_{m -2}, \dots, \ve_{n} + \ve_{n+1} \}
          &\quad\txt{ if $m = 2n$, } \\
      \{ \ve_2 + \ve_{m -1}, \ve_3 + \ve_{m -2}, \dots, \ve_n + \ve_{n+2} \}
          &\quad\text{ if $m = 2n+1$.}
                   \end{cases}
\end{equation*}

\par
Next consider $Z_4$.
From the above discussion, we have
\begin{equation*} 
\tag{5.4.5}
Z_4 = \{ \ve_i \pm \ve_j \mid i = 1 \text{ or } j = m \}. 
\end{equation*}
We shall determine $Z_4 \cap \vD_0^+$.  
We consider the formula (5.4.1). Since $\a \in Z_4$, we have 
$c_1 = 1$ or $(c_{m}, c_{m -1})  = (1, 0), (0,1)$. 
In the expression of (5.4.1), we only consider the first part, 
$x = c_1\a_1 + c_{m -1}\a_{m -1}+ c_{m}\a_{m}$. 
Then 
\begin{equation*}
x + \s(x) + \s^2(x) + \s^3(x) = (c_1 + c_{m -1} + c_{m})(\d - \th) + z',
\end{equation*}
where $z'$ is a sum of elements in $\vD_0$. 
\par
Hence if $3\d$ appears in this sum, then $(c_1, c_{m -1}, c_{m}) = (1,1,1)$.
If $2\d$ appears, then $(c_1, c_{m -1}, c_{m}) = (1,1,0)$ or $(1,0,1)$. 
If $\d$ appears, then $(c_1, c_{m -1}, c_{m}) = (1, 0, 0), (0, 1, 0), (0, 0, 1)$.  
Write $\a = x + z$, where $x$ and $z$ are as above.  
Then we have
\begin{equation*}
O(\a) = (c_1 + c_{m -1} + c_{m})O(\a_0) + 2(z + \s(z)).
\end{equation*}
Note that $O(\a_0) = \d - 2(\ve_2 - \ve_{m -1})$.
Set $\vS^+_0(k) = \{ \a \in \vD_0^+ \mid O(\a) = k\d\}$.
The condition that $\a \in \vS_0^+(k)$ is given by 
\begin{equation*}
\tag{5.4.6}
z + \s(z) = k(\ve_2 - \ve_{m -1}) = k(\sum_{2 \le i \le m -2}\a_i). 
\end{equation*}
Since 
$z + \s(z) = \sum_{2 \le i \le m -2}(c_i + c_{m -i})\a_i$, 
(5.4.6) implies a relation 
\begin{equation*}
\tag{5.4.7}
c_i + c_{m -i} = k, \qquad  (2 \le i \le m -2). 
\end{equation*}
\par
Note that in the case where $m = 2n$, $k = 2c_n$.  
Since $1\le k \le 3$, we must have $k = 2$. 
Thus $(c_1, c_{2n-1}, c_{2n}) = (1,1,0)$ or $(1,0,1)$. 
Summing up the above discussion, we have the following.
\par\medskip
Case (A) : $(X,\ul X) = (D_{2n}^{(1)}, A_{2n-2}^{(2)})$.
\begin{align*}
\tag{5.4.8}
Z_4 \cap \vD_0^+ = \{ \ve_1 - \ve_{2n}, \ve_1 + \ve_{2n}\}.
\end{align*}

Combining (5.4.8) with (5.4.4), we have
\begin{equation*}
\tag{5.4.9}
\vS_0^+ = \{ \ve_1 - \ve_{2n}, \ve_1 + \ve_{2n}, 
             \ve_2 + \ve_{2n-1}, \ve_3 + \ve_{2n-2}, \dots, \ve_n + \ve_{n+1} \}.
\end{equation*}
Hence $\vS_0^+$ is a positive subsystem of $\vD_0^+$ of type $(n+1)A_1$.

\par\medskip
Case (B) : $(X, \ul X) = (D_{2n+1}^{(1)}, C_n^{(1)})$
\par\medskip
By using (5.4.7), one can determine $Z_4 \cap \vS_0^+(k)$ 
as follows.
\begin{align*}
\tag{5.4.10}
Z_4 \cap \vD_0^+(3) &= \{ \ve_1 + \ve_{n+1} \}, \\
Z_4 \cap \vD_0^+(2) &= \{ \ve_1 - \ve_{2n+1}, \ve_1 + \ve_{2n+1}\}, \\
Z_4 \cap \vD_0^+(1) &= \{ \ve_1 - \ve_{n+1}, \ve_{n+1} - \ve_{2n+1}, 
                           \ve_{n+1} + \ve_{2n+1} \}.
\end{align*}
It follows that 
\begin{equation*}
\tag{5.4.11}
\begin{split}
\vS_0^+ &= \{ \ve_2 + \ve_{2n}, \ve_3 + \ve_{2n-1}, \dots, \ve_n + \ve_{n+2} \} \\
        &\quad  \sqcup
        \{ \ve_1 - \ve_{n+1}, \ve_{n+1} - \ve_{2n+1},  \ve_{n+1} + \ve_{2n+1}, 
                \ve_1 - \ve_{2n+1}, \ve_1 + \ve_{2n+1}, \ve_1 + \ve_{n+1} \}.
 \end{split}
\end{equation*}
Hence $\vD_0^+$ is a positive subsystem of $\vD_0^+$ of type $(n-1)A_1 + A_3$. 

\para{5.5.}
Next consider the case (C) : $(X, \ul X) = (D_{2n+1}^{(1)}, A_{2n-1}^{(2)})$. 
$\vD_0$ : type $D_{2n+1}$, where 
$\s : \a_0 \lra \a_{2n+1}, \a_i \lra \a_{2n + 1-i}$ for $i = 2, \dots, 2n-1$.   
In this case, $\eta_0 = \{ 0, 2n+1\}$, $I_0' = \{ 1,2, \dots, 2n\}$, 
and $\vD_0'$ is of type $A_{2n}$.
Here $\th = \a_1 + 2\a_2 + \cdots + 2\a_{2n-1} + \a_{2n} + \a_{2n+1}$. 
We write $\a \in \vD_0^+$ as 
$\a = \sum_{i \in I_0}c_i\a_i$, with $\eta_0 = \{ 0, i_1\}$.  
By the discussion in [SZ2, (5.3.1)], we see that 
$\vS_0^+ = \{ \a \in \vD_0^+ \mid O(\a) = \d\}$, and the condition 
$\a \in \vS_0^+$ is given by 
\begin{equation*}
\tag{5.5.1}
c_i + c_{\s(i)} = a_i,  \qquad (i \in I_0')
\end{equation*}
where we write $\d = \sum_{i \in I}a_i\a_i$. 
Here $\ul\vD_0$ is of type $C_n$ with $\a_{\ul n}$ long root. 
We have $c_{2n+1} = 1$ and 
\begin{equation*}
c_i + c_{\s(i)} = 2, \quad (2 \le i \le 2n-1), \qquad c_1 + c_{2n} = 1.
\end{equation*}
This implies that
\begin{equation*}
\tag{5.5.2}
\vS_0^+ = \{ \ve_1 + \ve_{2n+1}, \ve_2 + \ve_{2n}, \dots, \ve_n + \ve_{n+2} \}.
\end{equation*}
\par
Hence $\vS_0^+$ is a positive subsystem of $\vD^+_0$ of type $nA_1$. 

\para{5.6.}
Finally, consider the case (D) : $(X, \ul X) = (A_{2n+1}^{(1)}, C_n^{(1)})$. 
Here $I = \{ 0, 1, \dots, 2n+1\}$ with $\s : i \lra 2n+1-i$. 
Hence $\ul I = \{ \ul 0, \ul 1, \dots, \ul n\}$. 
$I_0' = \{ 1, 2, \dots, 2n\}$ and $\vD_0'$ is of type $A_{2n}$ on which 
$\s$ acts. $\ul I_0' = \{ \ul 1, \ul 2, \dots, \ul n\}$, and $\ul\vD_0'$ 
is of type $C_n$ with $\a_{\ul n}$ long root. 
Take $\a \in \vD_0^+$ and write it as $\a = \sum_{i \in I_0}c_i\a_i$. 
Here $\th = \a_1 + \cdots + \a_{2n+1}$, and $\s(\a_{2n+1}) = \d - \th$.
Assume that $\a \in \vS_0^+$. 
We have 
\begin{align*}
\a + \s(\a) = c_{2n+1}(\d - \th + \a_{2n+1}) + 
      \sum_{1 \le i \le n}(c_i + c_{2n+1-i})(\a_i + \a_{2n+1-i})
\end{align*}
Hence $c_{2n+1} = 1$ and the condition $\a \in \vS_0^+$ is given by 
\begin{equation*}
c_i + c_{2n+1-i} = 1, \quad (1 \le i \le n). 
\end{equation*}
This implies that $\a = \a_{n+1} + \cdots + \a_{2n+1} = \ve_{n+1} - \ve_{2n+2}$. 
If we write $\vD_0^+ = \{ \ve_i - \ve_j \mid 1 \le i < j \le 2n+2\}$, 
we have 

\begin{equation*}
\tag{5.6.1}
\vS_0^+ = \{ \ve_{n+1} - \ve_{2n+2} \}.
\end{equation*}
\par
Hence $\vS_0^+$ is a positive subsystem of $\vD_0^+$ of type $A_1$.

\para{5.7.}
We shall describe the set of positive real roots for $\vD$ appearing 
as $\ul \vD$ under the isomorphism $(Q^{\s})' \simeq \ul Q$ in 5.1. 
Let $\vD$ be the irreducible root system 
of type $X_n^{(r)}$ with vertex set $I = \{ 0,1, \dots, n\}$, where 
$I_0 = \{ 1, \dots, n\}$ corresponds to the finite system $I_0$.  In the case
where $X_n^{(r)} \ne A_{2n}^{(2)}$, the choice of 0 is determined uniquely, 
up to the diagram automorphisms, in which case, we follow the ordering of [K].
While in the case where $X_n^{(r)} = A_{2n}^{(2)}$, we use the ordering 
as in ($A_a1$) in 1.5, which is the ordering used in [BN], and is opposite to the 
one in [K]. We denote by $\vD_{0,l}^+$ the set of positive long roots in $\vD_0$, 
and $\vD_{0,s}^+$ the set of positive short roots in $\vD_0$. 
Let $\vD^{\re,+}$ be the set of positive real roots in $\vD$.  
\par
Assume that $X_n^{(r)} \ne A_{2n}^{(2)}$.  In this case, $\vD^{\re,+}$ is 
written as $\vD^{\re,+} = \vD^{\re,+}_{>} \sqcup \vD^{\re,+}_{<}$, where
\begin{align*}
\tag{5.7.1}
\vD^{\re,+}_{>} &= \{ \a + m\d \mid \a \in \vD_{0,s}^+, m \in \BZ_{\ge 0} \}
          \sqcup \{ \a + mr\d \mid \a \in \vD_{0,l}^+, m \in \BZ_{\ge 0} \}, \\
\vD^{\re,+}_{<} &= \{ \a + m\d \mid \a \in -\vD_{0,s}^+, m \in \BZ_{> 0} \}
          \sqcup \{ \a + mr\d \mid \a \in -\vD_{0,l}^+, m \in \BZ_{>  0} \}. \\
\end{align*} 

\par
 Next assume that $X_n^{(r)} = A_{2n}^{(2)}$. In this case, $\vD^{\re,+}$ 
is given as follows. Note that since our ordering is opposite to [K], 
this description of $\vD^{\re,+}$ is different from the one in [K]. 
We fix the inner product $(\ ,\ )$ so that $(\a, \a) = 1,2,4$.  
We give the ordering of $A_{2n}^{(2)}$ as in 1.5, namely 
$\a_0$ is such that $(\a_0, \a_0) = 4$. 
The finite root system $\vD_0$ corresponding to $I_0$ is of 
type $B_n$. 
Then $\vD^{\re,+}$ is given as 
$\vD^{\re,+} = \vD^{+}_s \sqcup \vD^{+}_l \sqcup \vD^{+}_{2s}$, 
where
\begin{align*}
\tag{5.7.2}
\vD^{+}_s &= \{ m\d + \a \mid \a \in \vD^+_{0,s}, m \in \BZ_{\ge 0}\}
           \sqcup \{ m\d -\a \mid \a \in \vD^+_{0,s}, m \in \BZ_{> 0} \}, \\ 
\vD^+_l &= \{ m\d + \a \mid \a \in \vD^+_{0,l}, m  \in \BZ_{\ge 0}\} 
           \sqcup \{ m\d -\a \mid \a \in \vD^+_{0,l}, m \in \BZ_{> 0} \}, \\
\vD^+_{2s} &= \{ m\d + 2\a \mid \a \in \vD_{0,s}, m \text{ : odd integer $\ge 0$, } \}, 
\end{align*}
where $\vD_{0,s}$ is the set of short roots in $\vD_0$.
Note that 
\begin{align*}
\vD^+_s &= \{ \a \in \vD^+ \mid (\a,\a) = 1\}, \\
\vD^+_l &= \{ \a \in \vD^+ \mid (\a,\a) = 2\}, \\
\vD^+_{2s} &= \{ \a \in \vD^+ \mid (\a,\a) = 4\}. 
\end{align*}

\para{5.8.}
Let $\ul\vD_{0,l}$ (resp. $\ul\vD_{0,s}$) be the set of long roots 
(resp. short roots) in $\ul\vD_0$.  We denote by 
$\ul\vD_{0,l}^+, \ul\vD_{0,s}^+$ the corresponding positive roots.
Set
\begin{equation*}
\tag{5.8.1}
\begin{aligned}
\vS_s^+ &= \{ \a \in \vD_0^+ \mid O(\a) \in \ul\vD_{0,s}^+ \}, \\
\vS_l^+ &= \{ \a \in \vD_0^+ \mid O(\a) \in \ul\vD_{0,l}^+ \}, \\
\vS_{l'} &= \{ \a \in \vD_0^+ \mid O(\a) \in \d + \ul\vD_{0,l} \}, \\
\vS_{l''} &= \{ \a \in \vD_0^+ \mid O(\a) \in 2\d + \ul\vD_{0,l} \}, \\
\vS_{l'''} &= \{ \a \in \vD_0^+ \mid O(\a) \in 3\d + \ul\vD_{0,l} \}, \\
\vS_{s'} &= \{ \a \in \vD_0^+ \mid O(\a) \in \d + \ul\vD_{0,s} \}, \\
\vS_{2s'} &= \{ \a \in \vD_0^+ \mid O(\a) \in \d + 2\ul\vD_{0,s} \}, \\
\vS_{2s''} &= \{ \a \in \vD_0^+ \mid O(\a) \in 3\d + 2\ul\vD_{0,s} \}, \\
\vS_0^+    &= \{ \a \in \vD_0^+ \mid O(\a) \in \BZ_{> 0}\d \}. 
\end{aligned}
\end{equation*}  
$\vS_0^+$ is already defined, and was determined in 5.3 $\sim$ 5.6.
Note that $\vS_{2s'}, \vS_{2s''}$ occurs only in the case where 
$(X, \ul X) = (D_{2n}^{(1)}, A_{2n-2}^{(2)})$. 
We define $\vS_{l'}^+, \vS_{l''}^+, \vS_{l'''}^+$ 
(resp. $\vS_{s'}^+, \vS_{2s'}^+, \vS_{2s''}^+$) 
by replacing 
$\ul\vD_{0,l}$ by $\ul\vD_{0,l}^+$ (resp. by replacing $\ul\vD_{0,s}$ 
by $\ul\vD_{0,s}^+$) in the corresponding formulas. 

\para{5.9.}
Let $\th$ be the highest root in $\vD^+_0$.  Then we have 
$\a_0 = \d - \th$. 
Let $\ul\th$ be the highest root in $\ul\vD_0^+$, and 
$\ul\th_s$ the highest short root in $\ul\vD_0^+$.  The following relations 
are verified by a direct computation.

\begin{equation*}
\tag{5.9.1}
O(\a_0) = \begin{cases}
             \d - 2\ul\th_s  &\quad\text{ Case (A), } \\
             \d - \ul\th     &\quad\text{ Case (B), } \\
             \d - \ul\th_s   &\quad\text{ Case (C), } \\
            2\d - \ul\th     &\quad\text{ Case (D). }
          \end{cases}   
\end{equation*}
This implies that

\begin{equation*}
\tag{5.9.2}
O(\th) = \begin{cases}
           3\d + 2\ul\th_s  &\quad\text{ Case (A), } \\
           3\d + \ul\th     &\quad\text{ Case (B), } \\
           \d  + \ul\th_s   &\quad\text{ Case (C), } \\
           2\d + \ul\th     &\quad\text{ Case (D). } 
         \end{cases}
\end{equation*}
Note that $O(\th) = 2(\th + \s(\th))$ in Case (D). 
\par
Let $W$ be the Wyl group of $\vD$, and $\ul W$ the Weyl group of $\ul\vD$.
As in 3.2, the correspondence $s_{\eta} \mapsto t_{\eta}$ gives an isomorphism
$\ul W \isom W^{\s}$.
Let $\ul W_0$ be the Weyl group associated to $\ul \vD_0$.
Then $\ul W_0$ is regarded as a subgroup of $\ul W$, and  
under the identification $\ul W \simeq W^{\s}$, we regard $\ul W_0$ as a subgroup 
of $W$.  
\par
We prove the following lemma.

\begin{lem} 
Under the notation in (5.8.1), we have
\begin{equation*}
\vD_0^+ = \begin{cases}
       \vS_s^+ \sqcup \vS_l^+ \sqcup \vS_{l'} 
              \sqcup \vS_{2s'} \sqcup \vS_{2s''} \sqcup \vS_0^+
                  &\quad\text{Case {\rm  (A)}, }  \\ 
       \vS_s^+ \sqcup \vS_l^+ \sqcup \vS_{s'}
              \sqcup \vS_{l'} \sqcup \vS_{l''} \sqcup \vS_{l'''} \sqcup \vS_0^+
                  &\quad\text{Case {\rm (B)}, } \\
      \vS_s^+ \sqcup \vS_l^+ 
             \sqcup \vS_{s'} \sqcup \vS_{l''} \sqcup \vS_0^+
                  &\quad\text{Case {\rm (C)}, }   \\
      \vS_s^+ \sqcup \vS_l^+
            \sqcup \vS_{l''} \sqcup \vS_0^+
                  &\quad\text{Case {\rm (D)}.}       
     \end{cases}
\end{equation*}
\end{lem}

\begin{proof}
Set
\begin{equation*}
Y = \{ \a \in \vD_0^+ \mid O(\a) \in \ul\vD_0^+\} = \vS_s^+ \sqcup \vS_l^+.
\end{equation*}    

We prove the lemma in a similar way as in [SZ2, Lemma 5.5], separately 
for each case.  
\par\medskip
Case (A). \ 
Let $E$ be the union of $\ul W_0$-orbits of 
$\s(\a_0)$ and of $\th$, and set $E^+ = E \cap \vD_0^+$. 
We know by (5.9.1) and (5.9.2) that $O(\a_0) = \d - 2\ul\th_s$ and
$O(\th) = 3\d + 2\ul\th_s$. Since any short root in $\ul\vD_0$ is $\ul W_0$-conjugate
to $\ul\th_s$, we see that for any $\a \in E^+$, $O(\a)$ can be written as 
$\d + 2\ul\b$ or $3\d + 2\ul\b$, where $\ul\b \in \ul\vD_{0,s}$.  
Hence $\a \in \vS_{2s'} \sqcup \vS_{2s''}$. 
It follows that 
\begin{equation*}
E^+ \subset \vS_{2s'} \sqcup \vS_{2s''} \subset \vD_0^+ - Y - \vS_{l'} - \vS_0^+.
\end{equation*}
We show that these inclusion relations are all equalities. 
For any $\ul\b \in \ul\vD_{0, s}$, there exists at least three  
$\a \in E^+$ such that $O(\a) = \d + 2\ul\b$, and at least one $\a \in E^+$
such that $O(\a) = 3\d + 2\ul\b$.   It follows that 
$4|\ul\vD_{0,s}| \le |E^+|$.  
On the other hand, under the notation of 5.4, we see that 
$\vD_0^+ = \{ \ve_i \pm \ve_j \mid 1 \le i < j \le 2n\}$, and
\begin{equation*}
\vS_0^+ - Y  - \vS_{l'} = \{ \ve_1 \pm \ve_j \mid 2 \le j \le 2n\}
               \cup \{ \ve_i \pm \ve_{2n} \mid 2 \le i \le 2n-1\}.
\end{equation*}
By (5.4.7), $(\vS_0^+ - Y - \vS_{l'}) \cap \vS_0^+ = \{ \ve_1 \pm \ve_{2n}\}$.
Hence $|\vS_0^+ - Y - \vS_{l'} - \vS_0^+| = 8n-8$.  As $\ul\vD_0$ is of type 
$B_{n-1}$, we have$|\ul\vD_{0, s}| = 2(n-1)$,  hence
$|E^+| = |\vD_0^+ - Y - \vS_{l'} - \vS_0^+|$.  It follows that 

\begin{equation*}
\tag{5.10.1}
E^+ = \vS_{2s'} \sqcup \vS_{2s''} = \vD_0^+ - Y - \vS_{l'} - \vS_0^+. 
\end{equation*} 
The required formula follows from this.
\par\medskip
Case (B) \ 
First note that, under the notation in 5.4, for $\a = \ve_i + \ve_j \in Z_2$, 
we have $\a + \s(\a) = \d + (\ve_i + \ve_j - \ve_{2n+2-i} - \ve_{2n+2-j})$. 
This comes from (5.4.3). 
The condition that $\a + \s(\a) \in \vD^{\re}$ is given by $2n +2-i = i$, namely, 
$i = n+1$. Hence we have  
\begin{equation*}
\tag{5.10.2}
\{ \a \in Z_2 \mid O(\a) \in 2\d + \ul\vD_{0,l}\} 
    = \{ \ve_i + \ve_{n+1} \mid 2 \le i \le 2n, i \ne n+1\}.
\end{equation*} 
Let $E$ be the union of $\ul W_0$-orbits of $\b_1, \b_2, \b_3 \in \vD_0$, 
where $O(\b_k) \in k\d + \ul\vD_{0,l}$ for $k = 1,2,3$, and set $E^+ = E \cap \vD_0^+$. 
By (5.9.1) and (5.9.2), we have $O(\a_0) = \d - \ul\th, O(\th) = 3\d + \ul\th$, 
we may choose $\b_1 = \s(\a_0), \b_3 = \th$. 
It follows from (5.10.2), we may choose, for example, $\b_2 = \ve_2 + \ve_{n+1}$. 
Then we have
\begin{equation*}
E^+ \subset \vS_{l'}\sqcup \vS_{l''} \sqcup \vS_{l'''}  \subset 
\vD_0^+ - Y - \vS_{s'} - \vS_0^+.
\end{equation*}
We show that these inclusion relations are actually equalities. 
For each $\ul \b \in \ul\vD_{0,l}$, there exist, at least three $\a \in E^+$
such that $O(\a) = \d + \ul\b$, at least one $\a \in E^+$ such that
$O(\a) = 2\d + \ul\b$, and at least one $\a \in E^+$ such that $O(\a) = 3\d + \ul\b$
(note that $\b_2 \in Z_2$). It follows that $5|\ul\vD_{0,l}| \le |E^+|$. 
On the other hand, by the computation for $Z_2, Z_4$ in 5.4, and by (5.10.2), 
\begin{align*}
\vD_0^+ - Y - \vS_{s'}
  &= \{ \ve_1 \pm \ve_j \mid 2 \le j \le 2n + 1 \}  \\ 
    &\quad\sqcup
    \{ \ve_i \pm \ve_{2n+1} \mid 2 \le i \le 2n \} \\
    &\quad\sqcup
     \{ \ve_i + \ve_{n+1} \mid 2 \le i \le 2n, i \ne n+1 \}. 
\end{align*}
The first two sets are contained in $Z_4$ and the third one is contained in $Z_2$
(see (5.4.5) and (5.10.2)). 
By (5.4.11), we have 
\begin{equation*}
(\vD_0^+ - Y - \vS_{s'}) \cap \vS_0^+ 
      = \{ \ve_1 \pm  \ve_{2n+1}, \ve_1 \pm \ve_{n+1}, \ve_{n+1} \pm \ve_{2n+1}  \}. 
\end{equation*}
It follows that 
\begin{equation*}
|\vD_0^+ - Y - \vS_{s'} - \vS_0^+| = 10(n-1).
\end{equation*}
Since $\ul\vD_0$ is of type $C_{n-1}$, we have $|\ul\vD_{0,l}| = 2(n-1)$, 
and so $|\vD_0^+ - Y - \vS_{s'} - \vS_0^+| = |E^+|$.  Hence

\begin{equation*}
\tag{5.10.3}
E^+ = \vS_{l'}\sqcup \vS_{l''} \sqcup \vS_{l'''}  =  
\vD_0^+ - Y - \vS_{s'} - \vS_0^+.
\end{equation*}

The required formula follows from this. 

\par\medskip
Case (C)\ 
In this case, $\vD_0^+ = \{ \ve_i \pm \ve_j \mid 1 \le i < j \le 2n+1\}$.
The action of $\s$ on $\a \in \vD_0^+$ is given as in (5.4.3).
It follows that 

\begin{equation*}
\tag{5.10.4}
\{ \a \in \vD_0^+ \mid O(\a) \in 2\d + \ul\vD_{0,l}\}
       = \{ \ve_i + \ve_{n+1} \mid 1 \le i \le 2n+1, i \ne n+1\}.
\end{equation*}

Let $E$ be the union of $\ul W_0$-orbit of $\b_1, \b_2 \in \vD_0^+$
such that $O(\b_1) \in \d + \ul\vS_{0,s}$, $\b_2 \in 2\d + \ul\vD_{0,l}$, 
and set $E^+ = E \cap \vD^+$. 
By (5.9.2), we may choose $\b_1 = \th$.  By (5.10.4), we may choose, for example, 
$\b_2 = \ve_1 + \ve_{n+1}$.  
Then we have

\begin{equation*}
E^+ \subset \vS_{s'} \sqcup \vS_{l''} \subset \vD_0^+ - Y - \vS_0^+.
\end{equation*}
We show that these inclusion relations are actually equalities.
For each $\ul\b \in \ul\vD_{0,s}$, there exists at least one $\a \in E^+$ 
such that $O(\a) \in \d + \ul\b$, and for each $\ul\b' \in \ul\vD_{0,l}$, 
there exists at least one $\a \in E^+$ such that $O(\a) \in 2\d + \ul\b'$. 
It follows that $|E^+| \ge |\ul\vD_0|$. 
On the other hand, we have
\begin{equation*}
\vD_0^+ - Y = \{ \ve_i + \ve_j \mid 1 \le i < j \le 2n+1\}. 
\end{equation*}
By (5.5.2), $\vS_0^+ \subset \vD_0^+ - Y$, and $|\vS_0^+| = n$. 
It follows that $|\vD_0^+ - Y - \vS_0^+| = 2n^2$. 
Since $\ul\vD_0$ is of type $C_n$, $|\ul\vD_0| = 2n^2$.  
Thus we have 
\begin{equation*}
\tag{5.10.5}
E^+ = \vS_{s'} \sqcup \vS_{l''} = \vD_0^+ - Y  - \vS_0^+.
\end{equation*} 
The required formula follows from this. 

\par\medskip
Case (D) \ 
Let $E$ be the $\ul W_0$-orbit of $\th$, and set 
$E^+ = E \cap \vD_0^+$.  By (5.9.2), $O(\th) = 2\d + \ul\th$.
Then we have
\begin{equation*}
E^+ \subset \vS_{l''} \subset \vD_0^+ - Y - \vS_0^+.
\end{equation*}
We show that these inclusions are actually equalities. 
We have $|E^+| \ge |\vD_{0,l}|$. On the other hand, 
\begin{equation*}
\vD_0^+ - Y = \{ \ve_i - \ve_{2n+2}\}.
\end{equation*}
Since $\vS_0^+ = \{ \ve_{n+1} - \ve_{2n+2}\}$, we see that
$|\vD_0^+ - Y - \vS_0^+| = 2n$. 
Since $\ul\vD_0$ is of type $C_n$, we have $|\ul\vD_{0,l}| = 2n$.
Thus we have
\begin{equation*}
E^+ = \vS_{l''} = \vD_0^+ - Y - \vS_0^+.
\end{equation*}
The required formula follows from this.
\par
The lemma is proved. 
\end{proof}

\para{5.11.}
For each $m \ge 0$, $x \in \{ s, 2s', 2s'', l, l', l'',l'''\}$, 
we define subsets $\vS_x(m), \vS_x'(m)$ of $\vD^{\re,+}$ as
\begin{equation*}
\tag{5.11.1}
\begin{aligned}
\vS_x(m) &= \bigcup_i\s^i\{ m\d + \a \mid \a \in \vS_x^+\}, \\
\vS_x'(m) &= \bigcup_i\s^i\{ m\d - \a \mid \a \in \vS_x^+\},
\end{aligned}
\end{equation*}
where we assume that $m \ge 1$ for $\vS_x'(m)$. 
We define an equivalence relation on $\vD^{\re}$ 
by the condition that $\b, \b' \in \vD^{\re}$ are equivalent 
if $O(\b) = O(\b')$. 
We denote by $\wh\vS_x(m)$ the set of equivalence classes in $\vS_x(m)$, 
and define $\wh\vS_x'(m)$ similarly.  
Let $\wh \vD^{\re,+}$ be the set of equivalence classes in 
$\vD^{\re,+} - \bigcup_{m \ge 0}\{ m\d \pm \a \mid \a \in \vS_0^+\}$.
For $x = 0$ and $m \ge 0$, we set

\begin{equation*}
\tag{5.11.2}
\vS_0(m) = \{ m\d + \a \mid \a \in \vS_0^+\}, \qquad
\vS'_0(m) = \{ m\d - \a \mid \a \in \vS_0^+\}.
\end{equation*}
\par
Also for each $m \ge 0$, we define subsets of $\ul\vD^{\re,+}$ by 

\begin{equation*}
\tag{5.11.3}
\begin{aligned} 
\ul\vS_s(m) &= \{ m\d + \a \mid \a  \in \ul\vD_{0,s}^+\}, &\quad
               \ul\vS_s'(m) &= \{ m\d - \a \mid \a  \in \ul\vD_{0,s}^+\},      \\       
\ul\vS_{2s}(m) &= \{ m\d + 2\a \mid \a  \in \ul\vD_{0,s}^+\}, &\quad
               \ul\vS_{2s}'(m) &= \{ m\d - 2\a \mid \a  \in \ul\vD_{0,s}^+\},      \\       
\ul\vS_l(m) &= \{ m\d + \a \mid \a  \in \ul\vD_{0,l}^+\}, &\quad
               \ul\vS_l'(m) &= \{ m\d - \a \mid \a  \in \ul\vD_{0,l}^+\}.      \\       
\end{aligned}
\end{equation*}
Note that we assume $m \ge 1$ in the right hand sides of (5.11.2) and (5.11.3).  

The following result is a generalization of [SZ2, Prop. 5.7].
The proof is straightforward from Lemma 5.10, by using the description of 
the root system $\ul\vD^{\re,+}$ given in 5.7.

\begin{prop}  
Under the identification $(Q^{\s})' \simeq \ul Q$, the followings hold.
\begin{enumerate}
\item \ For each $\b \in \vD^+$, $O(\b) \in \ul\vD^+$.  
\item \ In the cases {\rm (A), (B), (C),} 
we have a partition 
\begin{equation*}
\wh\vD^{\re,+} = \bigsqcup_{x, m\ge 0}\wh\vS_x(m)
                    \sqcup \bigsqcup_{x, m \ge 1}\wh\vS_x'(m).
\end{equation*}
The map $f : \b \mapsto O(\b)$ induces a bijection 
$\wh\vD^{\re,+} \isom \ul\vD^{\re,+}$.
The bijection is given more precisely as follows. 
(In the formulas below, we only give the part corresponding to $\wh\vS_x(m)$.
The part corresponding to $\wh\vS_x'(m)$ is described similarly.). 
\par\bigskip\noindent
Case {\rm (A)} 
\begin{equation*}
\begin{aligned}
\wh \vS_s(m) &\isom \ul\vS_s(m), \\
\wh\vS_l(m)  &\isom \ul\vS_l(m), &\qquad &(m \equiv 0 \mod 2),\\
\wh\vS_{l'}(m)  &\isom \ul\vS_l(m), &\qquad &(m \equiv 1 \mod 2),\\
\wh\vS_{2s'}(m) &\isom \ul\vS_{2s}(m), &\qquad &(m \equiv 1 \mod 4), \\
\wh\vS_{2s''}(m) &\isom \ul\vS_{2s}(m), &\qquad &(m \equiv 3 \mod 4).   
\end{aligned}
\end{equation*}
Case {\rm (B)}
\begin{equation*}
\begin{aligned}
\wh \vS_s(m) &\isom \ul\vS_s(m), &\qquad &(m \equiv 0 \mod 2), \\
\wh \vS_{s'}(m) &\isom \ul\vS_s(m), &\qquad &(m \equiv 1 \mod 2), \\
\wh\vS_l(m)  &\isom \ul\vS_l(m), &\qquad &(m \equiv 0 \mod 4), \\
\wh\vS_{l'}(m)  &\isom \ul\vS_l(m), &\qquad &(m \equiv 1 \mod 4), \\
\wh\vS_{l''}(m)  &\isom \ul\vS_l(m), &\qquad &(m \equiv 2 \mod 4), \\
\wh\vS_{l'''}(m)  &\isom \ul\vS_l(m), &\qquad &(m \equiv 3 \mod 4). \\
\end{aligned}
\end{equation*}
Case {\rm (C)}
\begin{equation*}
\begin{aligned}
\wh\vS_s(m) &\isom \ul\vS_s(m), &\qquad &(m \equiv 0 \mod 2), \\ 
\wh\vS_{s'}(m) &\isom \ul\vS_s(m), &\qquad &(m \equiv 1 \mod 2), \\ 
\wh\vS_{l}(m) &\isom \ul\vS_l(m), &\qquad &(m \equiv 0 \mod 4), \\ 
\wh\vS_{l''}(m) &\isom \ul\vS_l(m), &\qquad &(m \equiv 2 \mod 4). \\ 
\end{aligned}
\end{equation*}
\item 
In the case {\rm (D)}, 
we have a partition 
\begin{equation*}
\wh\vD^{\re,+} = \bigsqcup_{x, m\ge 0}\wh\vS_x(2m)
                    \sqcup \bigsqcup_{x, m \ge 0}\wh\vS_x'(2m+1).
\end{equation*}
For $\b \in \vS_x(2m)$ (resp. $\b \in \vS_x(2m +1)$), 
$O(\b)$ is written as $O(\b) = 2m\d + \ul\b$ (resp. $O(\b) = (2m+1)\d + \ul\b$)
with $\ul\b \in \ul\vD_0^+$. Define $f: \wh\vS_x(2m) \to \ul\vD$
(resp. $f: \wh\vS_x(2m +1) \to \ul\vD$ by $f(\b) = m\d + \ul\b$. 
Then $f$ induces a bijection $\wh\vD^{\re,+} \isom \ul\vD^{\re,+}$.  
More precisely, the map is given as follows;
\begin{equation*}
\begin{aligned}
\wh\vS_s(2m) &\isom \ul\vS_s(m),  \\ 
\wh\vS_l(2m) \sqcup \wh\vS_{l''}(2m) &\isom \ul\vS_l(m).
\end{aligned}
\end{equation*}
The case for $\wh\vS_x(2m+1)$ is given similarly.
\end{enumerate}
\end{prop}

\par\bigskip
\section{ Modified PBW-bases -- affine case }

\para{6.1.}
We keep the assumption in 5.1.
We consider the doubly infinite sequence 
$\ul\Bh = (\dots, \eta_{-1}, \eta_0, \eta_1, \dots)$ for
$\ul W$ associated to $\ul\xi \in \ul P\ck_{\cl}$ as defined in [SZ2, 1.5], 
applied for $\ul\BU_q^-$. We define a sequence 
$\Bh' = (\dots, i_{-1}, i_0, i_1, \dots)$ by replacing $\eta = \eta_i$ by 
a sequence $J_{\eta}$ in $I$, where 
\begin{equation*}
J_{\eta} = \begin{cases}
               (j_1, \dots, j_{|\eta|}) &\quad\text{ if $\d_{\eta} = 1$ and }
                             \eta = \{ j_1, \dots, j_{|\eta|}\}, \\
               (i,j,i) &\quad\txt{ if $\d_{\eta} = 2$ and $\eta = \{ i,j\}$. }
           \end{cases}             
\end{equation*}
Thus 
\begin{equation*}
\tag{6.1.1}
\Bh' = (\dots, \underbrace{i_{-k-k''}, \dots, i_{-k-1}}_{J_{\eta_{-1}}}
                 \underbrace{i_{-k}, \dots, i_{-1}, i_0}_{J_{\eta_{0}}}, 
           \underbrace{i_1, \dots, i_{k'}}_{J_{\eta_1}}, \cdots). 
\end{equation*}

For any $k < \ell$, $\ul w = s_{\eta_k}\cdots s_{\eta \ell}$ is a reduced expression 
of $\ul w \in \ul W$.   
We define $w \in W$ by 
\begin{equation*}
\tag{6.1.2}
w = t_{\eta_k}\cdots t_{\eta_{\ell}},
\end{equation*} 
where $t_{\eta}$ is defined as in 3.2. 
In view of Proposition 5.12, 
\begin{equation*}
\vD^+ \cap w\iv(-\vD^+) = f\iv(\ul\vD^+ \cap \ul w\iv(-\ul\vD^+)).
\end{equation*}
It follows that (6.1.2) is a reduced expression of $w$, and $\Bh'$ gives 
an infinite reduced word. 
We define $\ul\b_k \in \ul\vD^{\re,+}$ for any $k \in \BZ$ by
\begin{equation*}
\tag{6.1.3}
\ul\b_k = \begin{cases}
         s_{\eta_0}s_{\eta_{-1}}\cdots s_{\eta_{k+1}}(\ul\a_{\eta_k})
                            &\quad\text{ if $k \le 0$, } \\
         s_{\eta_1}s_{\eta_2}\cdots s_{\eta_{k-1}}(\ul\a_{\eta_k})
                            &\quad\text{ if $k > 0$. } 
       \end{cases}
\end{equation*} 
Then by [BN, 3.1], we have
\begin{equation*}
\tag{6.1.4}
\ul\vD_{>}^{\re,+} = \{ \ul\b_k \mid k \le 0\}, \qquad
\ul\vD_{<}^{\re,+} = \{ \ul\b_k \mid k > 0\}. 
\end{equation*}
\par
Also define $\b_k \in \vD^{\re,+}$ for $k \in \BZ$ by 
\begin{equation*}
\tag{6.1.5}
\b_k = \begin{cases}
         s_{i_0}s_{i_{-1}}\cdots s_{i_{k+1}}(\a_{i_k})
                            &\quad\text{ if $k \le 0$, } \\
         s_{i_1}s_{i_2}\cdots s_{i_{k-1}}(\a_{i_k})
                            &\quad\text{ if $k > 0$. } 
       \end{cases}
\end{equation*} 
\par
Recall $\vS_x(m), \vS_x'(m)$ in (5.11.1) for 
$x \in \SX = \{ s, 2s', 2s'', l, l', l'', l'''\}$
and $\vS_0(m)$, $\vS_0'(m)$ in (5.11.2). 
We define $\vD^{(0)}_{>}, \vD^{(0)}_{<}$ by

\begin{equation*}
\vD^{(0)}_{>} = \bigsqcup_{m \ge 0}\bigsqcup_{x \in \SX}\vS_x(m), 
\qquad
\vD^{(0)}_{<} = \bigsqcup_{m \ge 1}\bigsqcup_{x \in \SX}\vS'_x(m). 
\end{equation*}

Also set
\begin{equation*}
\vD^{(1)}_{>} = \bigsqcup_{m \ge 0}\vS_0(m), \qquad
\vD^{(1)}_{<} = \bigsqcup_{m \ge 1}\vS'_0(m).
\end{equation*}

Then $\vD^{(0)}_{>}$ and $\vD^{(0)}_{<}$ are $\s$-stable, and 
$\vD^{(1)}_{>} \sqcup \vD^{(1)}_{<}$ is $\s$-stable.
By Proposition 5.12 and (6.1.4), we see that

\begin{equation*}
\tag{6.1.6}
\{ \b_k \mid k \le 0\} = \vD^{(0)}_{>}, \qquad
\{ \b_k, \mid k > 0\} = \vD^{(0)}_{<}.
\end{equation*} 

Moreover, we have
$\vD^{\re,+} = \vD^{(0)}_{>} \sqcup \vD^{(1)}_{>} \sqcup \vD^{(1)}_{<} \sqcup \vD^{(0)}_{<}$.  

\para{6.2.}
Hereafter we apply the theory of PBW-basis associated to the convex order
developed by Muthiah-Tingley [MT] to our setup.  
We follow the notation in [SZ2, Section 4].  
For any real root $\b = \a + m\d \in \vD^{\re}$, we denote by $\ol\b = \a \in \vD_0$.
Set $\vD_+^{\min} = \vD^{\re,+} \sqcup \{ \d\}$. 
For any subset $Z$ of $\vD_+^{\min}$, we define a total order $\a \lve \b$, 
called a convex order, 
by the following two conditions, 
\par\medskip\noindent
(6.2.1) \ If $\a, \b \in Z$ with $\a + \b \in Z$, then $\a+ \b$ is in between 
$\a$ and $\b$. 
\par\medskip\noindent
(6.2.2) \ If $\a \in Z, \b \in \vD_+^{\min} - Z$ with $\a+\b \in Z$, then 
$\a \lv \a + \b$. 
\par\medskip\noindent
If in (6.2.2), the condition ``$\a \lv \a + \b$'' is replaced by
``$\a + \b \lv \a$'', such a total order is called a reverse convex order.     

\par\medskip
We define a total order $\lv$ on $\vD^{(0)}_{>}$ by 
$\b_0 \lv \b_{-1} \lv \b_{-2} \cdots$, and a total order $\lv$ on $\vD^{(0)}_{<}$
by $\cdots \lv \b_2 \lv \b_1$. Then by [SZ2, Lemma 4.3], $\lv$ gives a convex order
on $\vD^{(0)}_{>}$ and a reverse convex order on $\vD^{(0)}_{<}$.   
This implies that 
\par\medskip
\noindent
(6.2.3) \ $\vD^{(0)}_{>} \lv \d \lv \vD^{(0)}_{<}$ satisfies the condition 
(6.2.1). In particular, $\b \lv \d + \b$ for $\b \in \vD^{(0)}_{>}$
and $\b + \d \lv \b$ for $\b \in \vD^{(0)}_{<}$.  
\par\medskip
We shall define a total order on $\vD^{(1)}_{>}$ and on $\vD^{(1)}_{<}$. 
In the case where $(X, \ul X)$ is in the case (A), (C), ot (D), $\vS_0^+$ 
is a positive subsystem of $\vD^+$ of type $nA_1$ for some $n \ge 1$. 
In this case, write $\vS_0^+ = \{ \g_1, \dots, \g_t\}$ in any order. We define 
a total order on $\vS^{(1)}_{>}$ by the condition that $m\d + \g_i \lv m'\d + \g_j$
if $m < m'$ or if $m = m'$ and $i < j$. 
Similarly, we define a total order on $\vD^{(1)}_{<}$ by the condition that
$m\d - \g_i \lv m'\d - \g_j$ if $m > m'$ or if $m = m'$ and $i > j$.   
\par
We now consider the case (B).  In this case $\vD_0^+$ is a positive subsystem 
in $\vD^+$ of type $(n-1)A_1 + A_3$ by (5.4.11). 
Thus $\vD_0^+ = X_1 \sqcup X_2$, where $X_1$ is a positive system of type
$A_3$ and $X_2$ is a positive system of type $(n-1)A_1$.  
Let $\vD^{(1)}$ be the subsystem of $\vD$
spanned by $\g \in X_1, \g_0 = \d - \th_1$, and $\g', \d - \g'$ for $\g' \in X_2$, 
where $\th_1$ is the highest root in $X_1$.  
Then $\vD^{(1)}$ is an affine root system of type $A_3^{(1)} + (n-1)A_1^{(1)}$.
We denote by $W_1$ the Wyl group of $\vD^{(1)}$. We consider the infinite reduced
word $\Bh^{(1)} = (\dots, i_{-1}, i_0, i_1, \dots)$ of $W_1$ defined in [BN] as 
in [SZ, 1.5], and define $\b_k' \in (\vD^{(1)})^{\re,+}$ for $k \in \BZ$ 
as in  (6.1.5).   Then we have
\begin{equation*}
\tag{6.2.4}
\{ \b'_k \mid k \le 0\} = \vD^{(1)}_{>}, \qquad
\{ \b'_k \mid k > 0\}   = \vD^{(1)}_{<}. 
\end{equation*} 
Note that the discussion in the case (B) works also for the cases (A), (C), (D).
Thus again by [SZ2, Lemma 4.3], we see that 
\par\medskip\noindent
(6.2.5). $\vD^{(1)}_{>} \lv \d \lv \vD^{(1)}_{<}$ satisfies the 
condition (6.2.1).
\par\medskip
We define a total order $\lv$ on $\vD_+^{\min} = \vD^+ \cup \{ \d\}$ 
by extending the total order on $\vD^{(i)}_{>}$ and $\vD^{(i)}_{<}$ for $i = 1,2$, 
under the condition 
\begin{equation*}
\tag{6.2.6}
\vD^{(0)}_{>} \lv \vD^{(1)}_{>} \lv \d \lv \vD^{(1)}_{<} \lv \vD^{(0)}_{<}.
\end{equation*}

Recall that the coarse type of the convex order $\lve$ on $\vD_+^{\min}$
is defined as the unique $\ol w \in W_0$ such that 
$\ol w(\vD^+_0) = \{ \ol \b \in \vD_0 \mid \b \lv \d\}$ (see [SZ2, 4.8]).
The convex order is called the standard type if the coarse type $\ol w $ coincides with 
the identity element in $W_0$. 

\begin{lem}  
The total order $\lv$ gives a convex order on $\vD_+^{\min}$.
\end{lem}

\begin{proof}
We show that $\lv$ satisfies the condition (6.2.1) for $X = \vD_+^{\min}$.
Note that in this case, the condition (6.2.2) is redundant.  
It is enough to consider the case where $\b \in \vD^{(0)}_{>} \sqcup \vD^{(0)}_{<}$
and $\b' \in \vD^{(1)}_{>}\sqcup \vD^{(1)}_{<}$.
Take $\b \in \vD^{(0)}_{>}$, and assume that $\g = \b + \b' \in \vD^+$. 
We show that $\g \in \vD^{(0)}_{>}$ and $\b \lv \g$. 
We have $O(\g) = O(\b) + O(\b') \notin \BZ_{>0}\d$ since $O(\b) \notin \BZ_{>0}\d$
and $O(\b') \in \BZ_{>0}\d$. Thus $\g \in \vD^{(0)}_{>} \sqcup \vD^{(0)}_{<}$. 
If $\b' \in \vD^{(1)}_{>}$, we must have $\g \in \vD^{(0)}_{>}$, and 
$\b \lv \g$ (see (6.2.3)).  So consider the case where $\b' \in \vD^{(1)}_{<}$. 
Suppose that $\g \in \vD^{(0)}_{<}$.  We can write as $\b = m\d + \a$, 
$\b' = m'\d - \a'$ and $\b + \b' = m''\d - \a''$, where $\a, \a'' \in \vD_x^+$ 
and $\a' \in \vD_0^+$. Then we have $\a + \a'' = \a'$. 
This contradicts $\a' \in \vS_0^+$. Hence $\g \in \vS^{(0)}_{>}$, and 
we have $\b \lv \g$ by (6.2.3). Thus $\lv$ gives a convex order on $\vD_+^{\min}$. 
\end{proof}

\para{6.4.}
Let $\lv$ be the convex order on $\vD^{\min}_+$ defined in Lemma 6.3.
Let $\SC$ be the set of triples $\Bc = (\Bc_+, \Bc_0, \Bc_-)$, where
$\Bc_+ = (c_{\b})_{\b \lv \d}$, 
$\Bc_- = (c_{\g})_{\g \gv \d}$, and $\Bc_0 = (\r^{(i)})_{i \in I_0}$ 
is an $I_0$-tuple of partitions.    
For each finite sequence 
\begin{equation*}
\b_1 \lv \b_2 \lv \cdots \lv \b_N \lv \d \lv \g_M \lv \cdots \lv \g_2 \lv \g_1 
\end{equation*}
in $\vD^{\min}_+$, we define
\begin{equation*}
\tag{6.4.1}
L(\Bc, \lv) = f_{\b_1}^{\lv, (c_{\b_1})}\cdots f_{\b_N}^{\lv, (c_{\b_N})}
                   S_{\Bc_0}^{\lv}
             f_{\g_M}^{\lv, (c_{\g_M})}\cdots f_{\g_1}^{\lv, (c_{\g_1})}.
\end{equation*}
as in [SZ2, 4.19].  Then by the discussion in [SZ2, Section 4], 
$\SX_{\lv} = \{ L(\Bc, \lv) \mid \Bc \in \SC\}$ gives a PBW-basis of $\BU_q^-$. 
The coarse type $\ol w \in W_0$ of the convex order $\lv$ is 
given by $\ol w = \prod_{i \in \eta_0 - \{0\}}s_i$.
Hence by [MT, Theorem 4.13] that $S_{\Bc_0}^{\lv} = T_w S_{\Bc_0}$, where 
$S_{\Bc_0}$ is the one defined by [BN] (see [SZ2, 1.5]), and $w \in W$ is the minimal
length lift of $\ol w \in W_0$. 
Also by [SZ2, Lemma 6.7], for $\b \in \vD^{(0)}_{>} \sqcup \vD^{(0)}_{<}$, 
$f_{\b}^{\lv, (c)}$ coincides with $f_{\b_k}^{\Bh'}$, where
for $\b = \b_k$ in the notation of (6.1.6), 
we define the root vector $f_{\b_k}^{\Bh'}$ (with respect to the infinite reduced word
$\Bh'$ in (6.1.1)) by  
\begin{equation*}
f_{\b_k}^{\Bh'} = \begin{cases}
                    T_{i_0}T_{i_{-1}}\cdots T_{i_{k+1}}f_{i_k}^{(c)}
                             &\quad\text{ if }  k \le 0, \\
                    T_{i_1}\iv T_{i_2}\iv \cdots T_{i_{k-1}}\iv f_{i_k}^{(c)}
                             &\quad\text{ if } k > 0. 
                  \end{cases}
\end{equation*}  

\remark{6.5.}  In [SZ2, Lemma 6.4], it is stated that the convex order $\lv$ 
defined there is of standard type.  But this is incorrect, and the coarse type
$\ol w \in W_0$ is given in a similar way as above.  

\para{6.6.}
Let $\lv$ be a convex order of $\vD^{\min}_+$ such that $\b_0$ is minimal, 
i.e., $\b_0 \lv \b_{-1} \lv \b_{-2} \cdots$. We assume further that $\b_0 = \a_i$
is a simple root. Let $s_i = s_{\a_i} \in W$ be the simple reflection. 
We can define a new convex order $\lv^{s_i}$ by moving $\b_0$ to the right end, 
and by acting $s_i$ on the remaining part, namely, 
\begin{equation*}
s_i(\b_{-1}) \lv^{s_i} s_i(\b_{-2}) \lv^{s_i} \cdots \lv^{s_i}  s_i(\b_1) 
                                            \lv^{s_i} \b_0, 
\end{equation*}  
where $\b_0$ is maximal in $\lv^{s_i}$. 
Similarly, if $\b_1 = \a_i$ is maximal in $\vD_+^{\min}$, i.e., 
$\cdots \lv \b_2 \lv \b_1$, we can define a new convex order $\lv^{s_i}$ 
in a similar way, by moving $\b_1$ to the left end. 
\par
We apply this construction to our convex order on $\vD^{\min}_+$. 
In this case, the part in $\vD^{(0)}_{>} \sqcup \vD_{<}^{(0)}$ 
is determined by $\Bh'$, namely
for $\Bh' = (\dots, i_{-2}, i_{-1}, i_0, i_1, i_2, \dots)$, the order is given by
\begin{equation*}
\b_0 \lv \b_{-1} \lv \b_{-2} \lv \cdots \lv \b_2 \lv \b_1,
\end{equation*}
where $\b_0 = \a_{i_0}$, $\b_1 = \a_{i_1}$, 
and $\b_k$ are given as in (6.1.5). 
Thus the above construction of $\lv^{s_i}$ can be applied successively for 
$s_{i_0}, s_{i_{-1}}, \dots, s_{i_{k+1}}$, and one can define a convex order 
$\lv^w$ for $w = s_{i_0}s_{i_{-1}}\cdots s_{i_{k+1}}$ by 
$\lv^w = (\cdots ((\lv^{s_{i_0}})^{s_{i_{-1}}}) \cdots )^{s_{i_{k+1}}}$. 
We obtain 
\begin{equation*}
\tag{6.6.1}
\a_{i_k} \lv^w  s_{i_{k}}(\a_{i_{k-1}}) \lv^w s_{i_k}s_{i_{k-1}}(\a_{i_{k-2}}) 
             \lv^w \cdots 
\end{equation*}
Similarly, one can define a convex order $\lv^w$ for 
$w = s_{i_1}\cdots s_{i_{k-1}}$ ($k \ge 1$), and we obtain 
\begin{equation*}
\tag{6.6.2}
\cdots \lv^w  s_{i_k}s_{i_{k+1}}(\a_{i_{k+2}}) \lv^w s_{i_k}(\a_{i_{k+1}}) 
             \lv^w  \a_{i_k}
\end{equation*}

Here for each $p \le 0$, we consider the subsequence 
$(\eta_p, \dots, \eta_{-1}, \eta_0)$ of $\ul\Bh$, and let 
$(i_{k+1}, \dots, i_{-1}, i_0)$ ($k \le 0$) 
 be the subsequence of $\Bh'$ corresponding to
$(J_{\eta_p}, \dots, J_{\eta_0})$.  
Similarly, we define a subsequence $(i_1, \dots, i_{k-1})$ of $\Bh'$ 
associated to the subsequence $(\eta_1, \eta_2, \dots, \eta_p)$ for $p \ge 1$. 
Let $w = s_{i_0}\cdots s_{i_{k+1}}$ or $w = s_{i_1}\cdots s_{i_{k-1}}$, and 
consider the convex order $\lv^w$ defined above, which we denote by $\lv_p$.   
Thus $\lv_p$ gives a two-row order 
\begin{equation*}
\vD^{(0)}_{>,p} \lv_p \vD^{(1)}_{>,p} \lv_p \d \lv_p \vD^{(1)}_{<,p} \lv_p \vD^{(0)}_{<,p}. 
\end{equation*}
\par
Now we define $L(\Bc, \lv_p)$ by replacing $\lv$ by $\lv_p$ in (6.4.1). 
We note, by [MT, Cor. 4.14], that 
\begin{equation*}
\tag{6.6.3}
S_{\Bc_0}^{\lv_p} = \begin{cases}
                      T_{i_{p+1}}\iv\cdots T_{i_0}\iv(S_{\Bc_0}^{\lv}), 
                            &\quad\text{ if $p \le 0$,} \\
                      T_{i_{p-1}}\cdots T_{i_2}T_{i_1}(S_{\Bc_0}^{\lv}), 
                            &\quad\text{ if $p \ge 1$.}
                    \end{cases}
\end{equation*} 
Set
\begin{equation*} 
\SX_{\lv_p}  = \{ L(\Bc, \lv_p) \mid \Bc \in \SC_p\},
\end{equation*}
where $\SC_p$ is the parameter set for $\lv_p$. 
Then 
$\SX_{\lv_p}$ gives a PBW-basis of $\BU_q^-$. 
\par
By [MT, Prop.4.18], the following holds.

\begin{equation*}
\tag{6.6.4}
\begin{aligned}
T_{i_p}\iv L(\Bc, \lv_p) &= L(\Bc', \lv_{p-1}), &
             &\qquad\text{if $p \le 0$ and $c_{\a} = 0$, } \\
T_{i_{p+1}} L(\Bc, \lv_p) &= L(\Bc', \lv_{p+1}), &
             &\qquad\text{if $p \ge 1$ and $c_{\a} = 0$.} 
\end{aligned}
\end{equation*}
Here $\a = \a_{i_p}$ or $\a = \a_{i_{p+1}}$ according to $p \le 0$ or $p \ge 1$. 
$\Bc'$ is defined by $c'_{\b} = c_{s_{i_p}(\b)}$ for $\b \ne \a$, 
$c'_{\a} = 0$, and the imaginary part is defined by the action of $T_{i_p}\iv$ 
or $T_{i_{p+1}}$ accordingly (see (6.6.3)). 

\para{6.7.}
Following the discussion in 3.4, we modify the root vectors corresponding 
to $\vD^{(0)}_{>,p}$ and $\vD^{(0)}_{<,p}$. Note that for $\b \in \vD^{(0)}_{>,p}$
or $\b \in \vD^{(0)}_{<,p}$, $f_{\b}^{\lv}$ coincides with $f_{\b}^{\Bh'(w)}$, 
where $\Bh'(w)$ is the doubly infinite sequence obtained from $\Bh'$ by the action of 
$w \in W$ as in 6.6.
In the case where $\d_{\eta_k} = 1$, set 
$\Bd_k = (d_j)_{j \in \eta_k} \in \BN^{\eta_k} = \BN^{J_k}$, where $J_k$ is 
as in 6.1. We define  
\begin{align*}
\tag{6.7.1}
\wt f_{\eta_k}^{(\Bd_k)} = \begin{cases}
        \wt T_{\eta_p}\wt T_{\eta_{p-1}}\cdots \wt T_{\eta_{k+1}}
                (\prod_{j \in \eta_k}f_j^{(d_j)})  &\quad\text{ if $p \le 0$, } \\
\wt T_{\eta_{p+1}}\iv \wt T_{\eta_{p+2}}\iv \cdots \wt T_{\eta_{k-1}}\iv
                (\prod_{j \in \eta_{k}}f_j^{(d_j)}) &\quad\text{ if $p > 0$. }
                          \end{cases} 
\end{align*}
While in the case where $\d_{\eta_k} = 2$ with $\eta_k = \{ i,j\}$, 
we consider the subalgebra $\BL^-_{\eta_k}$ of $\BU_q^-$ isomorphic to 
$\BU_q(\Fs\Fl_3)^-$, and consider the canonical basis 
$\bB_{\eta_k} = \{ b(\Bd_k, \Bh_{\eta_k})\}$ as defined in 3.4, 
where $\Bd_k \in \BN^3 = \BN^{J_k}$.  
We define
\begin{align*}
\tag{6.7.2}
\wt f_{\eta_k}^{(\Bd_k)} = \begin{cases}
        \wt T_{\eta_p}\wt T_{\eta_{p-1}}\cdots \wt T_{\eta_{k+1}}
                (b(\Bd_k, \Bh_{\eta_k}))  &\quad\text{ if $p \le 0$, } \\
\wt T_{\eta_{p+1}}\iv \wt T_{\eta_{p+2}}\iv \cdots \wt T_{\eta_{k-1}}\iv
                (b(\Bd_k, \Bh_{\eta_k}))  &\quad\text{ if $p > 0$. }
                          \end{cases} 
\end{align*}

In the definition of $L(\Bc, \lv_p)$, we replace the root vectors 
belonging to $\vD^{(0)}_{>}$ and $\vD^{(0)}_{<}$ by (6.7.1) and (6.7.2),
and leave the root vectors belonging to $\vD^{(1)}_{>} \sqcup \vD_{<}^{(1)}$
unchanged.  
Thus one can define $L\nat(\Bd, \lv_p)$ as in 3.4. 
Set
\begin{equation*}
\tag{6.7.3}
\SX\nat_{\lv_p} = \{ L\nat(\Bd, \lv_p) \mid \Bd \in \SC\nat_p\},
\end{equation*} 
where $\SC\nat_p = \{ \Bd = (\Bd\nat_+, \Bd_0, \Bd\nat_-) \}$
is a parameter set for $\SX\nat_{\lv_p}$.  Here if we define $\Bd_{+}^{(0)}$
(resp.$\Bd_-^{(0)}$) 
as the part in $\Bd\nat_{+}$ (resp. $\Bd\nat_-$) corresponding to 
$\vD^{(0)}_{>,p}$ (resp. $\vD^{(0)}_{<,p}$), 
we have
\begin{equation*}
\Bd_{+}^{(0)} = (\Bd_k)_{k \le p} \text{ with } \Bd_k \in \BN^{J_k}, \qquad
\Bd_{-}^{(0)} = (\Bd_k)_{k > p} \text{ with } \Bd_k \in \BN^{J_k}, 
\end{equation*}  
where $J_k = J_{\eta_k}$.
Note that the parameter set $\SC\nat_p$ is naturally identified with $\SC_p$
under the correspondence $\Bd_k \lra (\cdots, c_{\b_j}, \cdots)_{j \in J_k}$
(see 2.1 for the case where $\d_{\eta_k} = 2$). 
$\SX\nat_{\lv_p}$ gives an $\BA$-basis of ${}_{\BA}\BU_q^-$, which we call
the modified PBW-basis of $\BU_q^-$. 
As in (6.6.4), we have the following.

\begin{equation*}
\tag{6.7.4}
\begin{aligned}
\wt T_{\eta_p}\iv L\nat(\Bc, \lv_p) &= L\nat(\Bc', \lv_{p-1}),  & 
        &\qquad\text{ if $p \le 0$ and $\Bc_p = 0$, }  \\
\wt T_{\eta_{p+1}} L\nat(\Bc, \lv_p) &= L\nat(\Bc', \lv_{p+1}),  & 
        &\qquad\text{ if $p > 0$ and $\Bc_p = 0$, }  \\
\end{aligned}
\end{equation*}

\para{6.8.}
Let $\ul\b_0 \lv \ul\b{-1} \lv \cdots \lv \d \lv \cdots \lv \ul\b_1$
be the convex order associated to $\ul\Bh$.  We know that 
$\ul\vD^{\re,+}_{>} = \{ \ul\b_k \mid k \le 0\}$ and 
$\ul\vD^{\re,+}_{<} = \{ \ul\b_k \mid k > 0\}$. Thus for each $\eta \in \ul I$, 
there exists $p \in \BZ$ such that $\eta = \eta_p$. 
We consider the modified PBW-basis $\SX\nat_{\lv_p}$ of $\BU_q^-$. 
In the case where $p \le 0$, we denote by $\BU_q^-[\eta]$ the $\BQ(q)$-subspace 
of $\BU_q^-$ spanned by $L\nat(\Bc, \lv_p)$ such that $\Bc_p = 0$, and set 
${}^*\BU_q^- = *(\BU_q^-[\eta])$, where $* : \BU_q^- \to \BU_q^-$ is 
the anti-algebra automorphism.  While in the case where $p > 0$, we define 
${}^*\BU_q^-[\eta]$ as the $\BQ(q)$-subspace of $\BU_q^-$ spanned by 
$L\nat(\Bc,\lv_p)$ such that $\Bc_p = 0$, and set $\BU_q^-[\eta] = *({}^*\BU_q^-[\eta])$. 
Note that in the case where $\d_{\eta} = 1$ with $\eta = \{ \a_{i_1}, \dots, \a_{i_k}\}$,
where $k = |\eta|$, we consider $\BL^-_{\eta}$ the subalgebra of $\BU_q^-$ generated by
$f_{i_1}, \dots, f_{i_k}$.  Then $\BL^-_{\eta}$ is of type $A_1 \times \cdots \times A_1$ 
($k$-times), and the set 
$\{ f_{i_1}^{(c_1)}\cdots f_{i_k}^{(c_k)}\mid \Bc = (c_1, \dots, c_k) \in \BN^k\}$ 
coincides with the canonical basis 
$\bB_{\eta} = \{ b(\Bc, \Bh_{\eta}) \mid \Bc \in \BN^{|\eta|}\}$.  
In the following discussion, we write $b(\Bc, \Bh_{\eta})$ simply as $b(\Bc_{\eta})$, 
where  
\begin{equation*}
\Bc_{\eta} = \begin{cases}
                 (c_1, \dots, c_{|\eta|}) \in \BN^{|\eta|}
                     &\quad\text{ if $\d_{\eta} = 1$}, \\
                 (c_1, c_2, c_3) \in \BN^3 
                     &\quad\text{ if $\d_{\eta} = 2$}.
             \end{cases}
\end{equation*}
Thus even for the case (6.7.1), we use the same notation as in (6.7.2), 
by using the canonical basis $\bB_{\eta}$. 
\par
In the case where $\eta = \eta_p$ with $p \le 0$, 
let ${}_{\BA}\BU_q[\eta]$ be the $\BA$-submodule of ${}_{\BA}\BU_q^-$ 
spanned by $L^{\sharp}(\Bc, \lv_p)$ such that $\Bc_p = 0$, and 
set ${}_{\BA}{}^*\BU_q^-[\eta] = *({}_{\BA}\BU_q^-[\eta])$. 
In the case where $\eta = \eta_p$ with $p > 0$, let ${}_{\BA}{}^*\BU_q^-[\eta]$
be the $\BA$-submodule of ${}_{\BA}\BU_q^-$ spanned by $L^{\sharp}(\Bc, \lv_p)$ 
such that $\Bc_p = 0$, 
and set ${}_{\BA}\BU_q^-[\eta] = *({}_{\BA}{}^*\BU_q^-[\eta])$.   

\begin{lem}   
\begin{enumerate}
\item 
 We have a direct sum decomposition 
\begin{equation*}
\tag{6.9.1}
\BU_q^- = \bigoplus_{b \in \bB_{\eta}}b\BU_q^-[\eta], \qquad
\BU_q^- = \bigoplus_{b \in \bB_{\eta}}{}^*\BU_q^-[\eta]b.
\end{equation*}
\item 
We have a direct sum decomposition of ${}_{\BA}\BU_q^-$ as $\BA$-modules, 
\begin{equation*}
\tag{6.9.2}
{}_{\BA}\BU_q^- = \bigoplus_{b \in \bB_{\eta}}b{}_{\BA}\BU_q^-[\eta], \qquad
{}_{\BA}\BU_q^- = \bigoplus_{b \in \bB_{\eta}}{}_{\BA}{}^*\BU_q^-[\eta]b.
\end{equation*}
\end{enumerate}
\end{lem} 

\begin{proof}
Assume that $\eta = \eta_p$ for $p \le 0$.  Then by using 
the PBW-basis $\SX\nat_{\lv_p}$, we obtain the first identity in (6.9.1).  
Note that $\s$ acts on $\BL^-_{\eta}$, and $\s$ gives a permutation 
on the canonical basis $\bB_{\eta}$ of $\BL^-_{\eta}$, which is the same as 
the $*$-operation on $\bB_{\eta}$.  Thus the second 
identity follows from the first by applying $*$ on both sides.  
In the case where $\eta = \eta_p$ with $p > 0$, by using $\SX\nat_{\lv_p}$, 
we obtain the second identity in (6.9.1).  By applying $*$, the first identity follows. 
(i) is proved.  (ii) follows from (i) since $\SX\nat_{\lv_p}$ is an $\BA$-basis of 
${}_{\BA}\BU_q^-$. 
\end{proof}

\begin{prop}   
Under the notation above, we have
\begin{align*}
\tag{6.10.1}
\BU_q^-[\eta] &= \{ x \in \BU_q^- \mid \wt T_{\eta}\iv(x) \in \BU_q^-\}, \\ 
\tag{6.10.2}
{}^*\BU_q^-[\eta] &= \{ x \in \BU_q^- \mid \wt T_{\eta}(x) \in \BU_q^-\}.
\end{align*}
\end{prop}

\begin{proof}
We consider the case where $\eta = \eta_p$ with $p \le 0$. 
By (6.7.4), we have 
\begin{equation*}
\tag{6.10.3}
\BU_q^-[\eta] \subset \{ x \in \BU_q^- \mid \wt T_{\eta}\iv(x) \in \BU_q^-\}.
\end{equation*}
We show the converse inclusion.  
Lemma 6.9 implies that the map $u\otimes v \mapsto uv$ gives an isomorphism 
as vector spaces $\BU_q^- \simeq \BL^-_{\eta}\otimes \BU_q^-[\eta]$. 
Hence the proof of (6.10.1) is reduced to the case where $\BU^-_q = \BL^-_{\eta}$.
It is enough to see that 
$\{ x \in \BL^-_{\eta} \mid \wt T_{\eta}\iv(x) \in \BL^-_{\eta}\} = \{ 0\}$. 
In general, we consider the braid group action $T_w : \BU_q \to \BU_q$. 
We have a weight space decomposition $\BU_q = \bigoplus_{\nu \in Q}(\BU_q)_{\nu}$, and 
we know that $T_w : (\BU_q)_{\nu} \to (\BU_q)_{w(\nu)}$.
Hence 
$\BL_{\eta}^- \cap \wt T_{\eta}(\BL_{\eta}^-) = \{ 0\}$
since $\wt T_{\eta}(\BL_{\eta}^-) = \BL_{\eta}^+$, where $\BL_{\eta}^+ = \BU_q^+$ 
is the positive part of $\BU_q$.  
This proves (6.10.1).  Then (6.10.2) is obtained from (6.10.1) by applying 
the *-operation, since $* \circ \wt T_{\eta}\, \circ * = \wt T\iv_{\eta}$. 
In the case where $\eta = \eta_p$ with $p > 0$, (6.10.2) is proved in a similar way.
Then (6.10.1) follows from (6.10.2).  The proposition is proved. 
\end{proof}

\para{6.11.}
Proposition 6.10 shows that $\BU_q^-[\eta]$ does not depend on the
choice of $\ul\Bh$. Note that $\wt T_{\eta} = \prod_{i \in \eta}T_i$ 
for $\eta \in \ul I$ with $\d_{\eta} = 1$, and 
$\wt T_{\eta} = T_iT_jT_i = T_jT_iT_j$ if $\eta = \{ i,j\}$ with $\d_{\eta} = 2$. 
Hence $\s \circ \wt T_{\eta} = \wt T_{\eta} \circ \s$.  Then Proposition 6.10 
implies that $\BU_q^-[\eta]$ is $\s$-stable. Note that this is not directly obtained 
from the definition since we don't know that $\SX\nat_{\lv_p}$ is stable under 
the action of $\s$.  
\par
In the case where $\d_{\eta} = 1$, the set $\bB_{\eta}^{\s}$ is given as 
$\bB_{\eta}^{\s} = \{ \wt f_{\eta}^{(a)} = \prod_{i \in \eta}f_i^{(a)} \mid a \in \BN \}$.
While in the case where $\d_{\eta} = 2$ with $\eta = \{ i,j\}$, 
\begin{equation*}
\bB_{\eta}^{\s} = \{ \wt f_{\eta}^{(a)} = f_i^{(a)}f_j^{(2a)}f_i^{(a)}
                                       = f_j^{(a)}f_i^{(2a)}f_j^{(a)} \mid a \in \BN \}.
\end{equation*}  
The following result is obtained immediately from Lemma 6.9.
Note that the uniqueness property follows from the property of 
PBW-basis $\SX\nat_{\lv_p}$. 

\begin{prop}  
Concerning the first formula in (6.9.1), the following holds.
\begin{enumerate}
\item  We have a direct sum decomposition 
\begin{equation*}
\BU_q^{-,\s} = \bigoplus_{a \in \BN}\wt f_{\eta}^{(a)}(\BU_q^-[\eta])^{\s}
                 \oplus K_{\eta},
\end{equation*}
where $K_{\eta}$ is a sum of $bx + \s(bx)$ for $b \in \bB_{\eta}$ such that $\s(b) \ne b$, 
and $x \in \BU_q^-[\eta]$. 
$x \in \BU_q^{-,\s}$ can be written uniquely as
$x = \sum_{a \in \BN}\wt f_{\eta}^{(a)}x_a + y$
with $x_a \in (\BU_q^-[\eta])^{\s}$ and $y \in K_{\eta}$.  
\item We also have a decomposition as $\BA$-modules
\begin{equation*}
{}_{\BA}\BU_q^{-,\s} = 
   \bigoplus_{a \in \BN}\wt f_{\eta}^{(a)}({}_{\BA}\BU_q^-[\eta])^{\s} \oplus {}_{\BA}K_{\eta},
\end{equation*}
where ${}_{\BA}K_{\eta}$ is the $\BA$-submodule of $K_{\eta}$ generated by $bx + \s(bx)$ 
as above with $x \in {}_{\BA}\BU_q^-$. 
\end{enumerate}
Similar results hold also for the second formula in (6.9.1). 
\end{prop}

\remark{6.13.}
The above results correspond to the results Lemma 3.6 $\sim$ Lemma 3.11 in [SZ2].
But the discussion here becomes much simpler than those used in [SZ2] 
once we use the PBW-basis $\SX\nat_{\lv_p}$. 

\begin{lem}  
For each $\eta \in \ul I$, we have
\begin{align*}
\tag{6.14.1}
\BU_q^-[\eta] &\subset \begin{cases}
                         \bigcap_{i \in \eta} \BU_q^-[i]  
                             &\quad\text{ if $\d_{\eta} = 1$, } \\
                         \BU_q^-[i] \cap \BU_q^-[j]
                             &\quad\text{ if $\d_{\eta} = 2$ with $\eta = \{i, j\}$.} 
                      \end{cases}   \\ \\
\tag{6.14.2}
{}^*\BU_q^-[\eta] &\subset \begin{cases}
                         \bigcap_{i \in \eta} {}^*\BU_q^-[i]  
                             &\quad\text{ if $\d_{\eta} = 1$, } \\
                         {}^*\BU_q^-[i] \cap {}^*\BU_q^-[j]
                             &\quad\text{ if $\d_{\eta} = 2$ with $\eta = \{i, j\}$.} 
                      \end{cases}   \\
\end{align*}
\end{lem}

\begin{proof}
Assume that $\eta = \eta_p$ with $p \le 0$.
We consider the case where $\d_{\eta} = 2$ with $\eta = \{ i,j\}$, 
Then $\BU_q^-[\eta]$ is spanned by 
$L\nat(\Bc, \lv_p)$ with $\Bc_p = \bold 0$.  
This basis is also expressed as a sum of the ordinary PBW-bases $L(\Bc', \lv_p)$. 
Since $\eta_p$ corresponds to $(i, j,i) \in \Bh'$, by the property of 
$\SX_{\lv_p}$ similar to (6.7.4), we see that $T_i\iv(x) \in \BU_q^-$ for 
$x \in \BU_q^-[\eta]$.  By the discussion using $(j,i,j)$, we also have
$T_j\iv(x) \in \BU_q^-$.  Hence $x \in \BU_q^-[i] \cap \BU_q^-[j]$. 
The proof for the case $\d_{\eta} = 1$ is similar.  Thus (i) holds.  Then (ii) 
follows from (i) by applying *-operation. 
\par
Next assume that $\eta = \eta_p$ with $p > 0$.  In this case, by a similar 
argument as above, one can prove (ii).  Then (i) follows from (ii) by applying *.
The lemma is proved.      
\end{proof}

\begin{prop}  
For each $\eta \in \ul I$, we have a direct sum decomposition 
\begin{equation*}
\tag{6.15.1}
\BU_q^{-,\s} = (\BU_q^-[\eta])^{\s} \oplus \wt f_{\eta}\BU_q^{-,\s} \oplus K_{\eta}.  
\end{equation*}
The first and the second terms are orthogonal each other 
with respect to the inner product $(\ ,\ )$.
\end{prop}

\begin{proof}
We show the first statement.  This is clear if $\d_{\eta} = 1$. 
We assume that $\d_{\eta} = 2$.  By the proof of Proposition 2.4, 
$\wt f_{\eta}\wt f_{\eta}^{(n)} \equiv \wt f_{\eta}^{(n+1)} \mod K_{\eta}$, up to scalar.
Hence,  $\bigoplus_{n > 0}\wt f^{(n)}_{\eta}\BU_q^-[\eta] \subset \wt f_{\eta}\BU_q^-$
 modulo $K_{\eta}$. Then (6.15.1) follows from  Proposition 6.12.
We show the second statement.
Assume that $\eta = \{i,j\}$ with $\d_{\eta} = 2$.  Then 
$\wt f_{\eta} = f_if_jf_i$, and 
$\wt f_{\eta}\BU_q^{-,\s} \subset f_i\BU_q^-$.  
By Lemma 6.14, $\BU_q^-[\eta] \subset \BU_q^-[i]$.  
We know that $(\BU_q^-[i], f_i\BU_q^-) = 0$, and this implies 
that $(\BU_q^-[i], \wt f_{\eta}\BU_q^-) = 0$.  The lemma holds.  
\end{proof}

\para{6.16.}
We define a partial order $\Bc < \Bc'$ in $\SC_p$ by the condition 
as follows; for $\Bc = (\Bc_+, \Bc_0, \Bc_-)$, 
$\Bc' = (\Bc'_+, \Bc'_0, \Bc'_-)$, $\Bc < \Bc'$ if and only 
if $\Bc_+ \le \Bc'_+$ and $\Bc_- \le \Bc'_-$, and at least one 
of the inequalities is strict, with respect to the lexicographic 
order $\le$ on $\Bc_+$ (resp. on $\Bc_-$) from the left (resp. the right). 
Then for each $L(\Bc, \lv_p) \in \SX_{\lv_p}$, there  exists 
a unique $b(\Bc, \lv_p) \in \wt\bB$ satisfying the property that 
$b(\Bc, \lv_p)$ is ${}^-$-invariant, and that 
\begin{equation*}
\tag{6.16.1}
b(\Bc, \lv_p) = L(\Bc, \lv_p) + \sum_{\Bd > \Bc}a_{\Bd}L(\Bd, \lv_p),
\end{equation*}  
where $a_{\Bd} \in q\BZ[q]$. Then $\pm b(\Bc, \lv_p) \in \bB$. A similar 
construction also works for $L\nat(\Bc, \lv_p)$, and one can define the 
canonical (signed) basis $b\nat(\Bc, \lv_p)$ by the property that  
$b\nat(\Bc, \lv_p)$ is ${}^-$-invariant and that 
\begin{equation*}
\tag{6.16.2}
b\nat(\Bc, \lv_p) = L\nat(\Bc, \lv_p) + \sum_{\Bd > \Bc}a_{\Bd}L\nat(\Bd, \lv_p),
\end{equation*}
where $a_{\Bd} \in q\BZ[q]$. 
\par
The $\BZ[q]$-submodule $\SL_{\BZ}(\infty)$ of $\BU_q^-$ is defined as in 
[MSZ, 1.10]. The set $\{ b\nat(\Bc, \lv_p) \mid \Bc \in \SC_p\}$
gives a $\BZ[q]$-basis of $\SL_{\BZ}(\infty)$, and   
the modified PBW-bases also
$\SX\nat_{\lv_p}$ gives a $\BZ[q]$-basis of $\SL_{\BZ}(\infty)$. 
\par
Note that the set of $\s$-orbit sums  of $\bB$ gives a basis of 
$\BU_q^{-,\s}$, and an $\BA$-basis of ${}_{\BA}\BU_q^{-,\s}$. 
The set $\wt\bB^{\s}$ is contained in 
$\bigoplus_{n \ge 0}\wt f_{\eta}^{(n)}(\BU_q^-[\eta])^{\s}$ or 
in $\bigoplus_{n \ge 0}(\BU_q^-[\eta])^{\s}\wt f_{\eta}^{(n)}$
by Proposition 6.12.  
For $b \in \wt\bB^{\s}$, we define $\ve_{\eta}(b)$ as the largest integer $a$ 
such that $b \in \wt f_{\eta}^{(a)}\BU_q^{-,\s}$, and $\ve^*_{\eta}(b)$ as the
largest integer $a$ such that $b \in \BU_q^{-,\s}\wt f_{\eta}^{(a)}$. 

\begin{prop}  
\begin{enumerate}
\item
Let $b \in \wt\bB^{\s}$ be such that $\ve_{\eta}(b) = 0$. There exists a unique
$b' \in \wt\bB^{\s}$ such that 
\begin{equation*}
\tag{6.17.1}
\wt f_{\eta}^{(n)}b \equiv b' \mod \wt f_{\eta}^{(n+1)}\BU_q^{-,\s} \oplus K_{\eta}. 
\end{equation*}
Then $\ve_{\eta}(b') = n$.
The correspondence $b \mapsto b'$ gives a bijection 
\begin{equation*}
\tag{6.17.2}
\pi_{\eta, n} : \{ b \in \wt\bB^{\s} \mid \ve_{\eta}(b) = 0\} \isom 
            \{ b' \in \wt\bB^{\s} \mid \ve_{\eta}(b') = n\}. 
\end{equation*}
\item  
Let $b \in \wt \bB^{\s}$ be such that $\ve^*_{\eta}(b) = 0$. There exists a unique
$b' \in \wt\bB^{\s}$ such that 
\begin{equation*}
\tag{6.17.3}
b\wt f_{\eta}^{(n)} \equiv b' \mod \BU_q^{-,\s}\wt f_{\eta}^{(n+1)} \oplus K_{\eta}. 
\end{equation*}
Then $\ve^*_{\eta}(b') = n$. 
The correspondence $b \mapsto b'$ gives a bijection 
\begin{equation*}
\tag{6.17.4}
\pi^*_{\eta, n} : \{ b \in \wt\bB^{\s} \mid \ve^*_{\eta}(b) = 0\} \isom 
            \{ b' \in \wt\bB^{\s} \mid \ve^*_{\eta}(b') = n\}. 
\end{equation*}
\end{enumerate}
\end{prop}

\begin{proof}
We assume that 
$\eta = \eta_p$ with $p \le 0$, and consider $\SX\nat_{\lv_p}$.
We also assume that $\eta = \{ i,j\}$ with $\d_{\eta} = 2$. 
We prove (i). 
Take $b \in \wt\bB^{\s}$ such that $\ve_{\eta}(b) = 0$.  
Then there exists $b\nat(\Bc, \lv_p)$ such that $b = \pm b\nat(\Bc, \lv_p)$. 
We write 
\begin{equation*}
b\nat(\Bc, \lv_p) = L\nat(\Bc, \lv_p) + \sum_{\Bd > \Bc}a_{\Bd}L\nat(\Bd, \lv_p)
\end{equation*}
with $a_{\Bd} \in q\BZ[q]$. 
Since $b \in \wt\bB^{\s}$, $\Bc = (\Bc_+, \Bc_0, \Bc_-)$
with $\Bc_p = \bold 0 = (0,0,0)$, and $\Bd_p = (d_1, d_1, d_1)$ for some $d_1 \ge 0$, 
namely, $L\nat(\Bd, \lv_p) = \wt f^{(d_1)}_{\eta}L\nat(\Bd', \lv_p)$ 
with $\Bd'_p = \bold 0$. 
Since $\wt f^{(n)}_{\eta}b\nat(\Bc, \lv_p) \in \BU_q^{-,\s}$, by Proposition 6.12 
it is written as a linear combination of $\wt f^{(a)}_{\eta}L\nat(\Bc', \lv_p)$ 
modulo $K_{\eta}$
for various $a$ and $\Bc'$, where $\Bc'_p = \bold 0$. 
It follows from the computation for $\BU_q(\Fs\Fl_3)$ (see the proof of Proposition 2.4)
that $\wt f_{\eta}^{(n)}\wt f_{\eta}^{(d_1)}$ is a linear combination of canonical 
basis in $\bB_{\eta}$, among them, the $\s$-invariant one is only $\wt f_{\eta}^{(d_1 + n)}$. 
Therefore, we have  
\begin{align*}
\wt f^{(n)}_{\eta}b\nat(\Bc, \lv_p) &= L\nat(\Bc', \lv_p) + 
                     \sum_{\Bd > \Bc}a_{\Bd}\begin{bmatrix}
                                              d_1 + n \\
                                                 n
                                            \end{bmatrix}
           \wt f^{(d_1 + n)}_{\eta}L\nat(\Bd', \lv_p) + z \\
    &= L\nat(\Bc', \lv_p) + \sum_{\Bd > \Bc}a_{\Bd}L\nat(\Bd'', \lv_p) + y + z,
\end{align*}
with
\begin{equation*}
 y = \sum_{d_1 > 0}a_{\Bd}\biggl(\begin{bmatrix}
                               d_1 + n \\
                                   n
                                 \end{bmatrix} - 1\biggr)L\nat(\Bd'', \lv_p),
\qquad z \in {}_{\BA}K_{\eta},
\end{equation*} 
where $L\nat(\Bc', \lv_p) = \wt f^{(n)}_{\eta}L\nat(\Bc, \lv_p)$, 
and $L\nat(\Bd'', \lv_p) = \wt f^{(d_1 + n)}_{\eta}L\nat(\Bd', \lv_p)$ 
are elements in $\SX\nat_{\lv_p}$.
Write $y + z = \sum_{\Be}c_{\Be}b\nat(\Be, \lv_p)$ with $c_{\Be} \in \BA$. 
Since $d_1 > 0$, we have $\Be_p > \Bn = (n,n,n)$ if $\Be_p$ is $\s$-stable.
 We write 
$c_{\Be}$ as $c_{\Be} = c_{\Be}^+ + c_{\Be}^0 + c_{\Be}^-$, where 
$c_{\Be}^+ \in q\BZ[q], c_{\Be}^- \in q\iv\BZ[q\iv]$ and $c_{\Be}^0 \in \BZ$.  
Then we have
\begin{align*}
\tag{6.17.5}
\wt f^{(n)}_{\eta}&b\nat(\Bc, \lv_p) - 
    \sum_{\Be}(c_{\Be}^0 + c_{\Be}^- +\ol{c_{\Be}^-})b\nat(\Be, \lv_p)  \\
     &= L\nat(\Bc', \lv_p) + \sum_{\Bd > \Bc}a_{\Bd}L\nat(\Bd'', \lv_p)
                            + \sum_{\Be}
                     (c_{\Be}^+ - \ol{c_{\Be}^-})b\nat(\Bd'', \lv_p). 
\end{align*}
The left hand side of (6.17.5) is ${}^-$-invariant, and the right hand side 
belongs to $\SL_{\BZ}(\infty)$, and is equal to 
$L\nat(\Bc', \lv_p) \mod q\SL_{\BZ}(\infty)$.
Hence by (6.16.2), this coincides with $b\nat(\Bc',\lv_p)$. 
It follows that 
\begin{align*}
\tag{6.17.6}
\wt f^{(n)}_{\eta}b\nat(\Bc, \lv_p) &= b\nat(\Bc', \lv_p) + 
       \sum_{\Be}(c_{\Be}^0 + c_{\Be}^- +\ol{c_{\Be}^-})b\nat(\Be, \lv_p) \\
     &\equiv b\nat(\Bc', \lv_p) + 
       \sum_{\substack{\s(\Be_p) = \Be_p \\\Be_p > \Bn}}
      (c_{\Be}^0 + c_{\Be}^- +\ol{c_{\Be}^-})b\nat(\Be, \lv_p) 
        \mod K_{\eta}, 
\end{align*} 
where $b\nat(\Be, \lv_p)$ in the last sum corresponds to 
$L\nat(\Be, \lv_p) \in \wt f_{\eta}^{(n')}\BU_q^-[\eta]$ for some $n' > n$.
Hence $b\nat(\Be, \lv_p) \in \wt f^{(n+1)}_{\eta}\BU_q^-$, and  
the sum in (6.17.6) is contained in $\wt f^{(n+1)}_{\eta}\BU_q^-$. 
Note that since $\wt f^{(n)}_{\eta}(b\nat(\Bc, \lv_p))$ is $\s$-stable, 
$b\nat(\Bc',\lv_p)$ and $b\nat(\Be, \lv_p)$ in the last sum are $\s$-stable
(or its $\s$-orbit is contained in $K_{\eta}$).  
Thus $b' = b\nat(\Bc',\lv_p)$ satisfies the condition in (6.17.1). 
It is clear that $\ve_{\eta}(b') = n$. 
\par 
Note that $b' = b\nat(\Bc', \lv_p)$ is determined uniquely from 
$L\nat(\Bc, \lv_p)$, which is the canonical (signed) basis corresponding 
to the PBW-basis $\wt f^{(n)}_{\eta}L\nat(\Bc, \lv_p)$.   
\par
Conversely, assume that $b' \in \wt\bB^{\s}$ is such that $\ve_{\eta}(b') = n$.
Take $b\nat(\Bc', \lv_p)$ such that $b\nat(\Bc',\lv_p) = \pm b$. 
Then $L\nat(\Bc', \lv_p) = \wt f^{(n)}_{\eta}L\nat(\Bc'', \lv_p)$ 
with $\Bc''_p = \bold 0$.  Let $x$ 
be the projection of $b\nat(\Bc', \lv_p)$ onto $\wt f^{(n)}_{\eta}\BU_q^-[\eta]$.
Then $x \equiv L\nat(\Bc', \lv_p) \mod q\SL_{\BZ}(\infty)$, and $x$ is $\s$-stable. 
Write $x = \wt f^{(n)}_{\eta}x_0$ with $x_0 \in (\BU_q^-[\eta])^{\s}$. 
Then $x_0 \equiv L\nat(\Bc'', \lv_p) \mod q\SL_{\BZ}(\infty)$. 
Consider $b = b\nat(\Bc'', \lv_p)$.  Then 
$b$ is determined uniquely from $x_0$ by a similar condition as 
in (6.16.2), hence $b$ is $\s$-stable.  Clearly $\ve_{\eta}(b) = 0$. 
It follows from the previous construction that $\pi_{\eta,n}(b) = b'$. 
The map $b' \mapsto b$ gives the inverse map of $\pi_{\eta,n}$.  
Hence $\pi_{\eta,n}$
is a bijection. This proves (i).  The case where $\d_{\eta} = 1$ is 
proved similarly, and is simpler.  Now (ii) follows from (i) by applying 
the *-operation. 
\par
Next assume that $\eta = \eta_p$ with $p > 0$.  In this case, by a similar 
argument as above, we can prove (ii).  Then (i) is obtained from this by 
applying *-operation.  The proposition is proved. 
\end{proof}

The $\BA$-version of Proposition 6.17 is given as follows;

\begin{cor}  
Let $b \in \wt\bB^{\s}$ be such that $\ve_{\eta}(b) = 0$.
Then there exists a unique $b' \in \wt\bB^{\s}$ such that
\begin{equation*}
\tag{6.18.1}
\wt f^{(n)}_{\eta}b \equiv b' \mod 
      \sum_{n' > n}\wt f^{(n')}_{\eta}{}_{\BA}\BU_q^{-,\s} \oplus {}_{\BA}K_{\eta}.
\end{equation*}
\end{cor}

\para{6.19.}
For each $\eta \in \ul I$, by making use of the bijection (6.17.2), 
we define a map $\wt F_{\eta}$ as a composite of bijections
$\pi_{\eta, n}\iv$ and $\pi_{\eta, n+1}$

\begin{equation*}
\wt F_{\eta} :\{ b \in \wt\bB^{\s} \mid \ve_{\eta}(b) = n\} 
   \isom  \{ b' \in \wt\bB^{\s} \mid \ve_{\eta}(b') = 0\} 
   \isom  \{ b'' \in \wt\bB^{\s} \mid \ve_{\eta}(b'') = n+1 \}  
\end{equation*}

In the case where $\ve_{\eta}(b) > 0$, we define $\wt E_{\eta}$ as the inverse
map of $\wt F_{\eta}$, and set $\wt E_{\eta}(b) = 0$ if $\ve_{\eta}(b) = 0$. 
The maps $\wt F_{\eta}, \wt E_{\eta} : \wt\bB^{\s} \to \wt\bB^{\s} \cup \{ 0\}$
are called Kashiwara operators. 

\begin{prop}  
For each $b \in \wt\bB^{\s}$, there exists a sequence $\eta_1, \dots, \eta_N \in \ul I$
such that $b = \pm \wt F_{\eta_1}\cdots \wt F_{\eta_N}1$. 
\end{prop}

\begin{proof}
Take $b \in \wt\bB^{\s}$, and assume that $b \ne \pm 1$. 
Then there exists $\eta \in \ul I$  such that $\ve_{\eta}(b) \ne 0$.  
In fact assume that $\ve_{\eta}(b) = 0$ for any $\eta$.  By Lemma 6.14, 
if $\ve_{\eta}(b) = 0$, then $\ve_i(b) = 0$ for any $i \in \eta$. 
Thus $\ve_i(b) = 0$ for any $i \in I$. By a well-known result (see [L1, Lemma 1.2.15]), 
this implies that $b = \pm 1$. 
\par
Now take $\eta$ such that $\ve_{\eta}(b) \ne 0$. Then 
$b' = \wt E_{\eta}(b) \in \wt\bB^{\s}$ 
satisfies the condition that $\ve_{\eta}(b') < \ve_{\eta}(b)$. 
Then the proposition follows by induction on the weight of $b$. 
\end{proof}

We are now in a position to prove the surjectivity of $\Phi$ in the affine case. 
The proof is almost parallel 
to the one in [SZ2, Proposition 2.10].
\begin{prop}  
Assume that $(X, \ul X)$ satisfies the condition in Theorem 2.6.
Then the map $\Phi : {}_{\BA'} \ul\BU_q^-\to \BV_q$ is surjective, 
hence Theorem 2.6 holds. 
\end{prop}

\begin{proof}
The image of $\wt\bB^{\s}$ gives a signed basis of $\BV_q$.
Thus it is enough to show, for each $b \in \wt\bB^{\s}$, that
\begin{equation*}
\tag{6.21.1}
\pi(b) \in \Im \Phi.
\end{equation*}
Take $b \in \wt\bB^{\s}$. Let $\nu = \nu(b) \in Q_-$ be the weight of $b$. 
We prove (6.21.1) by induction on $|\nu|$. If $|\nu| = 0$, it certainly holds. 
So assume that $|\nu| \ge 1$, and assume that (6.21.1) holds for $b'$ 
such that $|\nu(b')| < |\nu|$.  
By Proposition 6.20, there exists $\eta \in \ul I$
such that $b = \wt F^n_{\eta}(b')$ for some $b' \in \wt\bB^{\s}$ with 
$\ve_{\eta}(b') = 0$.  Thus $b'$ satisfies (6.21.1) by induction 
hypothesis.  We may further assume that
(6.21.1) holds for $b'' \in \wt\bB^{\s}$ such that $\nu(b'') = \nu$ and that
$\ve_{\eta}(b'') > n$ (note that since $\dim (\BU_q^-)_{\nu} < \infty$, 
if $n' >> 0$, there does not exist $b''$ such that $\nu(b'') = \nu$ and that 
$\ve_{\eta}(b'') = n'$).  
By Corollary 6.18, $b$ can be written as 
$b = \wt f^{(n)}_{\eta}b' + y + z$, where $z \in {}_{\BA}K_{\eta}$, and 
$y$ is an $\BA$-linear combination of $b'' \in \wt\bB^{\s}$ such that 
$\nu(b'') = \nu$ and that $\ve_{\eta}(b'') > n$. 
By induction, $\pi(b') \in \Im \Phi$.  Since $\Phi$ is a homomorphism by 
Proposition 2.8, we have
$\pi(\wt f^{(n)}_{\eta} b') = \Phi(\ul f^{(n)}_{\eta})\pi(b') \in \Im \Phi$. 
Since ${}_{\BA}K_{\eta} \subset J$, we have $\pi(z) = 0$.  By induction hypothesis, 
$\pi(y) \in \Im \Phi$. This implies that $\pi(b) \in \Im \Phi$.  
Hence (6.21.1) holds for $\nu$.  The proposition is proved. 
\end{proof}

\para{6.22.}
We consider the isomorphism $\Phi : {}_{\BA'}\ul\BU_q^- \isom \BV_q$ 
as in Theorem 2.6. Recall the $\BZ$-module 
$(Q^{\s})' = \bigoplus_{\eta \in \ul I}\BZ\wt\a_{\eta}$, and we identify 
this with $\ul Q = \bigoplus_{\eta \in \ul I}\BZ\a_{\eta}$, as in 5.2.  
We have weight space decompositions 
$\BU_q^- = \bigoplus_{\nu \in Q_-}(\BU_q^-)_{\nu}$ and 
$\ul\BU_q^- = \bigoplus_{\nu \in \ul Q_-}(\ul\BU_q^-)_{\nu}$.  
Under the identification $(\ul Q)_- \simeq (Q_-^{\s})' \subset Q_-$, 
Theorem 2.6 implies that 
we have a weight space decomposition 
$\BV_q = \bigoplus_{\nu \in (Q^{\s}_-)'}(\BV_q)_{\nu}$, and 
$\Phi$ induces an isomorphism ${}_{\BA'}(\ul\BU_q)_{\nu} \isom (\BV_q)_{\nu}$
for each $\nu \in \ul Q_-$. 
Note that $\{ \pi(b) \mid b \in \wt\bB^{\s}\}$ gives a signed basis of $\BV_q$.  Hence 
$\{ \pi(b) \mid b \in \wt\bB^{\s}(\nu)\}$ gives a signed basis of $(\BV_q)_{\nu}$, 
where $\wt\bB^{\s}(\nu) = \{ b \in \wt\bB^{\s} \mid \nu(b) = \nu\}$.  
In particular, we have
\begin{equation*}
\tag{6.22.1}
|\wt\bB^{\s}(\nu)|/2 = \dim (\ul\BU_q^-)_{\nu} \quad \text{  for  } \quad 
\nu \in (Q^{\s}_-)' \simeq \ul Q_-. 
\end{equation*}

\para{6.23.}
We consider the modified PBW-basis 
$\SX\nat_{\lv} = \{ L\nat(\Bc, \lv) \mid \Bc \in \SC\}$ 
for the case where $p = 0$.  
Note that it is not verified that $\s$ leaves the set $\SX\nat_{\lv}$ invariant.
By 6.4, $S^{\lv}_{\Bc_0} = T_w S_{\Bc_0}$, and $w$ is the minimal lift of 
$\ul w = \prod_{i \in \eta_0 - \{0\}}s_i$. Recall that $\Bc_0 = (\r^{(i)})_{i \in I_0}$
is an $I_0$-tuple of partitions. Let $\SC_0$ be the set of $I_0'$-tuple of partitions, 
and we regard $\Bc_0 \in \SC_0$ as an $I_0$-tuple of partition such that 
$\r^{(i)} = \emptyset$ for $i \notin I_0'$.  
It follows that $S_{\Bc_0}^{\lv} = S_{\Bc_0}$ for $\Bc_0 \in \SC_0$. 
$\s$ acts naturally on $\SC_0$ by $\s(\r^{(i)}) = \r^{\s(i)}$, and 
we have $\s(S_{\Bc_0}) = S_{\s(\Bc_0)}$ for $\Bc_0 \in \SC_0$.  We denote by 
$\SC^{\s}_0$ the set of $\s$-fixed elements in $\SC_0$.   
Let $\Bc_+ = (c_{\b})_{\b \lv \d}, \Bc_- = (c_{\g})_{\g \gv \d}$ be as in 6.4.
Let $\SC_+$ be the set of $\Bc_+$ such that $c_{\b} = 0$ for $\b \in \vD^{(1)}_{>}$, 
and $\SC_-$ the set of $\Bc_-$ such that $c_{\g} = 0$ for $\g \in \vD^{(1)}_{<}$. 
We denote by $\SC_{\pm}^{\s}$ the set of $\s$-fixed elements in $\SC_{\pm}$. 
Let $\SC^{\s}$ be the set of triples $\Bc = (\Bc_+, \Bc_0,\Bc_-)$ such that 
$\Bc_{\pm} \in \SC_{\pm}^{\s}, \Bc_0 \in \SC_0^{\s}$.  The set of triples
$\ul\SC = \{ \ul \Bc = (\ul\Bc_;, \ul \Bc_0, \ul\Bc_+)$, and the set of 
PBW-basis $\ul\SX_{\ul\Bh} = \{ L(\ul\Bc, 0) \mid \ul\Bc \in \ul\SC\}$ 
are defined as in [BN].  As in [SZ2], one can check that we have a natural 
bijection $\SC^{\s} \simeq \ul\SC$. 
Let $(\SX\nat_{\lv})^{\s}$ be the set of $\s$-fixed PBW-bases 
contained in $\SX\nat_{\lv}$.  The following result was proved 
in [SZ2, Theorem 6.13] in the case where $\s$ is admissible with $\ve \ne 4$.  
We can prove it in the general case by using (6.22.1) in a similar way. 

\begin{thm}  
Assume that $\s$ is admissible with $\ve = 4$,  or not admissible satisfying the condition 
in Theorem 2.6.  Then  
\begin{enumerate}
\item  \ $(\SX\nat_{\lv})^{\s} = \{ L\nat(\Bc, \lv) \mid \Bc \in \SC^{\s}\}$. 
\item  $\wt\bB^{\s} = \pm\{ b\nat(\Bc, \lv) \mid \Bc \in \SC^{\s} \}$.
\end{enumerate}
\end{thm} 

\para{6.25.}
We denote by $\ul\Bc \mapsto \Bc$ the bijection $\ul\SC \isom \SC^{\s}$. 
For each $\Bc \in \SC^{\s}$, $L\nat(\Bc, \lv)$ gives an element in 
${}_{\BA'}\BU_q^{-,\s}$, hence $\pi(L\nat(\Bc, \lv)) \in \BV_q$. On the other
hand, for each $\ul\Bc \in \ul\SC$, one can consider $L(\ul\Bc, 0) \in \ul\SX_{\ul\Bh}$
and its image $\Phi(L(\ul\Bc, 0)) \in \BV_q$.  We have the following result.

\begin{thm}  
Let $\Phi :{}_{\BA'}\ul\BU_q^- \isom \BV_q$ be the isomorphism given in 
Theorem 1.4 or Theorem 2.6.  In particular $\ul X$ is not of type 
$A_2^{(2)}$ nor $A_1^{(1)}$.  Then 
\begin{enumerate}
\item \ $\Phi(L(\ul\Bc, 0)) = \pi(L\nat(\Bc, \lv))$ for each $\ul\Bc \in \ul\SC$. 
\item \ $\Phi(b(\ul\Bc, 0)) = \pi(b\nat(\Bc, \lv))$ for each $\ul\Bc \in \ul\SC$. 
\item \ The natural map 
\begin{equation*}
\xymatrix{
   {}_{\BA'}\BU_q^{-,\s} \ar[r]^{\ \ \pi}  & \BV_q  \ar[r]^{\Phi\iv} & {}_{\BA'}\ul\BU_q^-
}
\end{equation*}
induces a bijection $\wt\bB^{\s} \isom \wt{\ul\bB}$, 
which is explicitly given by $b\nat(\Bc, \lv) \mapsto b(\ul\Bc, 0)$. 
\end{enumerate} 
\end{thm}

\begin{proof}
Recall that $\ul\Bh = (\cdots, \eta_{-1}, \eta_0, \eta_1, \dots)$ is the doubly infinite 
sequence given in 6.1. In order to prove (i), it is enough to see that 

\begin{equation*}
\tag{6.26.1}
\begin{aligned}
\Phi(\ul T_{\eta_0}\ul T_{\eta_{-1}}\cdots \ul T_{\eta_{k+1}}(\ul f_{\eta_k}))
          &= \pi(\wt T_{\eta_0}\wt T_{\eta_{-1}}\cdots \wt T_{\eta_{k+1}}(\wt f_{\eta_k})), 
          &     &\quad\text{ if $k \le 0$, }  \\
\Phi(\ul T_{\eta_1}\iv \ul T_{\eta_{2}}\iv \cdots \ul T_{\eta_{k-1}}\iv(\ul f_{\eta_k}))
          &= \pi(\wt T_{\eta_1}\iv\wt T_{\eta_2}\iv\cdots 
                             \wt T_{\eta_{k-1}}\iv(\wt f_{\eta_k})), 
          &     &\quad\text{ if $k \ge 1$. }.
\end{aligned}
\end{equation*}
This is reduced to the case where the rank of $\ul X$ is 2. As we have
excluded the case where $\ul X$ is of type $A_2^{(2)}$ or $A_1^{(1)}$, $\ul X$ 
is of finite type of rank 2. We have verified (6.26.1) in [SZ1] 
in the case where $\s$ is admissible with $\ve \ne 4$. Thus the remaining one is the case 
where $(X, \ul X)$ is of type $(A_4, C_2)$, which was already verified in 
the proof of Theorem 3.7. 
Hence (6.26.1) holds, and (i) follows.  (ii) follows from (i), and 
(iii) follows from (ii).  The theorem is proved.    
\end{proof}

\par

\par\vspace{1.5cm}
\noindent
Y. Ma \\
School of Mathematical Sciences, Tongji University \\ 
1239 Siping Road, Shanghai 200092, P.R. China  \\
E-mail: \verb|1631861@tongji.edu.cn|

\par\vspace{1.5cm}
\noindent
T. Shoji \\
School of Mathematica Sciences, Tongji University \\ 
1239 Siping Road, Shanghai 200092, P.R. China  \\
E-mail: \verb|shoji@tongji.edu.cn|

\par\vspace{0.5cm}
\noindent
Z. Zhou \\
School of Mathematical Sciences, Tongji University \\ 
1239 Siping Road, Shanghai 200092, P.R. China  \\
E-mail: \verb|forza2p2h0u@163.com|

\end{document}